\documentclass[12pt]{article}
\usepackage{authblk}
\usepackage[numbers]{natbib}
\usepackage{graphicx}
\usepackage{algorithm,algorithmic}
\usepackage{tikz}
\usepackage{rotating}
\usepackage{color}
\usepackage{amsmath, amssymb, amsthm}
\usepackage{enumerate}
\usepackage{geometry}
\usepackage{indentfirst}
\usepackage{epstopdf}
\usepackage{multirow}
\usepackage{bm}
\usepackage{hyperref}
\graphicspath{ {./Figures/}}
\hypersetup{hypertex=true,
colorlinks=true,
linkcolor=blue,
anchorcolor=blue,
citecolor=blue
}

\newcommand{\romA}{\uppercase\expandafter{\romannumeral1} }
\newcommand{\romB}{\uppercase\expandafter{\romannumeral2} }
\newcommand{\romC}{\uppercase\expandafter{\romannumeral3} }

\providecommand{\keywords}[1]
{
\textbf{\text{Keywords: }}#1
}
\makeatletter

\newcommand\backmatter{%
  \if@openright
    \cleardoublepage
  \else
    \clearpage
  \fi
   }

\makeatother

\newtheorem{proposition}{Proposition}
\newtheorem{remark}{Remark}
\begin{document}
\title{High-order WENO finite-difference methods for hyperbolic nonconservative systems of Partial Differential Equations}

\author[1]{
Baifen Ren\footnote{School of Mathematical Sciences, Ocean University of China, Qingdao 266100, China. E-Mail: {\tt renbaifen@stu.ouc.edu.cn}},
Carlos Parés\footnote{Corresponding author. School of Mathematical Sciences, Departamento de Análisis Matemático, Estadística e Investigación Operativa, y Matemática aplicada, Universidad de Málaga, Bulevar Louis Pasteur, 31, 29010, Málaga, Spain. E-Mail: {\tt pares@uma.es}}
}
\date{}
\maketitle
\begin{abstract}
 This work aims to extend the well-known high-order WENO finite-difference methods for systems of conservation laws to nonconservative hyperbolic systems.
 The main difficulty of these systems both from the theoretical and the numerical points of view comes from the fact that the definition of weak solution is not unique: according to the theory developed by Dal Maso, LeFloch, and Murat in 1995, it depends on the choice of a family of paths. A new strategy is introduced here that allows non-conservative products to be written as the derivative of a generalized flux function that is defined locally on the basis of the selected family of paths. WENO reconstructions are then applied to this generalized flux. Moreover, if a Roe linearization is available, the generalized flux function can be evaluated through matrix-vector operations instead of path-integrals.
 Two different known techniques are used to extend the methods to problems with source terms and the well-balanced properties of the resulting schemes are studied. 
 These numerical schemes are applied to a coupled Burgers' system and to the two-layer shallow water equations in one- and two- dimensions to obtain high-order methods that preserve water-at-rest  steady states.
 

\end{abstract}
\keywords{WENO finite difference scheme, High order accuracy, Well-balanced scheme, Nonconservative equations, Path-conservative method.}
\section{Introduction}

We aim to construct well-balanced high-order WENO finite-difference schemes for hyperbolic nonconservative problems of the form
\begin{equation}\label{eq:1}
	U_t + A_1(U) U_x + A_2 (U) U_y = S_1(U)H_x + S_2(U)H_y,
 \end{equation}
where the unknown $U(x,y,t)$ takes value in an open convex set $\Omega \in \mathbb{R}^N$; $A_i(U)$, $i=1,2$ are smooth matrix-valued functions; $S_i(U)$, $i=1,2$ are vector-valued functions, and $H(x,y)$ is a known function from
$\mathbb{R}^2$ to $\mathbb{R}$.  The 1D case
 \begin{equation}\label{eq:3}
	U_t + A(U) U_x  = S(U)H_x
 \end{equation}
 will also be considered.

PDE systems of the form
\begin{equation}\label{eq:2}
	U_t + F_1(U)_x + F_2(U)_y + B_1(U)U_x + B_2(U)U_y= S_1(U)H_x + S_2(U)H_y,
 \end{equation}
 where $F_i(U)$, $i=1,2$ are the flux function and $B_i(U)$, $i=1,2$ are matrix-valued functions, can be considered as particular cases of \eqref{eq:1}.
 Systems of conservation laws ($B_i \equiv 0$, $S_i \equiv 0$, $i=1,2)$ and systems of balance laws ($B_i \equiv 0$, $i=1,2)$ are in turn particular cases of \eqref{eq:2}.

 The numerical methods will be applied to the 1D and 2D hyperbolic two-layer shallow-water equations that govern the flow of two superposed layers of immiscible homogeneous fluids. This system is used in different ocean and coastal engineering simulations and there is a vast literature focusing on its numerical analysis: see for instance 
 \cite{Castro-Diaz2011, CASTRO2004202, CHALMERS2013111, Chu2022, DONG2024106193, Liu2021, 10.2166/hydro.2020.207}. In order to illustrate the general procedure, we will also consider the 1D coupled Burgers' system introduced in \cite{Castro2001107}:
 \begin{equation} \label{CBs}
 \left\{
 \begin{array}{l}
 u_t + u u_x + u v_x = 0, \\
 v_t + v u_x + v v_x = 0,
 \end{array}
 \right.
 \end{equation}
 that can be written in the form \eqref{eq:3} with
 $$
 A(U) = \left[ \begin{array}{cc} u & u \\ v & v \end{array} \right],
 \quad U \in \Omega,$$
 with
 $$
 \Omega = 
 \left\{ \left[ \begin{array}{c} u \\ v \end{array} \right]  \in \mathbb{R}^2 \text{ s.t. }  u + v > 0\right\}.$$
 The eigenvalues of $A(U)$ are
 $$ \lambda_1(U) = 0, \quad \lambda_2(U) = u + v,$$
 and as
 so that the system is strictly hyperbolic in $\Omega$.

 The major difficulty of systems \eqref{eq:1} or \eqref{eq:2}, from both the theoretical and the numerical points of view, comes from the fact that the presence of nonconservative products makes that, unlike for systems of conservation laws, the definition of weak solution is not unique. In the theory developed by  Dal Maso, LeFloch, and Murat in \cite{Maso1995DefinitionAW}, nonconservative products are defined as Borel measures based on the choice of a family of paths, i.e. a Lipschitz-continuous function $\Psi: [0,1] \times \Omega \times \Omega  \to \Omega$ satisfying
 $$
 \Psi(0;U_L, U_R) = U_L, \quad \Psi(1; U_L, U_R) = U_R$$
 for all $U_L$, $U_R \in \Omega$, and
 $$
 \Psi(s; U,U) = U$$
 for all $U \in \Omega$ and $s \in [0,1]$. Once the family of paths has been used, the generalized Rankine-Hugoniot conditions satisfied by the admissible weak solutions at a jump are the following
 $$
 \sigma(U^+ - U^-) = \int_0^1 \left( \sum_{i=1}^2 n_i A_i(\Psi(s;U^-, U^+)) \right) \partial_s \Psi(s; U^-, U^+)\, ds,
 $$
 where $\sigma$ is the propagation speed; $U^\pm$ the lateral limits of the solution; and $\vec n = (n_1, n_2)$ a unit vector normal to the jump.  For instance, in the particular case of system \eqref{CBs}, it can be checked that 
  the choice of the family of straight segments 
 \begin{equation}\label{psi1} \Psi_1(s; U^-,U^+) = U^- + s(U^+ - U^-), \quad s \in [0,1] \end{equation}
leads to the jump conditions
\begin{equation}\label{jumppsi1}
 \left\{
 \begin{array}{l}
 \displaystyle \bar u (u^+ - u^-) + \bar u (v^+ - v^-) = \sigma (u^+ - u^-),\\
  \displaystyle \bar v (u^+ - u^-) + \bar v (v^+ - v^-) = \sigma (v^+ - v^-),
  \end{array}
  \right.
 \end{equation}
with
$$ \bar u = \frac{u^+ + u^-}{2}, \quad  \bar v = \frac{v^+ + v^-}{2},$$
while the choice 
\begin{equation}\label{psi2}
 \Psi_2(s; U^-,U^+) = 
 \left[
 \begin{array}{c}
 u^- + s(4 - 3s)(u^+ - u^-) \\
  v^- + s(v^+ - v^-) \\
  \end{array}
 \right], \quad s \in [0,1],
 \end{equation}
leads to 
 \begin{equation}\label{jumppsi2}
 \left\{
 \begin{array}{c}
 \displaystyle \bar u (u^+ - u^-) + u^+ (v^+ - v^-) = \sigma (u^+ - u^-),\\
  \displaystyle v^- (u^+ - u^-) + \bar v (v^+ - v^-) = \sigma (v^+ - v^-).
  \end{array}
  \right.
 \end{equation}
 Any choice of family paths leads to a consistent definition of weak solutions from the mathematical point of view. Therefore, given a particular application, the choice of the adequate family of paths has to be based on the consistency with the physics of the problem: for instance, the family of paths can be given by the viscous profiles corresponding to the neglected viscous terms (see \cite{CFMP2013} for a general discussion and \cite{Berthon} for the particular case of System \eqref{CBs}).

 Based on this theory, a framework for designing finite-volume methods for nonconservative systems was introduced in \cite{doi:10.1137/050628052} based on the concept of path-conservative methods. These methods have been extensively applied to solve nonconservative systems: see, for instance,  \cite{Diaz_Cheng_Chertock_Kurganov_2014, Castro_2006, CASTRO2017131, Castro2010Mc, doi:10.1137/110828873}.

 In this paper, we focus on WENO finite-difference schemes. In the past decades, these methods have been widely applied to systems of conservation and balance laws: see for instance \cite{BORGES20083191, CASTRO20111766, HENRICK2005542, JIANG1996202, Shu_1998, Zhu2017}.
In \cite{Chu2022} a fifth-order A-WENO finite-difference scheme for 1D and 2D systems of nonconservative hyperbolic systems was introduced. This scheme is based on the path-conservative central-upwind method and the global flux approach in which the integral of the nonconservative term is considered as a new flux function of the system. A high-order conservative method is then applied to the formal system of conservation laws. The evaluation of the new flux function requires the computation of the cell integrals of the nonconservative products using a high-order quadrature formula. Therefore, in addition to the flux reconstruction of the flux, it also requires high-order reconstructions of $U$ at the quadrature points. If this approach is used, the method is reduced to the conservative one obviously when $A(U)$ is the Jacobian of a flux function $F$. 
Recently, in \cite{Balsara2024} a finite-difference WENO method has been introduced based on state reconstruction that does not require the computation of cell integrals of the conservative products. In contrast to global flux-based methods, they reduce to the conservative method only when $A(U)$ is linear.

The goal of this paper is to introduce a \textit{local flux} approach to design high-order finite-difference methods for nonconservative systems. The key idea is to apply a standard high-order WENO reconstruction operator to the nonconservative products computed using the selected family of paths: more precisely, in the 1D case, the selected WENO operator will be applied to reconstruct quantities of the form
$$
\int_0^1  A^\pm(\Psi(s;U_i, U_j) ) \partial_s \Psi(s; U_i, U_j)\, ds, \quad j \in \mathcal{S}_i,
$$
where $\mathcal{S}_i$ represents the stencil of the $i$th point and $A^\pm(U)$ represents a splitting of the matrix system $A(U)$.   It will be seen that, if a path-consistent Roe linearization is available (see \cite{ROE1997250, TOUMI1992360}), the quantities to be reconstructed can be computed by using matrix-vector products instead of path-integrals.
The main advantages of this new one are the following:
\begin{itemize}
    \item the accuracy in space of the methods only depends on the order of the selected WENO reconstruction provided that the selected family of paths satisfies a symmetry property to be described;
    \item no integrals involving quadrature points in the cells have to be computed which avoids having to use an additional reconstruction operator with uniform accuracy in the entire cells;
    \item the methods reduce to the standard finite-difference conservative WENO schemes when $A(U)$ is the Jacobian matrix of a flux function $F(U)$.
\end{itemize}
Two different matrix splittings will be considered here based on the standard Lax-Friedrichs and Upwind approaches to illustrate the strategy. Nevertheless it can also be applied to more general splittings or even to WENO reconstructions that are not based on splitting technique, as in the case of A-WENO. The application of the local flux approach combined with A-WENO reconstructions will be discussed in a forthcoming work.

A relevant property to be satisfied by the numerical methods solving systems of the form \eqref{eq:1} or \eqref{eq:2} is the preservation of some or all the steady-state solutions of the system, i.e. the well-balanced property. For instance, in the context of the one or the two-layer shallow-water equations, a minimal requirement to the numerical methods is to exactly preserve the steady states corresponding to water-at-rest, i.e. to satisfy the {\it C-property} according to \cite{BERMUDEZ19941049}.
Different techniques have been proposed to design well-balanced schemes including hydrostatic reconstruction related \cite{S1064827503431090, Castro20072055}, relaxation methods \cite{Relaxation:10.1137/06067167X}, consistent discretization of the flux and source term \cite{LEVEQUE1998346, xing2005high}, etc. See also \cite{François_Bouchut_2010, doi:10.1137/0733001, Mandli_2013}.  

In the context of finite-volume methods, a general strategy to design well-balanced methods has been described in
\cite{Castro2020jsc}. 
In this strategy, a stationary solution whose average is the numerical approximation at every cell has to be computed at every time step. Then, a standard reconstruction operator is applied to the differences in the cell values at the stencil and the cell averages of the local stationary solution.  In \cite{PARES2021109880} this strategy has been extended to WENO finite-difference methods for systems of balance laws. Two different strategies will be followed here to obtain well-balanced numerical methods: one of them is the extension to nonconservative systems of the strategy introduced in this last reference, while the other consists of combining the Upwind splitting scheme with an adequate choice of family of paths. 

As it is well known (see \cite{Castro_2008}), in the case of nonconservative systems, the numerical solutions obtained with finite-difference or similar methods that are formally consistent with the definition of weak solution related to a given family of paths may converge to functions that are not weak solutions according to that family. Nevertheless, it will be seen in Section \ref{sec:numerical_sol} that the numerical results obtained for the two-layer shallow-water equations are similar to those obtained with other methods. Nevertheless, in order to ensure the convergence to functions that are weak solutions according to the selected family of paths, the numerical dissipation close to shocks has to be controlled: see for instance \cite{BeljadidCiCP-21-913}. In \cite{ PIMENTELGARCIA2022111152, pimentelgarca2024highorder} high-order finite-volume numerical methods that are able to correctly capture isolated shock waves have been designed based on the use of a discontinuous reconstruction operator in cells where a shock is detected: similar strategies can be adapted to the numerical methods introduced here, what will be done in future works.

The rest of the paper is organized as follows. In Section \ref{sec:Path_weno}, firstly, the new WENO path-conservative schemes for 1D homogeneous (i.e. without source terms) nonconservative systems are introduced and their high-order accuracy property is proved. In this section the general problem is considered, so that the family of paths is, in principle,  arbitrary. Nevertheless, it will be shown that a symmetry property has to be satisfied to ensure the high-order accuracy of the method. In Section \ref{sec:well-balance}, source terms are included and two strategies to obtain well-balanced methods are described. In Section \ref{sec:nuemrical_2D}, the proposed schemes are extended to 2D nonconservative systems. In Section \ref{Sec:two_swe}, we apply the proposed scheme to 1D and 2D two-layer shallow water equations: the numerical results are presented in Section \ref{sec:numerical_sol}. Finally, some conclusions are drawn in Section \ref{sec:conclusion}.

\section{Path-Conservative WENO finite-difference reconstruction methods}
\label{sec:Path_weno}
\subsection{Conservative WENO finite difference schemes: a brief overview }\label{ss:nummeth}
The goal of this paper is to extend high-order finite-difference methods based on flux reconstructions for conservation laws system
\begin{equation}
\label{eq:general_cf}
U_t + F(U)_x=0,
\end{equation}
to nonconservative systems
\begin{equation}
\label{eq:noncon}
U_t + A(U)U_x =0.
\end{equation}
It will be assumed here that the system is strictly hyperbolic, i.e. for every $U$, the matrix $A(U)$ has N different real eigenvalues
$$\lambda_1(U), \dots, \lambda_N(U).$$
In particular, systems of the form
\begin{equation}
\label{eq:general_noncon}
U_t + F(U)_x +B(U)U_x=0
\end{equation}
will be considered that can be written in the form \eqref{eq:noncon} with
$$
A(U) = J(F(U)) + B(U),$$
where $J(F(U))$ represents the Jacobian of the flux function $F(U)$. 

Semi-discrete high-order finite-difference methods for systems of conservation laws \eqref{eq:general_cf} have the form:
\begin{equation}
\label{eq:general_disc}
\frac{dU_i}{dt} + \frac{1}{\Delta x}\left({F}_{i+1/2} - {F}_{i-1/2}\right) = 0,
\end{equation}
where $F_{i + 1/2}$ is a high-order reconstruction of the flux function. 
The computational domain is \( [a, b] \). Uniform meshes with a constant step size \( \Delta x \) will be considered, with cell centers denoted as \( x_i \). The following notation is used for the cell interface:
$$
x_{i + \frac{1}{2}} = x_i + \frac{\Delta x}{2}.
$$
In the particular case of the  WENO reconstruction of order $p = 2k + 1$ for systems of balance laws, two flux reconstructions are computed using the values at the points $x_{i-k}, \dots, x_{i+k}$:
\begin{eqnarray}
F^L_{i+1/2} &  = &  \mathcal{R}^L(F(U_{i-k}), \dots, F(U_{i+k})), \\
F^R_{i-1/2} & = &  \mathcal{R}^R(F(U_{i-k}), \dots, F(U_{i+k})). 
\end{eqnarray}
These are the so-called left- and right-biased reconstructions, related by the equality:
$$
\mathcal{R}^L(F(U_{i-k}), \dots, F(U_{i+k})) = \mathcal{R}^R(F(U_{i+k}), \dots, F(U_{i-k})). 
$$
In order to compute the numerical flux $F_{i+ 1/2}$, first a splitting of the flux function  is considered
$$
F(U) = F^+(U) + F^-(U),
$$
in such a way that the eigenvalues of the Jacobian $J^+(U)$ (resp. $J^-(U)$) of $F^+(U)$ (resp. $F^-(U)$) are positive (resp. negative). A standard choice is the Lax-Friedrichs flux-splitting:
$$
F^\pm(U) = \frac{1}{2} \left( F(U) \pm \alpha U \right),
$$
where $\alpha$ is the maximum of the absolute value of the eigenvalues of $\{J(U_i)\}$, this maximum being taken over either local (WENO-LLF) or global (WENO-LF):  see \cite{JIANG1996202, Shu_1998}.

Then, the reconstruction operator is applied to $F^\pm$:
\begin{eqnarray}
F^{+}_{i+1/2} &  = &  \mathcal{R}^L(F^+(U_{i-k}), \dots, F^+(U_{i+k})), \\
F^{-}_{i+1/2} & = &  \mathcal{R}^R(F^-(U_{i+1-k}), \dots, F^-(U_{i+1+k})), 
\end{eqnarray}
and finally,
\begin{equation}\label{split}
F_{i+1/2} =  F^{+}_{i+1/2} + F^{-}_{i+1/2}.
\end{equation}
The reconstruction then satisfies
$$
\frac{1}{\Delta x} \left( F_{i + 1/2} - F_{i-1/2} \right) = F(U)_x + O(\Delta x^{2k+1}), \quad \forall i.
$$ 
Any version of WENO reconstructions can be selected. In particular, in the numerical tests shown in Section \ref{sec:numerical_sol}, WENOZ is used (see \cite{CASTRO20111766}): for the sake of completeness, the expression of the fifth-order WENOZ reconstruction is recalled in Appendix \ref{app:A}.

\subsection{Extension to nonconservative systems }\label{ss:nummeth_noncon}
In order to extend these numerical methods  to \eqref{eq:noncon}, let us first rewrite \eqref{eq:general_disc} as follows:
\begin{equation}
\label{eq:minus_disc}
\frac{dU_i}{dt}  + \frac{1}{\Delta x}\left({F}_{i+1/2} - {F}\left(U_{i}\right)+{F}\left(U_{i}\right)- {F}_{i-1/2}\right) = 0,
\end{equation}
or, equivalently
\begin{equation}
\frac{dU_i}{dt} + \frac{1}{\Delta x} \left( \widehat{D}^-_{i + 1/2} +  \widehat{D}^+_{i - 1/2}\right) = 0,
\label{eq:numerical_sch}
\end{equation}
with
$$ \widehat{D}^-_{i + 1/2} = {F}_{i+1/2} - {F}\left(U_{i}\right), \quad  \widehat{D}^+_{i - 1/2} =  {F}\left(U_{i}\right)- {F}_{i-1/2}.
$$

One has then
\begin{eqnarray*}
\widehat{D}^-_{i + 1/2} & = &  {F}_{i+1/2} - {F}(U_{i})\\
& = &  F^{+}_{i+1/2} - {F^+}(U_{i}) +  F^{-}_{i+1/2} -  {F^-}(U_{i}) \\
& = & \mathcal{R}^L(F^+(U_{i-k}), \dots, F^+(U_{i+k})) - {F^+}(U_{i}) \\
& + & \mathcal{R}^R(F^-(U_{i+1-k}), \dots, F^-(U_{i+1 + k})) - {F^-}(U_{i})\\
& = & \mathcal{R}^L(F^+(U_{i-k}) -  {F^+}(U_{i}), \dots, F^+(U_{i+k}) - {F^+}(U_{i}) )\\
& + & \mathcal{R}^R(F^-(U_{i+1-k}) - {F^-}(U_{i}), \dots, F^-(U_{i+1 + k}) - {F^-}(U_{i})) \\
& = & \mathcal{R}^L(D^+_{i,i-k}, \dots, D^+_{i,i+k} ) + \mathcal{R}^R(D^-_{i, i+1-k}, \dots, D^-_{i,i+1+k}) ,
\end{eqnarray*}
where the following notation has been used
$$
D^\pm_{j,k} = F^\pm(U_k) - F^\pm(U_j), \quad \forall j,k.
$$
Analogously
\begin{eqnarray*}
\widehat{D}^+_{i - 1/2} & = &  {F}(U_{i}) - F_{i-1/2}\\
& = & \mathcal{R}^L(D^+_{i-1-k,i}, \dots, D^+_{i-1 +k,i} ) + \mathcal{R}^R(D^-_{i-k,i}, \dots, D^-_{i+k,i}) .
\end{eqnarray*}

Finally, the last ingredient required to extend the numerical methods to nonconservative systems is a family of paths. Let us consider, in principle, an arbitrary family
$\Psi:  [0,1] \times \Omega \times \Omega \to \Omega$. Using $\Psi$ we have:
$$
D^\pm_{j,k} = F^\pm(U_k) - F^\pm(U_j) = \int_0^1
JF^\pm \left( \Psi(s; U_j, U_k) \right)\partial_s \Psi(s; U_j, U_k)\, ds,
$$
where $JF^\pm$ represent the Jacobian matrices of $F^\pm$.
The natural extension of the numerical method \eqref{eq:general_disc} to the system \eqref{eq:noncon} is then given by \eqref{eq:numerical_sch} with
\begin{eqnarray}
\widehat{D}^-_{i + 1/2} & = &  \mathcal{R}^L(D^+_{i,i-k}, \dots, D^+_{i,i+k} ) + \mathcal{R}^R(D^-_{i, i+1-k}, \dots, D^-_{i,i+1+k}) , \label{eq:D-}\\
\widehat{D}^+_{i - 1/2} & = &  \mathcal{R}^L(D^+_{i-1-k,i}, \dots, D^+_{i-1 +k,i} ) + \mathcal{R}^R(D^-_{i-k,i}, \dots, D^-_{i+k,i}) . \label{eq:D+}
\end{eqnarray}
and
\begin{equation}\label{eq:Djk}
D^\pm_{j,k} = \int_0^1 A^\pm(\Psi(s; U_j, U_k) ) \partial_s \Psi(s; U_j, U_k)\, ds,
\end{equation}
where 
$A^\pm(\Psi(s; U_j, U_k)$ is a matrix-splitting to be adequately chosen.
 For instance, for system \eqref{CBs}, since the two eigenvalues are positive, a natural choice is given by 
\begin{equation}\label{CBssplit}
  A^+(\Psi(s; U_j, U_k) = A(\Psi(s; U_j, U_k), \quad A^-(\Psi(s; U_j, U_k) = 0. 
\end{equation}
Please note that, while for systems of conservation laws only two reconstructions per intercell are required, here 4 reconstructions are needed in $x_{i+1/2}$: two reconstructions to compute $\widehat{D}^+_{i + 1/2}$ and two others to compute $\widehat{D}^-_{i + 1/2}$.

Observe that, while in the case of a system of conservation laws, the resulting numerical method is independent of the chosen family of paths (since it is equivalent to \eqref{eq:general_disc}), for nonconservative systems the numerical method depends on the chosen family of paths. For instance, in the particular case of System \eqref{CBs}, if the matrix-splitting is given by \eqref{CBssplit},
the choice of the family of paths \eqref{psi1}, whose corresponding jump condition is \eqref{jumppsi1}, leads to
\begin{equation}\label{fluc1}
D^+_{j,k} =  \left[
 \begin{array}{c}
 \displaystyle \bar u_{j,k} (u_k - u_j) + \bar u_{j,k} (v_k - v_j) \\
  \displaystyle \bar v_{j,k} (u_k - u_j) + \bar v_{j,k} (v_k - v_j) 
  \end{array}
  \right], \quad D^-_{j,k} = 0,
\end{equation}
with
$$
\bar u_{j,k} = \frac{u_j + u_k}{2}, \quad \bar v_{j,k} = \frac{v_j + v_k}{2}, $$
while the choice \eqref{psi2}, whose corresponding  jump condition is \eqref{jumppsi2}, 
leads to
\begin{equation}\label{fluc2}
D^+_{j,k} =  \left[
 \begin{array}{c}
 \displaystyle \bar u_{j,k} (u_k - u_j) + u_{k} (v_k - v_j) \\
  \displaystyle v_{j} (u_k - u_j) + \bar v_{j,k} (v_k - v_j) 
  \end{array}
  \right], \quad D^-_{j,k} = 0.
\end{equation}
Definitions \eqref{psi1} and \eqref{psi2} lead to different numerical results: both of them are convergent for smooth solutions but, as it will be seen in Section \ref{ss:accuracy}, only  \eqref{fluc1} leads to a high-order accurate method. On the other hand, the two methods are expected to give different results for discontinuous solutions, since they are formally consistent with the different jump conditions, \eqref{jumppsi1} or \eqref{jumppsi2}, corresponding to the selected paths. According to \cite{Castro_2008} this formal consistency does not ensure that the limits of the numerical solutions satisfy the expected jump conditions. In fact, the methods introduced here can fail in capturing correctly the discontinuities, as any other standard finite-difference type method. Nevertheless, they can be combined with techniques like the ones recently developed in \cite{PIMENTELGARCIA2022111152, pimentelgarca2024highorder} to improve their convergence to the sought weak solutions.   


While for \eqref{CBs} the fluctuations  \eqref{eq:Djk}
can be easily computed, this may be more difficult in other cases where numerical quadrature can be used to compute the integrals. Nevertheless, an alternative form of the method can be given if a Roe linearization is available in which the path-integrals are replaced by matrix-vector products. Remember that a Roe linearization (see \cite{ROE1997250, TOUMI1992360})  is a matrix-valued function $A_{\Psi}: \Omega \times \Omega \mapsto \mathbb{R}^{N }\times \mathbb{R}^{ N}$ that satisfies the following properties:
\begin{enumerate}
\item For each $U, V \in \Omega$ , $A_{\Psi}\left(U,V\right)$ has $N$ distinct real eigenvalues:
$$
\lambda_1\left(U, V\right)<\lambda_2\left(U, V\right)<\cdots<\lambda_N\left(U, V\right) .
$$

 \item $A_{\Psi}(U, U)=A(U)$, for every $U \in \Omega  $.

\item  For any $U, V \in \Omega $,
\begin{equation}
A_{\Psi}\left(U, V\right)\left(V - U\right)=\int_0^1 A\left(\Psi\left(s ; U, V \right)\right) \frac{\partial \Psi}{\partial s}\left(s ; U, V\right) ds .
\label{eq:integral_prop}
\end{equation}
\end{enumerate}
For instance, it can be easily checked that  the matrices
\begin{equation}\label{eq:CBs_Roe}
A_{\Psi_1}(U^-,U^+) = \left[
\begin{array}{cc}
\bar u & \bar u \\
\bar v & \bar v 
\end{array}
\right], \quad 
A_{\Psi_2}(U^-,U^+) = \left[
\begin{array}{cc}
\bar u & u^+ \\
v^- & \bar v 
\end{array}
\right],
\end{equation}
are Roe linearizations for system \eqref{CBs} related to the family of paths $\Psi_1$ and $\Psi_2$ given by \eqref{psi1} and \eqref{psi2} respectively.  If
a Roe linearization is available (as it is the case for the two-layer shallow-water system if the family of straight segments is selected) the fluctuations $D_{j,k}^{\pm}$ can be computed as in \cite{Castro2010Mc}:
\begin{equation}\label{eq:Djk_Roe}
D^\pm_{j,k} = A^\pm_{\Psi}\left(U_j, U_k\right)\left(U_k - U_j\right),
\end{equation}
where 
\begin{equation}
A^\pm_{\Psi}\left(U_j,U_k\right)=\frac{1}{2}\left(A_{\Psi}\left(U_j,U_k\right)\pm Q_{\Psi}\left(U_j,U_k\right)\right)
    \label{eq:Q_split}
\end{equation}
represents a splitting of the Roe linearization. Two different splittings will be considered here:
\begin{itemize}
\item 
Upwind splitting:
 \begin{equation} \label{eq:Rusanov_sch}
Q_{\Psi}\left(U_j, U_k\right) = \left|A_{j,k} \right|,
\end{equation}
where
$$
\left|A_{j,k}\right| = {R}_{j,k} \left| \Lambda_{j,k} \right| {L}_{j,k}. 
$$
Here,  $\left|\Lambda_{j,k} \right|$ is the $N$-dimensional diagonal matrix whose coefficients are the absolute values of the eigenvalues of $A_{j,k}$:
$$ |\lambda_{j,k; 1}|,  \dots,|\lambda_{j,k;N}|;
$$
$R_{j,k}$ is a matrix whose $l$th column $\vec r_{j,k;l}$ is an eigenvector associated to $\lambda_{j,k;l}$; and ${L}_{j,k}= {R_{j,k}^{-1}}$ is a matrix whose arrows are left-eigenvalues.
A standard entropy-fix can be used to avoid the appearance of non-entropy discontinuities, like considering a regularization $|\cdot|_\epsilon$ of the absolute value function like in \cite{HARTEN1983235, SHU1988439}.
\item Lax-Friedrichs(LF) splitting: 
\begin{equation} \label{eq:lf_sch}
Q_{\Psi}\left(U_j, U_k\right) = \alpha I,
\end{equation}
 where $I$ is the identity matrix and $\alpha $ is the global maximum of the absolute value of the eigenvalues, $\alpha  \geq \lvert \lambda_{j,k;l} \rvert,\; l=1, \dots, N$.
 \end{itemize}
The matrices involved in the Upwind splitting can be equivalently written as follows
\begin{equation}
A_\Psi^\pm\left(U, V \right) = P_{\Psi}^\pm\left(U, V\right)A_{\Psi}\left(U, V\right),
\label{eq:PpmA}
\end{equation}
where
\begin{equation}\label{eq:Ppm}
 P^\pm_{\Psi}\left(U, V\right)= R_{\Psi}\left(U, V\right)  M^\pm_{\Psi}\left(U, V\right)  R^{-1}_{\Psi}\left(U, V\right) .
\end{equation}
Here $M^\pm_{\Psi}\left(U, V\right)$ represents the diagonal matrix whose coefficients are
$$
\frac{1}{2}\left(1 \pm \operatorname{sign}\left(\lambda_l(U,V) \right)\right), \quad l=1, \ldots, N,
$$
and $R_{\Psi}\left(U, V\right)$ is a matrix whose $l$th columns is an eigenvector $\vec r_l\left(U, V\right)$ associated to $\lambda_l(U,V)$. The fluctuations corresponding to this splitting can be then computed as follows: given two indices $j,k$, first the coordinates  $\{ \alpha_{j,k;l} \}_{l=1}^N$ of $U_k - U_j$ in the basis of eigenvectors of the Roe matrix $A_{j,k}$, i.e.
$$
U_k - U_j = \sum_{l=1}^N \alpha_{j,k;l}\vec r_{j,k; l},
$$
are computed by solving a linear system 
\begin{equation}\label{linsys}
R_{j,k} \vec \alpha_{j,k} = U_k - U_j.
\end{equation}
Then  one has
$$
D^\pm_{j,k} =  \sum_{l=1}^N  \alpha_{j,k;l}\lambda^\pm_{j,k;l} \vec r_{j,k; l},$$
where, given $\lambda \in \mathbb{R}$, $\lambda^\pm$ represent the positive and negative part of $\lambda$, i.e.
\[
\lambda^{+}= \frac{\lambda + \left| \lambda \right|}{2}, \quad \lambda^{-} = \frac{\lambda - \left| \lambda \right|}{2}.
\]
On the other hand, the method based on the LF splitting may be oscillatory if the reconstructions are not performed in characteristic fields. To avoid this, the reconstructions are computed in practice as follows:
\begin{eqnarray}
\widehat{D}^-_{i + 1/2} & = &   R_{i,i+1} \mathcal{R}^L( L_{i,i+1} D^+_{i,i-k}, \dots,  L_{i,i+1} D^+_{i,i+k} ) \nonumber \\
& & +  R_{i,i+1} \mathcal{R}^R( L_{i,i+1} D^-_{i, i+1-k}, \dots,  L_{i,i+1} D^-_{i,i+1+k}), \label{eq:D-char}\\
\widehat{D}^+_{i - 1/2} & = &  R_{i-1,i}  \mathcal{R}^L(L_{i-1,i}D^+_{i-1-k,i}, \dots, L_{i-1,i} D^+_{i-1 +k,i} ) \nonumber \\
& & +  R_{i-1,i} \mathcal{R}^R(L_{i-1,i} D^-_{i-k,i}, \dots, L_{i-1,i} D^-_{i+k,i}) . \label{eq:D+char}
\end{eqnarray}

Observe that, if the LF splitting is chosen and reconstructions in characteristic variables are performed, the right and left eigenvectors of the Roe matrices $A_{i,i+1}$ have to be computed.  On the other hand, if the Upwind reconstruction is selected, the eigenvectors and eigenvalues of all the Roe matrices $A_{k,j}$ are required. Moreover, the linear system \eqref{linsys}
has to be solved. Therefore, it is more computationally expensive for homogeneous problems. Nevertheless, it will be seen in Section \ref{sec:well-balance} that the numerical treatment of the source term can compensate for this disadvantage.

\begin{remark}
Since the expression of WENO reconstructions is a linear combination of the fluxes whose  coefficients depend nonlinearly on the data through the smoothness indicators, it can be shown that, for problems of the form
\eqref{eq:general_noncon}, 
 the numerical method \eqref{eq:general_disc} can be written in the form
\begin{equation}
\frac{dU_i}{dt}  = -\frac{1}{\Delta x} \left( F_{i+1/2} - F_{i-1/2} + \widehat{B}^-_{i + 1/2} +  \widehat{B}^+_{i - 1/2}\right),
\label{eq:numerical_sch_2}
\end{equation}
where $F_{i + 1/2}$ and  $\widehat B_{i + 1/2}$  are, respectively, standard WENO reconstructions of the flux function and  the nonconservative terms, in which the nonlinear coefficients are the same.
\end{remark}

\subsection{Accuracy of the methods}\label{ss:accuracy}

Let us check that \eqref{eq:numerical_sch}-\eqref{eq:D-}-\eqref{eq:D+}-\eqref{eq:Djk} is a high-order numerical method for \eqref{eq:noncon}. 
\begin{proposition} \label{prop:order} Let us consider a smooth solution $U(x,t)$ of \eqref{eq:noncon} and assume that $A(U)$ and $\Psi$ are smooth. We also assume that $\Psi$ satisfies
\begin{equation} 
    \int_0^1 A(\Psi(s; V, U )) \partial_s \Psi(s; V, U)\, ds = - \int_0^1 A(\Psi(s; U, V )) \partial_s \Psi(s; U, V)\, ds
    \label{eq:path_cons}
\end{equation}
for all $U, V \in \Omega$.
Then we have
\begin{equation}
\partial_t U(x_i,t) + \frac{1}{\Delta x} \left( \widehat{D}^-_{i + 1/2} +  \widehat{D}^+_{i - 1/2}\right) = O(\Delta x^{2k + 1}), \label{eq:order}
\end{equation}
where $p = 2k+1$ is the order of the reconstruction operator.
\end{proposition}
\begin{proof}
Given an index $i$ and a time $t$, let us define the function $G^t_i(x)$ as follows:
\begin{equation}\label{eq:G_i}
G^t_i(x) = \int_0^1 A(\Psi(s; U(x_i,t), U(x,t) ) \partial_s \Psi(s; U(x_i,t), U(x,t))\, ds.
\end{equation}
This function satisfies
$$
G^t_i(x_i) = 0, \quad \partial_x G_i^t(x_i) = A(U(x_i,t))U_x(x_i, t).$$
In effect,
$$
G_i^t(x_i) =  \int_0^1 A(\Psi(s; U(x_i,t), U(x_i, t) ) \partial_s \Psi(s; U(x_i,t), U(x_i,t))\, ds = 0, $$
since
$$
\Psi(s; U(x_i,t), U(x_i,t)) = U(x_i, t), \quad \forall s.
$$
On the other hand:
\begin{eqnarray*}
\partial_x G^t_i(x_i) &=& \lim_{h \to 0 }\frac{G_i^t(x_i + h) - G^t_i(x_i)}{h} \\
& = & \lim_{h \to 0} \frac{1}{h}  \int_0^1 A(\Psi(s; U(x_i,t), U(x_i+h,t) ) )\partial_s \Psi(s; U(x_i,t), U(x_i + h,t))\, ds \\
&  = &  \lim_{h \to 0} \int_0^1 \frac{1}{h}\left(A(\Psi(s; U(x_i,t), U(x_i + h,t) ) ) - A(U(x_i, t)) \right) \partial_s \Psi(s; U(x_i,t), U(x_i + h,t))\, ds \\
&  & \quad + \lim_{h \to 0} A(U(x_i,t))  \frac{1}{h} \int_0^1\partial_s \Psi(s; U(x_i,t), U(x_i + h,t))\, ds \\
&  = &  \lim_{h \to 0} \int_0^1 \frac{1}{h}\left(A(\Psi(s; U(x_i,t), U(x_i + h,t) )) - A(U(x_i, t)) \right) \partial_s \Psi(s; U(x_i,t), U(x_i + h,t))\, ds \\ 
&  & \quad + \lim_{h \to 0} A(U(x_i,t))  \frac{U(x_i + h, t) - U(x_i, t)}{h}\\
& = & A(U(x_i,t))U_x(x_i, t),
\end{eqnarray*}
where, in the first term, it has been used again that 
$$\partial_s \Psi(s; U(x_i,t), U(x_i + h,t)) \to \partial_s \Psi(s; U(x_i,t), U(x_i,t)) = 0.
$$
Observe that, for all $j$:
$$
D_{i,j} = G^t_i(x_j), \quad D_{j,i} = -G_i^t(U)(x_j).$$
Therefore
\begin{eqnarray*}
\widehat{D}^-_{i + 1/2} & = &  \mathcal{R}^L(D^+_{i,i-k}, \dots, D^+_{i,i+k} ) + \mathcal{R}^R(D^-_{i, i+1-k}, \dots, D^-_{i,i+1+k}) \\
 & = &  \mathcal{R}^L(G_i^{t,+}(x_{i-k}), \dots, G_i^{t,+}(x_{i+k}) ) + \mathcal{R}^R(G^{t,-}_i(x_{i+1-k}), \dots, G^{t,-}_i(x_{i+1+k})) = \hat G^t_{i, i+1/2},\\
\widehat{D}^+_{i - 1/2} & = &  \mathcal{R}^L(D^+_{i-1-k,i}, \dots, D^+_{i-1 +k,i} ) + \mathcal{R}^R(D^-_{i-k,i}, \dots, D^-_{i+k,i}) \\
& = &  - \mathcal{R}^L(G^{t,+}_i(x_{i-1-k}), \dots, G^{t,+}_i(x_{i-1 +k} )) -\mathcal{R}^R(G^{t,-}_i(x_{i-k}), \dots,  G^{t,-}_i(x_{i+k})) = - \hat G_{i, i-1/2},
\end{eqnarray*}
where $G_i^{t,\pm}$ represents the splitting of the function $G_i^t$ and $\hat G^t_{i, i\pm1/2}$ is its WENO reconstruction. Therefore we have:
\begin{eqnarray*}
 \frac{1}{\Delta x} (\widehat{D}_{i-1/2}^+ + \widehat{D}_{i+1/2}^- ) &  =  & \frac{1}{\Delta x} (\hat G_{i,i+1/2} - \hat G_{i,i-1/2} ) \\
& = & \partial_x G^t_i(x_i) + O(\Delta x^{2k+1}) \\
& =  &A(U(x_i, t)) U_x(x_i, t) + O(\Delta x^{2k+1}),
\end{eqnarray*}
which leads to \eqref{eq:order}.
 \end{proof}

 The symmetry condition \eqref{eq:path_cons} is satisfied by the family of straight segments
 $$\Psi(s; U,V) = U + s(V-U).$$
 In effect
 \begin{eqnarray*}
 \int_0^1 A(\Psi(s; V, U ) )\partial_s \Psi(s; V, U)\, ds & = &
 \left(\int_0^1 A(V + s(U-V))\, ds \right) (U - V) \\
 & = & - \left( \int_0^1 A(U + s(V-U))\, ds \right) (V - U) \\
& =  &  - \int_0^1 A(\Psi(s; U, V ) ) \partial_s \Psi(s; U, V)\, ds.
\end{eqnarray*}
Therefore, in the particular case of system \eqref{CBs}, the definition \eqref{fluc1}, based on the choice of straight segments, leads to a high-order method. On the other hand, \eqref{eq:path_cons} is not satisfied for \eqref{psi2} and, it will be seen in  Section \ref{num:bug} that the method corresponding to \eqref{fluc2} is only first-order accurate, what shows that this condition is necessary as well.

 According to the proof of Proposition \ref{prop:order}, the numerical method can be interpreted as follows: the PDE system is first formally rewritten as the system of balance laws
 \begin{equation}\label{eq:scl_lf}
\partial_t{U_i}+ \partial_x\mathcal{F}^L_{i} = 0,
\end{equation}
with
$$
 \mathcal{F}^L_{i}(x,t) = G_i^t(x),
 $$
where $G_i$ is the function given by \eqref{eq:G_i}; then WENO reconstructions are applied to the generalized flux function $\mathcal{F}^L_i$. In the particular case of a system of the form \eqref{eq:general_noncon} it can be easily checked that this is equivalent to reconstructing the generalized flux function
$$
\mathcal{F}^L = F + \mathcal{B}^L_i,$$
where
$$
\mathcal{B}^L_i(x,t) = \int_0^1 B(\Psi(s; U(x_i,t), U(x,t) ) \partial_s \Psi(s; U(x_i,t), U(x,t))\, ds,$$
which is defined for every $i$, while in the global-flux approach 
the generalized flux function to be reconstructed is
$$
\mathcal{F}^G = F + \mathcal{B}^G,$$
where
$$
\mathcal{B}^G(x,t) = \int_{a}^x B(U)U_x \, dx$$
is globally defined.
This is why it was said above that a \textit{local flux} approach is followed here. Following this approach,  $\mathcal{B}^L_i$ is approximated at the node points of the stencil as follows
$$
\mathcal{B}^L_i(x_j,t) \approx B_{\Psi}(U_i, U_j) (U_j - U_i), \quad j \in \mathcal{S}_i, $$
where $B_\Psi$ is the linearization of $B$ used in the Roe matrix, While in the case of the global flux approach (similar to the approach in the finite volume method as \cite{CAO2023111790}), $\mathcal{B}^G$ is numerically approximated at the nodes using a recursive formula such that
$$
\mathcal{B}^G_0 = 0; \quad \mathcal{B}^G_{i+1}
= \mathcal{B}^G_{i} + \Delta x \sum_{l=0}^M \alpha_l B(U_i^l) D_xU_i^l,
\quad i =0,\dots, NP-1,
$$
where $\alpha^l$, $l = 0, \dots, M$ are the weights of the selected quadrature form and $U_i^l$, $D_xU_i^l$, $l = 0, \dots, M$ are high-order approximations of $U$ and $U_x$ at the quadrature points, $NP$ is the total number of discrete points. Therefore $M+1$ additional reconstructions (of state in this case) are necessary. Summing up, while in the local flux approach two flux WENO reconstructions per point are needed (to compute $\widehat{D}_{i-1/2}^+$ and $\widehat{D}_{i+1/2}^-$), in the global flux approach one flux WENO reconstruction per intercell and  $M+1$ state reconstructions per cell are necessary. It means that the local flux approach requires $2NP$ flux reconstructions per stage of the ODE solver used for the temporal discretization while the global flux approach requires $(M+2)NP$. On the other hand, if the order of the WENO reconstruction is $2k + 1$, then $M$ has to be greater or equal than $k$ so that the numerical method preserves the order of the WENO reconstructions. Therefore, the number of reconstructions in the global flux approach is greater than $(k+2)NP$ compared to the $2NP$ reconstructions in the local flux approach.

\section{Problems with source terms and well-balanced property}
\label{sec:well-balance}
\subsection{Well-balanced property}
Let us first consider a system of the form \eqref{eq:noncon} in which $\lambda = 0$ is an eigenvalue of $A(U)$ for every $U \in \Omega$. The well-balanced property of the methods is related to the preservation of the stationary solutions $U^*$ of the system, which satisfy the equation
$$
A(U^*)U^*_x = 0.$$
Observe that, $U^*_x$ is an eigenvector associated with the null eigenvalue for all $x$ such that $U^*(x)_x \not= 0$.
As an example, it can be easily checked that the  stationary solutions of \eqref{CBs} are the set of functions
\begin{equation}\label{eq:ssCBs}
U^*(x) = \left[ \begin{array}{c} u^*(x) \\ v^*(x) \end{array}\right]
\text{ s.t. } u^*(x) + v^*(x) = constant.
\end{equation}
 The method described in Section \ref{ss:nummeth_noncon} has then the well-balanced property given by the following results the proof of which is trivial:
\begin{proposition}\label{prop:wb_St}
Let $U^*(x)$ be a stationary solution of system \eqref{eq:noncon}. If, for every $x_L < x_R$ one has
\begin{equation}\label{eq:Apsiwb}
A_\Psi(U^*(x_L), U^*(x_R)) (U^*(x_R) - U^*(x_L)) = 0
\end{equation}
and the selected matrix-splitting is such that
\begin{equation}\label{eq:spl_wb}
A_\Psi(U,V) (V- U) = 0 \implies A^\pm_\Psi(U,V) (V- U) = 0, 
\end{equation}
then the numerical method  \eqref{eq:numerical_sch}-\eqref{eq:D-}-\eqref{eq:D+}-\eqref{eq:Djk_Roe} is well-balanced for $U^*$, i.e. $\{ U^* (x_i) \}$ is an equilibrium of the ODE system \eqref{eq:numerical_sch}.
\end{proposition}
Observe that  \eqref{eq:spl_wb} is satisfied for the Upwind splitting approach, as can be easily deduced from \eqref{eq:PpmA}, but not for the LF splitting: in effect, in this case one has
$$
A_\Psi(U,V) (V- U) = 0 \implies A^\pm_\Psi(U,V) (V- U) = \pm \frac{\alpha}{2} (V - U). $$
Nevertheless, the modification of the identity matrix technique introduced in \cite{Castro2010Mc} can be applied to modify the splitting so that \eqref{eq:Apsiwb} is satisfied.

As an application of Proposition \ref{prop:wb_St}, it can be easily checked that the property \eqref{eq:Apsiwb} is satisfied for the Roe matrix $A_{\Psi_1}$ defined in \eqref{eq:CBs_Roe} for every stationary solution \eqref{eq:ssCBs} of System \eqref{CBs}. Therefore, the choices of the family of straight segments and the Upwind splitting lead to a numerical method for \eqref{CBs} that is fully well-balanced (and high-order accurate). On the other hand, \eqref{eq:Apsiwb} is not satisfied for the Roe matrix $A_{\Psi_2}$.

Property \eqref{eq:Apsiwb} is discussed in \cite{M2AN_2004} in relation with the well-balanced property of Roe methods. Let us only recall that, given a stationary solution $U^*$, this property is satisfied if the family of paths is such that, for all $x_L, x_R \in \mathbb{R}$ with $x_L < x_R $, the functions
$$
s \in [0,1] \to \Psi \left(s; U^*(x_L), U^*(x_R) \right) \in \Omega
$$
and
$$
x \in [x_L, x_R] \to U^*(x) \in \Omega$$
define the same curve in $\Omega$.
In effect, if this is the case one has:
\begin{eqnarray*}
A_\Psi(U,V) (V- U) & = & \int_0^1 A(\Psi(s; U, V) ) \partial_s \Psi(s; U, V)\, ds \\
& = & \int_{x_L}^{x_R} A(U^*(x )) U^*_x(x) \,dx   \\
& = & 0,
\end{eqnarray*}
where a change of parameter has been applied to obtain the second equality and the fact that $U^*$ is a stationary solution has been used in the third one. In particular, if the family of straight segments is chosen, the property \eqref{eq:Apsiwb} is satisfied for all stationary solutions such that the curve defined by $x \to U^*(x)$ lies in a straight line: this is the case of \eqref{CBs} whose stationary solutions \eqref{eq:ssCBs} lie in a straight line of equation $u + v = constant$ in the $u,v$ plane.

\subsection{Source terms}
Let us consider now problems with source term
\begin{equation}
\label{eq:noncon+st}
U_t + A(U)U_x =S(U)H_x,
\end{equation}
where $H(x)$ is a known function, whose stationary solutions satisfy
\begin{equation}
A(U^*)U^*_x =S(U^*)H_x.
\end{equation}
\subsubsection{Strategy 1 }
\label{ss:wb_s1}
The first strategy consists in writing \eqref{eq:noncon+st} in the form \eqref{eq:noncon} as follows (see \cite{Castro2009,  M2AN_2004}):
\begin{equation}\label{eq:nonconsform}
W_t + \mathcal{A}(W)W_x = 0,
\end{equation}
with
$$
W = \left[ \begin{array}{c} U \\ H  \end{array}\right] \in \Omega \times \mathbb{R}, \quad \mathcal{A}(W) = \left[ \begin{array}{c|c} A(U) & -S(U) \\\hline 0 & 0 \end{array}\right],$$
and then the  strategy described in Section \ref{ss:nummeth_noncon} is applied to \eqref{eq:nonconsform}. To do this, a family of paths
$$
\widetilde{\Psi}(s; W_L, W_R) = \left[ \begin{array}{c} \Psi_U (s; W_L, W_R) \\ \Psi_H (s; W_L, W_R) \end{array} \right]
$$
satisfying \eqref{eq:path_cons} (as, for instance, the family of straight segments) and a Roe linearization have to be chosen first. As in \cite{M2AN_2004}, let us assume that a Roe matrix of the form
$$
 \mathcal{A}_{\widetilde{\Psi}}(W_L, W_R) = \left[ \begin{array}{c|c} A_{\widetilde \Psi}(W_L,W_R) & -S_{\widetilde{\Psi}}(W_L,W_R) \\\hline 0 & 0 \end{array}\right]
 $$
is available, where
\begin{itemize}
    \item $A_{\widetilde{\Psi}}(W_L, W_R)$ has $N$ real different eigenvalues $\lambda_i(W_L,W_R)$, $i = 1, \dots, N$;
    \item $A_{\widetilde{\Psi}}(W, W) = A(W)$ for all $W = [U, H]^T$;
    \item $S_{\widetilde{\Psi}}(W, W) = S(W)$ for all $W = [U, H]^T$; 
    \item for all $W_L, W_R \in \Omega \times \mathbb{R}$
    \begin{eqnarray*}
& & A_{\widetilde{\Psi}}(W_L, W_R) (U_R - U_L) = \int_0^1 A(\Psi_U(s; W_L, W_R) ) \partial_s \Psi_U(s; W_L, W_R)\, ds ;\\
& & S_{\widetilde{\Psi}}(W_L, W_R) (H_R - H_L)  =
 \int_0^1 S(\Psi_U(s; W_L, W_R) ) \partial_s \Psi_H(s; W_L, W_r)\, ds.\end{eqnarray*}
\end{itemize}
In this paragraph, the Upwind splitting is considered. Some algebraic computations show that the corresponding splitting is given by the matrices
$$
 \mathcal{A}^\pm_{\widetilde{\Psi}}(W_L, W_R) = \left[ \begin{array}{c|c} P^\pm_{\widetilde{\Psi}}(W_L,W_R)A_{\widetilde \Psi}(W_L,W_R) & -P^\pm_{\widetilde{\Psi}}(W_L,W_R)S_{\widetilde{\Psi}}(W_L,W_R) \\\hline 0 & 0 \end{array}\right]
 $$
 where $P^\pm_{\widetilde{\Psi}}$ are the projection matrices defined as in \eqref{eq:Ppm}.
 
If the trivial equation for the artificial unknown $H$ is removed, the numerical method can be written again as  \eqref{eq:numerical_sch}-\eqref{eq:D-}-\eqref{eq:D+} where now
\begin{equation}\label{eq:Djk_Roe_St} 
D^\pm_{j,k} = P_{j,k}^\pm \left( A_{j,k} (U_k - U_j) - S_{j,k} (H(x_k) - H(x_j)) \right),
\end{equation}
with
$$
 P_{j,k}^\pm  = P^\pm_{\widetilde{\Psi}}(W_j, W_k), \quad A_{j,k}  =A_{\widetilde{\Psi}}(W_j, W_k), \quad S_{j,k}  = S_{\widetilde{\Psi}}(W_j, W_k). $$
 Accordingly, the fluctuations can be computed as follows:
given two indices $j,k$, first the coordinates $\{ \alpha_{j,k;l} \}_{l=1}^N$ of $U_k - U_j - A_{j,k}^{-1} S_{j,k} (H(x_k) - H(x_j))$ in the basis of eigenvectors of the Roe matrix $A_{j,k}$ are computed:
\begin{equation}\label{eq:sist_alpha}
U_k - U_j - A_{j,k}^{-1} S_{j,k} (H(x_k) - H(x_j)) = \sum_{l=1}^N \alpha_{j,k;l}\vec r_{j,k; l}.
\end{equation}
Then, one has:
$$
A_{j,k}(U_k - U_j) - S_{j,k} (H(x_k) - H(x_j)) = \sum_{i=1}^N  \alpha_{j,k;l}\lambda_{j,k;l}\vec r_{j,k; l},$$
and then
$$
D_{j,k}^\pm =  \sum_{i=1}^N  \alpha_{j,k;l}\lambda^\pm_{j,k;l}\vec r_{j,k; l}.$$
The reconstruction is then performed as in the case of homogeneous problems. 

Proposition \ref{prop:wb_St} can be then applied to this particular case to show that, given $H(x)$,  the numerical method is well-balanced for a stationary solution $U^*$ provided that \eqref{eq:Apsiwb} holds, i.e. if
\begin{equation}
A_{\widetilde\Psi}(U^*(x_L), U^*(x_R)) (U^*(x_R) - U^*(x_L)) = 
S_{\widetilde\Psi}(U^*(x_L), U^*(x_R)) (H(x_R) - H(x_L)) 
\end{equation}
for all  $x_L < x_R$. This is the case if
\begin{equation}\label{eq:widetildepsi}
s \in [0,1] \to \widetilde \Psi \left(s; \left[ \begin{array}{c} U^*(x_L) \\H(x_L) \end{array} \right]; \left[ \begin{array}{c} U^*(x_R) \\H(x_R) \end{array} \right] \right)
\end{equation}
and
\begin{equation}\label{eq:UstarHstar}
x \in [x_L, x_R] \to \left[\begin{array}{c} U^*(x) \\ H(x) \end{array}\right]
\end{equation}
define the same curve for all $x_L < x_R$. In particular, if the family of straight segments is chosen, then the numerical method is well-balanced for every stationary solution such that \eqref{eq:UstarHstar} lies in a straight line for every $x_L < x_R$. This property will be used in Section \ref{Sec:two_swe} to define numerical methods that preserve water-at-rest solutions for the two-layer shallow-water system. More sophisticated families of paths could be considered to preserve more general stationary solutions, as the ones based on the Generalized Hydrostatic Reconstruction introduced in \cite{Castro20072055} which will be done in a forthcoming paper.

The numerical treatment of the source term in this strategy can be interpreted as it was done in Section \ref{ss:nummeth_noncon} for the nonconservative products $B(U)U_x$: the source term is first written as the derivative of a new flux function
$$
S(U)H_x = \partial_x \mathcal{S}^L_i$$
with
$$
\mathcal{S}^L_i(x,t) = \int_0^1 S(\Psi_U(s; W(x_i,t), W(x,t) ) \partial_s \Psi_H(s; W(x_i,t), W(x,t))\, ds,$$
while in the global-flux approach, it is rewritten as
$$
S(U)H_x = \partial_x \mathcal{S}^G$$
with
$$
\mathcal{S}^G(x,t) = \int_{x_0}^x S(U)H_x \, dx.$$
Again,  $\mathcal{S}^L_i$ is approximated at the node points of the stencil as follows
$$
\mathcal{S}^L_i(x_j,t) \approx S_{\widetilde{\Psi}}(W_i, W_j) (H(x_j) - H(x_i)), \quad j \in \mathcal{S}_i, $$
and no integrals at the cells have to be approximated thus avoiding the need to calculate new  reconstructions at the quadrature points.


\subsubsection{Strategy 2 }
\label{ss:wb_s2}
Strategy 2 extends to the nonconservative system the technique proposed in \cite{PARES2021109880} for systems of balance laws. Unlike Strategy 1 the well-balanced property of the methods based on this strategy will not depend on either the choice of the family of paths or the matrix splitting. 

Let us consider first the numerical method for \eqref{eq:noncon+st} given by 
\begin{equation}
\frac{dU_i}{dt} + \frac{1}{\Delta x} \left( \widehat{D}^-_{i + 1/2} +  \widehat{D}^+_{i - 1/2}\right) = S(U_i)H_x(x_i),
\label{eq:numerical_sch_nowb}
\end{equation}
where the fluctuations $\widehat{D}^+_{i - 1/2}$ are defined by 
\eqref{eq:Djk_Roe} with any choice of family of paths, Roe matrix and matrix-splitting. Under the hypothesis of Proposition \ref{prop:wb_St}, this method is highly accurate but in principle does not preserve any stationary solution.  Let us modify the method so that a
given $m$-parameter family  of stationary solutions
\begin{equation}\label{eq:mparss}
U^*(x; c_1, \dots, c_m),
\end{equation}
with $m \leq N$, is preserved. To do this, let us assume that there exists a $m \times N$ matrix $C$ with rank $m$ such that, given any point $\bar x$ and any state $\bar U$, there exists a unique stationary solution of the family satisfying
\begin{equation}\label{eq:CUstar=CU}
C U^*(\bar x)  =  C \bar U,
\end{equation}
i.e., this system of equations determines the value of the $m$ parameters.
The idea is then to rewrite the equation at $x_i$ equivalently as follows:
$$
\frac{dU_i}{dt} + A(U)U_x - A(U^*_i)U^*_{i,x} = (S(U(x_i,t))-S(U^*_i(x_i,t)))H_x(x_i),$$
where $U^*_i$ is the unique stationary solution of the $m$-parameter family satisfying
\begin{equation}\label{eq:wbst2}
C U_i^*(x_i) =  C U(x_i,t).
\end{equation}
The idea is then to discretize $A(U)U_x$ and $A(U^*_i)U^*_{i,x}$ together by applying the strategy introduced in Section \ref{ss:nummeth_noncon} what leads to the numerical method:
\begin{equation}
\frac{dU_i}{dt} + \frac{1}{\Delta x} \left( \widehat{D}^-_{i + 1/2} +  \widehat{D}^+_{i - 1/2}\right) = (S(U_i) - S(U^*_i(x_i)))H_x(x_i),
\label{eq:numerical_sch_st2}
\end{equation}
with
\begin{eqnarray}
\widehat{D}^-_{i + 1/2} & = &  \mathcal{R}^L(D^+_{i,i-k} - D^{*,+}_{i,i-k} , \dots, D^+_{i,i+k} - D^{*,+}_{i,i+k}) \nonumber \\
& & + \mathcal{R}^R(D^-_{i, i+1-k} - D^{*,-}_{i, i+1-k}, \dots, D^-_{i,i+1+k} - D^{*,-}_{i,i+1+k}) , \label{eq:D-wb}\\
\widehat{D}^+_{i - 1/2} & = &  \mathcal{R}^L(D^+_{i-1-k,i} - D^{*,+}_{i-1-k,i}, \dots, D^+_{i-1 +k,i} - D^{*,+}_{i-1 +k,i} ) \nonumber\\
& & + \mathcal{R}^R(D^-_{i-k,i} - D^{*,-}_{i-k,i} , \dots, D^-_{i+k,i} - D^{*,-}_{i+k,i}) , \label{eq:D+wb}
\end{eqnarray}
where the starred fluctuations are given by
\begin{equation}\label{eq:D*jkRpe}
D^{*,\pm}_{j,k} = A^\pm_\Psi (U_i^*(x_j), U_i^*(x_k) ) ( U^*_i(x_k) - U_i^*(x_j)),
\end{equation}
The following result then holds.
\begin{proposition} The numerical method \eqref{eq:numerical_sch_st2}-\eqref{eq:D-wb}-\eqref{eq:D+wb} is well-balanced for all the stationary solutions of the family \eqref{eq:mparss}, i.e. $\{ U^* (x_i;c_1, \dots, c_m) \}$ is an equilibrium of the ODE system \eqref{eq:numerical_sch} for every stationary solution $U^*$.
\end{proposition}
The proof is straightforward. As an application, let us consider the family of stationary solutions
$$U^*(x;c) = \left[ \begin{array}{c}u^*(x; c) \\
v^*(x;c) \end{array} \right]  = \left[ \begin{array}{c}\frac{c}{2} + \sin(x) \\
\frac{c}{2} - \sin(x) \end{array} \right], \quad c \in \mathbb{R}$$
of System \eqref{CBs}. Given $\bar x$ and $\bar U = [\bar u, \bar v]^T$ the equation
$$
u^*(x;c) + v^*(x;c)  = \bar u + \bar v
$$
determines the value of the parameter
$$ c = \bar u + \bar v,$$
i.e. $m = 1$ and $C = [1 ,1]$ in this case.
Therefore, if the family of straight segments and the Roe matrix $A_{\Psi_1}$ in \eqref{eq:CBs_Roe} are chosen, the numerical method \eqref{eq:numerical_sch_st2} with
$$
U^*_i(x) = \left[ \begin{array}{c}\frac{u_i + v_i}{2} + \sin(x) \\
\frac{u_i + v_i}{2} - \sin(x) \end{array} \right]$$
is high-order accurate and well-balanced  for the given family of stationary solution. This technique will be used in next section to design again a numerical method that preserves water-at-rest solutions for the two-layer shallow-water system that is based on the LF splitting.

 For some particular problems, this strategy can be extended to design fully well-balanced methods, i.e. methods that preserve all the stationary solutions. In effect, let us suppose that, given $\bar x$ and $\bar U$, there is only one stationary solution $U^*(x)$ such that
 $$
 U^*(\bar x) = \bar U$$
(i.e. $m =N$ and $C$ is the identity matrix) or, there are several but it is possible to use a criterion to select one of them (like the flow regime, for instance). Then, the strategy can be applied by taking $U^*_i$ as the unique or the selected stationary solution such that
 $$
 U^*_i(x_i) = U(x_i, t).$$
 In this case, the numerical method reduces to \eqref{eq:numerical_sch}, with the fluctuations given by \eqref{eq:D-wb}-\eqref{eq:D+wb}.
 
 In particular,  this technique allows the design of fully well-balanced numerical methods for the shallow-water system: see \cite{PARES2021109880}. In fact, the numerical methods introduced here based on Strategy 2 reduce to the ones in this reference when they are applied to systems of balance laws, as the shallow-water system. In the reference, it has been shown that theses methods deal correctly with discontinuous bottom functions $H$: the fully-well balanced methods are able to correctly capture the stationary contact discontinuities standing on the points of discontinuity of $H$: see \cite{PARES2021109880} for details. This technique can be extended to derive numerical methods that are fully well-balanced for the 1D two-layer shallow-water systems but the computation of moving stationary solution is more involved: this will be the object of future work.
\subsection{Implementation}\label{subsec:Imple}
We summarize here the implementation of two well-balanced methods: Method 1 in which Strategy 1 is combined with the Upwind splitting, and Method 2 corresponding to Strategy  2, LF splitting, and reconstruction in characteristic variables. Let us assume that approximations $U_j$ to the solution are available  at the cell points; 
the fluctuations and the source terms at a point $x_i$ are then computed as follows;
\begin{enumerate}
    \item Compute the Roe matrices $A_{i,j} = A_{\Psi}(U_i, U_j)$, $j = i-k, \dots, i+k$ and their eigenvalues  $\{ \lambda_{i,j;l} \}_{l=1}^N$.
    \item Compute the fluctuations at the stencil points as follows:
    \begin{enumerate}
           \item  Method 1:
        \\
       Compute the eigenvector matrices $R_{i,j}$, $j = i-k, \dots, i+k$ .\\
    Solve the linear systems
    $$R_{i,j}\vec \alpha = U_j - U_i - A_{i,j}^{-1} S_{i,j} (H(x_j) - H(x_i))$$
    to obtain $\{ \alpha_{i,j;l} \}_{l=1}^N$ such that
    $$U_j - U_i - A_{i,j}^{-1} S_{i,j} (H(x_j) - H(x_i)) = \sum_{l=1}^N \alpha_{i,j;l}\vec r_{i,j; l},$$
       Compute then $$D_{i,j}^\pm =  \sum_{i=1}^N  \alpha_{i,j;l}\lambda^\pm_{i,j;l}\vec r_{i,j; l}.$$
        \item Method 2:\\ 
        Compute first the  stationary solution $U^*_i$ of the family
        \eqref{eq:mparss} satisfying 
        $$CU^*_i(x_i) = C U_i$$
        and evaluate it at $x_j$, $j = i-k, \dots,i+k$.\\
        Compute the Roe matrices $A^*_{i,j} = A_{\Psi}(U_i^*(x_i), 
        U^*_i(x_j))$.\\
        Compute then  
 \begin{equation*}
 \begin{aligned}
D^\pm_{i,j} &= \dfrac{1}{2}\left(A_{i,j}\left(U_j - U_i\right)\pm \alpha \left(U_j - U_i\right)\right),\\
 D^{*,\pm}_{i,j} &= \dfrac{1}{2}\left(A^*_{i,j}\left(U_i^*(x_j)  - U_i^*(x_i) \right)\pm \alpha \left(U_i^*(x_j)  - U_i^*(x_i)\right)\right),
 \end{aligned}
 \end{equation*}
 where $\alpha$ is chosen so that
 $\alpha > |\lambda_{i,j;l}|$ for all $i,j,l$.
     \end{enumerate}
    \item Compute the WENO reconstructions of the fluctuations at the cell interfaces as follows:
    \begin{enumerate}
    \item  Method 1:
\begin{eqnarray*}
\widehat{D}^-_{i + 1/2} & = &  \mathcal{R}^L(D^+_{i,i-k}, \dots, D^+_{i,i+k} ) + \mathcal{R}^R(D^-_{i, i+1-k}, \dots, D^-_{i,i+1+k}) , \\
\widehat{D}^+_{i - 1/2} & = &  \mathcal{R}^L(D^+_{i-1-k,i}, \dots, D^+_{i-1 +k,i} ) + \mathcal{R}^R(D^-_{i-k,i}, \dots, D^-_{i+k,i}).
\end{eqnarray*}
        \item Method 2:\\
        Compute the matrices of right and left eigenvectors, $R_{i-1,i}$, $L_{i-1,i}$, (resp. $R_{i,i+1}$, $L_{i,i+1}$)
        of $A_{i-1,i}$ (resp. $A_{i,i+1}$)\\
        Then compute:
    \begin{eqnarray*}
\widehat{D}^-_{i + 1/2} & = & R_{i,i+1} \mathcal{R}^L(L_{i,i+1}\left(D^+_{i,i-k} - D^{*,+}_{i,i-k}\right) , \dots,L_{i,i+1} \left(D^+_{i,i+k} - D^{*,+}_{i,i+k}\right)) \nonumber \\
& & + R_{i,i+1}\mathcal{R}^R(L_{i,i+1}\left(D^-_{i, i+1-k} - D^{*,-}_{i, i+1-k}\right), \dots, L_{i,i+1}\left(D^-_{i,i+1+k} - D^{*,-}_{i,i+1+k}\right)) , \\
\widehat{D}^+_{i - 1/2} & = &  R_{i-1,i}\mathcal{R}^L(L_{i-1,i}\left(D^+_{i-1-k,i} - D^{*,+}_{i-1-k,i}\right), \dots, L_{i-1,i}\left(D^+_{i-1 +k,i} - D^{*,+}_{i-1 +k,i}\right) ) \nonumber\\
& & + R_{i-1,i}\mathcal{R}^R(L_{i-1,i}\left(D^-_{i-k,i} - D^{*,-}_{i-k,i}\right) , \dots, L_{i-1,i}\left(D^-_{i+k,i} - D^{*,-}_{i+k,i}\right)). 
\end{eqnarray*}\\
Compute finally the source term:
$$
\left(S(U_i) - S(U^*_i(x_i))\right)H_x(x_i).$$
\end{enumerate}
\end{enumerate}
\begin{remark}
Although in Section \ref{ss:nummeth_noncon} it was said that the numerical methods based on the Upwind splitting are more computationally expensive than those based on the LF splitting, observe that the numerical treatment of the source term in Method 2 requires the computation of the fluctuations corresponding to the chosen stationary solutions so that, as it will be seen in Section \ref{sec:numerical_sol}, the computational costs of both methods are comparable.
\end{remark}

\section{Extension to 2D systems}
\label{sec:nuemrical_2D}
\subsection{Homogeneous problems}
\label{sec:nuemrical_2D_general}
This section extends the 1D path-conservative fifth-order WENO scheme to 2D nonconservative systems of the form
\begin{equation}\label{eq:2dnoncons}
	U_t + A_1(U) U_x + A_2 (U) U_y = 0.
 \end{equation}
 The system is supposed again to be strictly hyperbolic, i.e. for all $U$ and all $\theta \in [0, 2\pi)$, the matrix
 $$
 \cos(\theta) A_1 (U) + \sin(\theta) A_2(U)$$
 has $N$ different real eigenvalues. Systems with flux terms and nonconservative products 
\begin{equation}
\label{eq:general_cf2d}
U_t + F(U)_x + G(U)_y + C(U)U_x + D(U)U_y = 0,
\end{equation}
can be considered as particular cases in which
$A_1(U)=J(F(U)) +C(U), A_2(U)=J(G(U)) +D(U).$

Let us assume again that Roe linearizations $A_{i, \Psi}(U,V)$ of $A_i(U)$, $i = 1,2$ are available for the selected family of paths
$\Psi$. We consider uniform Cartesian meshes with points $(x_i, y_j)$ with steps $\Delta x$ and $\Delta y$ in the $x$ and $y$ direction.
WENO methods can be extended dimension by dimension:
\begin{equation}
 U^{'}_{i,j} = -\frac{1}{\Delta x} \left( \widehat{D}^-_{i + 1/2,j} + \widehat{D}^+_{i - 1/2,j}\right)-\frac{1}{\Delta y} \left( \widehat{D}^-_{i ,j+ 1/2} + \widehat{D}^+_{i,j - 1/2}\right),\label{eq:semi_2d}   
\end{equation}
where $U_{i,j} \approx U(x_i, y_j)$ and
\begin{eqnarray}
\widehat{D}^-_{i + 1/2,j} & = &  \mathcal{R}^L(D^+_{i,i-k;j}, \dots, D^+_{i,i+k;j} ) + \mathcal{R}^R(D^-_{i, i+1-k;j}, \dots, D^-_{i,i+1+k;j}),\\
\widehat{D}^+_{i - 1/2,j} & = &  \mathcal{R}^L(D^+_{i-1-k,i;j}, \dots, D^+_{i-1 +k,i;j} ) + \mathcal{R}^R(D^-_{i-k,i;j}, \dots, D^-_{i+k,i;j}) ,\\
\widehat{D}^-_{i,j + 1/2} & = &  \mathcal{R}^L(D^+_{i;j,j-k}, \dots, D^+_{i;j,j+k} ) + \mathcal{R}^R(D^-_{i; j,j+1-k}, \dots, D^-_{i;j,j+1+k}),\\
\widehat{D}^+_{i,j - 1/2} & = &  \mathcal{R}^L(D^+_{i;j-1-k,j}, \dots, D^+_{i;j-1 +k,j} ) + \mathcal{R}^R(D^-_{i;j-k,j}, \dots, D^-_{i;j+k,j}) . 
\end{eqnarray}
Here, the following notation has been used:
\begin{eqnarray} 
D^\pm_{i,l;j} &=& A^\pm_{1;i,l;j} (U_{l,j} - U_{i,j}) ,\quad l=i-k, \dots, i+k,\\
D^\pm_{i;j,l} &=& A^\pm_{2;i;j,l} (U_{i,l} - U_{i,j}), \quad l=j-k, \dots, j+k,
\end{eqnarray}
where 
$$
A_{1; i,l;j} = A_{1,\Psi}(U_{i,j}, U_{l,j}), \quad 
A_{2; i;j,l} = A_{2,\Psi}(U_{i,j}, U_{i,l})
$$
and the super-indices $\pm$ represent their splitting matrices: both the Upwind and the LF splittings can be readily extended to 2D.


Proposition \ref{prop:order} can be easily extended to prove that \eqref{eq:semi_2d} is a numerical method of order $p = 2k +1$ 
provided that Property  \eqref{eq:path_cons} is satisfied for $A_1$ and $A_2$.
\subsection{Problems with source terms}
\label{sec:numerical_2D_wb}
We now consider systems of the form
\begin{equation}\label{eq:2dnoncons_st}
	U_t + A_1(U) U_x + A_2 (U) U_y = S_1(U)H_x + S_2(U)H_y,
 \end{equation}
 where $H(x,y)$ is again a known function. Strategy 1 described in Section \ref{sec:well-balance} can be readily extended to problem \eqref{eq:2dnoncons_st}: it is enough to redefine the fluctuations as follows:
\begin{eqnarray}
D^\pm_{i,l;j} & = & A^\pm_{1;i,l;j} (U_{l,j} - U_{i,j}) - S^\pm_{1;i,l;j} (H(x_l,y_j) - H(x_i,y_j)) ,\\
D^\pm_{i;j,l} & = & A^\pm_{2;i;j,l} (U_{i,l} - U_{i,j}) - S^\pm_{2;i;j,l} (H(x_i, y_l) - H(x_i, y_j)) ,
\end{eqnarray}
where $A_{1;i,k;j}$, $S_{1; i,k;j} $, represent respectively the intermediate matrix and source term given by the Roe linearization for the states $U_{i,j}$ and $U_{l,j}$
and  $A_{2;i;j,l}$, $S_{2; i; j,l} $ the corresponding to the states $U_{i,j}$, $U_{i,l}$. 

It can be shown as in Proposition \ref{prop:wb_St} that the numerical method is well-balanced for stationary solutions that satisfy that
 given $(x_L, y), (x_R,y), (x, y_D), (x, y_U)  \in \mathbb{R}^2$ one has 
\begin{eqnarray}\label{eq:spl_wb_2d_1}
& & A_{1,\widetilde\Psi}(U^*(x_L,y), U^*(x_R,y)) (U^*(x_R,y) - U^*(x_L,y))\\\nonumber
& &\qquad \qquad  = 
S_{1,\widetilde\Psi}(U^*(x_L,y), U^*(x_R,y)) (H(x_R,y) - H(x_L,y)), \\
\label{eq:spl_wb_2d_2}
& & A_{2,\widetilde\Psi}(U^*(x,y_D), U^*(x,y_U)) (U^*(x,y_U) - U^*(x,y_D)) \\\nonumber
& & \qquad \qquad= 
S_{2,\widetilde\Psi}(U^*(x,y_D), U^*(x,y_U)) (H(x,y_U) - H(x,y_D))
\end{eqnarray} 
These equalities are satisfied if the functions
\begin{equation}\label{eq:pathx}
s \in [0,1] \to \Psi_U \left(s; \left[ \begin{array}{c} U^*(x_L,y) \\H(x_L,y) \end{array} \right]; \left[ \begin{array}{c} U^*(x_R,y) \\H(x_R,y) \end{array} \right] \right),
\end{equation}
\begin{equation}\label{eq:pathy}
s \in [0,1] \to \Psi_U \left(s; \left[ \begin{array}{c} U^*(x,y_D) \\H(x,y_D) \end{array} \right]; \left[ \begin{array}{c} U^*(x,y_U) \\H(x,y_U) \end{array} \right] \right),
\end{equation}
define respectively the same curves in $\Omega$ as
\begin{equation}\label{eq:ustarx}
x \in [x_L, x_R] \to U^*(x,y),\end{equation} 
\begin{equation}\label{eq:ustary}
y \in [y_D, y_U] \to U^*(x,y).\end{equation}
This geometrical property is much more restrictive in 2D than in 1D and, in general, only stationary solutions that are essentially 1D or some particular families of stationary solutions satisfy them, as will be seen in the two-layer shallow-water case. 

Strategy 2 can be extended easily  to 2D problems to design numerical methods that preserve a given family of $m$-parameter stationary solutions
$$
U^*(x,y;  c_1, \dots, c_m)$$
with $m < N$, assuming again that, given $\bar x, \bar y, \bar U$, 
there exists a unique choice of the parameters such that
$$
CU^*_i(\bar x, \bar y) = C \bar U,
$$
Once an element of the family $U^*_{i,j}$ has been determined, the numerical solution is written as follows:
\begin{eqnarray}
  & & U^{'}_{i,j} + \frac{1}{\Delta x} \left( \widehat{D}^-_{i + 1/2,j} + \widehat{D}^+_{i - 1/2,j}\right) + \frac{1}{\Delta y} \left( \widehat{D}^-_{i ,j+ 1/2} + \widehat{D}^+_{i,j - 1/2}\right)\\ \nonumber
  & & \qquad \qquad = (S_1(U_{i,j}) - S_1(U^*_{i,j}(x_i, y_j)))H_x(x_i, y_j) +(S_2(U_{i,j}) - S_2(U^*_{i,j}(x_i, y_j)))H_y(x_i, y_j)  ,\label{eq:semi_2d_st2}   
\end{eqnarray}
with
\begin{eqnarray}
\widehat{D}^-_{i + 1/2,j} & = &  \mathcal{R}^L(D^+_{i,i-k;j} - D^{*,+}_{i,i-k;j} , \dots, D^+_{i,i+k;j} - D^{*,+}_{i,i+k;j}) \nonumber \\
& & + \mathcal{R}^R(D^-_{i, i+1-k;j} - D^{*,-}_{i, i+1-k;j}, \dots, D^-_{i,i+1+k;j} - D^{*,-}_{i,i+1+k;j}) ,\\
\widehat{D}^+_{i - 1/2,j} & = &  \mathcal{R}^L(D^+_{i-1-k,i;j} - D^{*,+}_{i-1-k,i;j}, \dots, D^+_{i-1 +k,i;j} - D^{*,+}_{i-1 +k,i;j} ) \nonumber\\
& & + \mathcal{R}^R(D^-_{i-k,i;j} - D^{*,-}_{i-k,i;j} , \dots, D^-_{i+k,i;j} - D^{*,-}_{i+k,i;j}) ,\\
\widehat{D}^-_{i,j + 1/2} & = &  \mathcal{R}^L(D^+_{i;j,j-k} - D^{*,+}_{i;j,j-k} , \dots, D^+_{i;j,j+k} - D^{*,+}_{i;j,j+k}) \nonumber \\
& & + \mathcal{R}^R(D^-_{i;j,j+1-k} - D^{*,-}_{i;j,j+1-k}, \dots, D^-_{i;j,j+1+k} - D^{*,-}_{i;j,j+1+k}) ,\\
\widehat{D}^+_{i ,j- 1/2} & = &  \mathcal{R}^L(D^+_{i;j-1-k,j} - D^{*,+}_{i;j-1-k,j}, \dots, D^+_{i;j-1 +k,j} - D^{*,+}_{i;j-1 +k,j} ) \nonumber\\
& & + \mathcal{R}^R(D^-_{i;j-k,j} - D^{*,-}_{i;j-k,j} , \dots, D^-_{i;j+k,j} - D^{*,-}_{i;j+k,j}) .
\end{eqnarray}
Here, as in the 1D case, in the fluctuations $D^{*, \pm}_{k,l;m}$, $U_{k,m}$ and $U_{l,m}$ are replaced by $U^*_{i,j}(x_k, y_m)$ and $U^*_{i,j}(x_l, y_m)$, and in the fluctuations
$D^{*, \pm}_{k; l,m}$, $U_{k,l}$ and $U_{k,m}$ are replaced by $U^*_{i,j}(x_k, y_l)$ and $U^*_{i,j}(x_k, y_m)$.


\section{Application to the two-layer shallow water model}
\label{Sec:two_swe}
In this Section, we apply the proposed well-balanced schemes in Section \ref{sec:nuemrical_2D} to 2D two-layer shallow water system  that governs the flow of two superposed layers of immiscible fluids with different constant densities:
\begin{equation}
\begin{aligned}
& \left(h_1\right)_t+\left(h_1u_{1,1}\right)_x+\left(h_1u_{1,2}\right)_y=0, \\
& \left(h_1u_{1,1}\right)_t+\left(h_1 u_{1,1}^2+\frac{g}{2} h_1^2\right)_x+\left(h_1 u_{1,1} u_{1,2}\right)_y=-g h_1 Z_x-g h_1\left(h_2\right)_x, \\
& \left(h_1u_{1,2}\right)_t+\left(h_1 u_{1,1} u_{1,2}\right)_x+\left(h_1 u_{1,2}^2+\frac{g}{2} h_1^2\right)_y=-g h_1 Z_y-g h_1\left(h_2\right)_y, \\
& \left(h_2\right)_t+\left(h_2u_{2,1}\right)_x+\left(h_2u_{2,2}\right)_y=0, \\
& \left(h_2u_{2,1}\right)_t+\left(h_2 u_{2,1}^2+\frac{g}{2} h_2^2\right)_x+\left(h_2 u_{2,1} u_{2,2}\right)_y=-g h_2 Z_x-g r h_2\left(h_1\right)_x, \\
& \left(h_2u_{2,2}\right)_t+\left(h_2 u_{2,1} u_{2,2}\right)_x+\left(h_2 u_{2,2}^2+\frac{g}{2} h_2^2\right)_y=-g h_2 Z_y-g r h_2\left(h_1\right)_y,
\label{MODEL:TL-SWEs}
\end{aligned}
\end{equation}
where
\begin{itemize}
	\item $h_k$, $k=1,2$ is the thickness of the $k$th layer;
	\item $\vec u _k = (u_{k,1}, u_{k,2})$, $k = 1,2$ is the velocity of the $k$th layer;
	\item $\displaystyle r = \frac{\rho_1}{\rho_2}$, where $\rho_k$, $k = 1,2$ is the density of the $k$th layer ($\rho_1 < \rho_2$).  (index 1 corresponds to the upper layer);
 	\item $c_k = \sqrt{g h_k}$, $k = 1,2$;
  \item $Z(x,y)$ is the bottom function.
\end{itemize}
The equations can be written in the form \eqref{eq:2dnoncons_st} with $N = 6$, $H \equiv - Z$,
$$ U = \left[ \begin{array}{c} h_1, h_1 u_{1,1}, h_1 u_{1,2},  h_2, h_2 u_{2,1}, h_2 u_{2,2} \end{array} \right]^T,$$
\begin{equation}
\begin{aligned}
    A_1(U) &=
	\left[
	\begin{array}{cccccc}
		0 & 1 & 0 & 0 & 0 & 0 \\
		c_1^2 - u_{1,1}^2 & 2 u_{1,1} & 0 & c_1^2 & 0 & 0 \\
		-u_{1,1}u_{1,2} & u_{1,2} & u_{1,1} & 0 & 0 & 0 \\
		0 & 0 & 0 & 0 & 1 & 0 \\
		rc_2^2 & 0 & 0 & c_2^2 - u_{2,1}^2 & 2 u_{2,1} & 0 \\
		0 & 0 & 0 & -u_{2,1}u_{2,2} & u_{2,2} & u_{2,1}
		\end{array}
	\right],\quad S_1(U) = \left[ \begin{array}{c}  0 \\ gh_1 \\ gh_1 \\ 0 \\ 0 \\ 0\end{array} \right], \\
\label{eq:A1_2d}
A_2(U) &=
	\left[
	\begin{array}{cccccc}
		0 & 0 & 1 & 0 & 0 & 0 \\
		-u_{1,1}u_{1,2} & u_{1,2} & u_{1,1} & 0 & 0 & 0 \\
		c_1^2 - u_{1,2}^2 & 0 & 2 u_{1,2} & c_1^2 & 0 & 0 \\
		0 & 0 & 0 & 0 & 0 & 1 \\
		0 & 0 & 0 & -u_{2,1}u_{2,2} & u_{2,2} & u_{2,1}\\
		rc_2^2 & 0 & 0 & c_2^2 - u_{2,2}^2 & 0& 2 u_{2,2}
		\end{array}
	\right],\quad S_2(U) = \left[ \begin{array}{c}  0 \\ 0 \\0 \\ 0 \\ gh_2 \\ gh_2 \end{array} \right]
\end{aligned}
\end{equation}

The characteristic equation of $A_1(U)$ is
$$
(\lambda - u_{1,1})(\lambda - u_{2,1})\Bigl( \bigl( (\lambda - u_{1,1})^2 - gh_1 \bigr)
\bigl( (\lambda - u_{2,1})^2 - gh_2 \bigr) - r g^2 h_1 h_2\Bigr) = 0.
$$
The eigenvalues are then the four roots $\lambda_k$, $k = 1,..,4$ of the equation
$$
\bigl( (\lambda - u_{1,1})^2 - gh_1 \bigr)
\bigl( (\lambda - u_{2,1})^2 - gh_2 \bigr)  =   r g^2 h_1 h_2,
$$
and
$$
\lambda_5 = u_{1,1}, \quad \lambda_6 = u_{2,1}.
$$
The corresponding eigenvectors are
\begin{equation}\label{eq:eigenvectorsA1}
\vec R_k = \left[
\begin{array}{c}
	1 \\
	\lambda_k \\
	u_{1,2} \\
	\mu_k \\
	\mu_k \lambda_k \\
	\mu_k u_{2,2}
\end{array}\right], \quad k=1, \dots, 4,
\quad \vec R_5 = \left[ \begin{array}{c} 0 \\0 \\ 1 \\ 0 \\ 0 \\ 0 \end{array} \right], \quad
\vec R_6 = \left[ \begin{array}{c} 0 \\ 0 \\ 0 \\ 0 \\ 0 \\ 1 \end{array} \right], 
\end{equation}
with
$$
\mu_k = \frac{(\lambda_k - u_{1,1})^2}{c_1^2}-1.
$$
The eigenvalues and eigenvectors of $A_2(U)$ can be computed similarly.

For $r \approx 1$, first-order approximation of  $\lambda_1, \dots, \lambda_4$ is given similar as \cite{Schijf1953TheoreticalCO}:
\begin{equation}
 \lambda_{e x t}^{ \pm}=U_m \pm \sqrt{g\left(h_1+h_2\right)},\quad
      \lambda_{\text {int }}^{ \pm}=U_c \pm \sqrt{g^{\prime} \frac{h_1 h_2}{h_1+h_2}\left[1-\frac{\left(u_1-u_2\right)^2}{g^{\prime}\left(h_1+h_2\right)}\right]},
\label{eq:num_eig_values}
\end{equation}
where
$$
U_m=\frac{h_1 u_1+h_2 u_2}{h_1+h_2}, \quad U_c=\frac{h_1 u_2+h_2 u_1}{h_1+h_2},
$$
and $g' = \left(1-r\right)g$ is the reduced gravity.
Observe that $\lambda_{\text {int }}$ become complex when
$$
\frac{\left(u_1-u_2\right)^2}{g^{\prime}\left(h_1+h_2\right)}>1,
$$
so that the system is expected to lose hyperbolicity in these cases. Numerical techniques to overcome sporadic episodes of loss of hyperbolicity 
can be found in \cite{Castro-Diaz2011, KRVAVICA2018187}.

The 1D two-layer shallow-water model will be also considered: it can be written in the form \eqref{eq:noncon+st} with $N = 4$, $H \equiv -Z$,
\begin{equation}
U = \left[ \begin{array}{c} h_1 \\ h_1 u_1 \\ h_2 \\ h_2 u_2 \end{array} \right], \quad 
A(U)=\left[ \begin{array}{cccc}
0 & 1 & 0 & 0 \\
g h_1-u_1^2 & 2 u_1 & g h_1 & 0 \\
0 & 0 & 0 & 1 \\
r g h_2 & 0 & g h_2-u_2^2 & 2 u_2
\end{array}\right], \quad S = \left[
\begin{array}{c}
0 \\ gh_{1} \\ 0 \\ gh_{2}
\end{array}
\right],
\label{eq:matrix_A}
\end{equation}
where now $u_k$, $h_k$, $k = 1,2$ are respectively the velocity and thickness of the layers. The family of straight segments will be considered here to compute nonconservative products.
The following Roe linearization corresponding to the family of straight segments is available (see \cite{Fernandez-Nieto2011}):  
$$A(U_L,U_R) = \left[\begin{array}{ccccc}
0 & 1 & 0 & 0 \\
- \bar u_1^2 + \bar c_1^2 & 2 \bar u_1 & \bar c_1^2 & 0  \\
0 & 0 & 0 & 1  \\
r \bar c_2^2 & 0 & - \bar u_2^2+ \bar c_2^2 & 2 \bar u_2
\end{array}\right], \quad S(U_L,U_R) =  \left[\begin{array}{c} 0 \\ g \bar h_1 \\ 0 \\ g\bar h_2 \end{array}\right],
$$
where
$$
\bar u_k =\frac{\sqrt{h_{L,k}} u_{L, k} +\sqrt{h_{R,k}} u_{R, k}}{\sqrt{h_{L,k}}+\sqrt{h_{R, k}}},\quad \bar h_k = \frac{h_{L, k} +h_{R, k}}{2}, \quad \bar c_k =\sqrt{g \bar h_k }, \quad k=1,2.
$$
This Roe matrix and its natural extension to 2D will be used to implement WENO methods. 
Steady-states solutions that correspond to water-at-rest equilibria constitute a 2-parameter family:
$$u_{1,1}=u_{1,2}=u_{2,1}=u_{2,2} \equiv 0, \; h_1 = c_1, \; h_2 = - Z + c_2,$$
where $c_1$, $c_2$ are arbitrary parameters ($c_1 > 0$, $c_2 > \max(Z)$). We will focus here on methods that preserve this family.
Both Strategies 1 and 2 described in Section \ref{sec:well-balance} for 1D problems and \ref{sec:numerical_2D_wb} for 2D problems can be followed to design numerical methods that preserve water-at-rest solutions. In effect, for Strategy 1, the equalities \eqref{eq:spl_wb}, \eqref{eq:spl_wb_2d_1}, \eqref{eq:spl_wb_2d_2} can be easily checked for these stationary solutions: the equality of the curves given by  \eqref{eq:pathx} and \eqref{eq:ustarx} or those given by \eqref{eq:pathy} and \eqref{eq:ustary} can be easily checked if the family of straight segments is chosen: it derives from the linear nature of the relationships between variables that characterize water-at-rest solutions. 

For Strategy 2, observe that, given $\bar x,\bar y$ and a state 
$\bar U = [\bar h_1, \bar h_1\bar u_{1,1}, \bar h_1\bar u_{1,2},\bar h_2, \bar h_2 \bar u_{2,1}, \bar h_2\bar u_{2,2}]^T$, there exists a  unique water-at-rest stationary solution such as
$$
h_1^*(\bar x, \bar y) = \bar h_1, \quad h_2^*(\bar x, \bar y) = \bar h_2,$$
which is the one given by:
$$h_1^*(x,y) = \bar h_1,\quad  \bar h^*_2(x,y)= -Z(x,y) + Z(\bar x,\bar y) + \bar h_2, \quad u^*_{1,1} = u^*_{1,2} = u^*_{2,1} = u^*_{2,2} = 0,$$
i.e. in this case 
$$
C = \left[ \begin{array}{cccccc} 1 & 0 & 0 & 0 & 0 & 0 \\ 0 & 0 & 0 & 1 & 0 & 0 \end{array}\right],$$
so that the stationary solution $U^*_{i,j}$ used to implement the method is given by
\begin{equation}
    U^*_{i,j}(x,y) = \left[ h_{1;i,j}, 0 , 0 , -Z(x,y) + Z(x_i,y_j) + h_{2;i,j} , 0, 0 \right]^T.
\end{equation}

When the Upwind splitting scheme is used, systems of the form
\begin{equation*}
R\cdot \vec \alpha = U_R - U_L - A_l^{-1}S_k(H_R - H_L),
\end{equation*}
have to be solved to compute $\alpha_i$, $i = 1, \dots, 6$ such that \eqref{eq:sist_alpha} is satisfied. Here, $A_l = A_l(U_L, U_R)$, $S_l = S_l (U_L, U_R)$, $l =1,2$ represent the Roe linearization in the $x$ or $y$ direction,  and
$$
R = \left[ \begin{array}{c|c|c} \vec R_1 & \dots & \vec R_6 \end{array} \right], \quad
\vec{\alpha} = \left[ \begin{array}{c} \alpha_1 \\ \vdots \\ \alpha_6 \end{array} \right].
$$
Let us suppose that $l = 1$. In this case, $\vec R_i$, $i = 1, \dots, 6$ are given by \eqref{eq:eigenvectorsA1}. It can be checked that the system can be solved as follows:
\begin{enumerate}
	\item Solve system
 $$
 R \left[ \begin{array}{c} \alpha_1 \\ \alpha_2 \\\alpha_3 \\ \alpha_4 \end{array} \right] = V_R - V_L - A^{-1}S(H_R - H_L),
 $$
 where
 $$
 V_L = \left[ \begin{array}{c} h_{1,L} \\ h_{1,L} u_{1,1,L} \\ h_{2,L} \\ h_{2,L}u_{2,1,L} \end{array} \right], \quad V_R = \left[ \begin{array}{c} h_{1,R} \\ h_{1,R} u_{1,1,R} \\ h_{2,R} \\ h_{2,R}u_{2,1,R} \end{array} \right] , $$
	\begin{equation*} R = 
 	\left[ \begin{array}{ccc} 1, & \dots, & 1 \\
	   \lambda_1,  & \dots,  & \lambda_4 \\
       \mu_1, & \dots, & \mu_4 \\
       \mu_1\lambda_1, & \dots, & \mu_4 \lambda_4 \end{array} \right] ,
\quad
	A =
	\left[
	\begin{array}{cccccc}
		0 & 1  & 0 & 0  \\
		\bar c_1^2 - \bar u_{1,1}^2 & 2 \bar u_{1,1}  & \bar c_1^2 & 0 \\
		0 & 0 & 0  & 1  \\
		r\bar c_2^2 & 0 & \bar c_2^2 - \bar u_{2,1}^2 & 2 \bar u_{2,1}
		\end{array}
	\right], \quad 
 	S =
	\left[
	\begin{array}{c}
		0   \\
		g \bar h_1 \\
		0  \\
		g \bar h_2
		\end{array}
	\right].
\end{equation*}

	\item If $\lambda_i \not=0$, $i = 5,6$ compute
\begin{eqnarray*}
		\alpha_5 & = & \frac{F_3 - u_{1,2}\sum_{j=1}^4  \alpha_j}{\lambda_5}, \\
	    \alpha_6 & = & \frac{F_6 - u_{2,2}\sum_{j=1}^4  \mu_j \alpha_j}{\lambda_6}.
\end{eqnarray*}
Otherwise  (which is the case in water-at-rest stationary solutions) define
$$\alpha_5 = \alpha_6 = 0.$$

\end{enumerate}
In other words, the system to be solved in the 2D case reduces to those arising in 1D problems. Observe that, following this algorithm, the cases in which eigenvalues $\lambda_i$, $i=5,6$, vanish, and thus $A_k$ cannot be inverted, the coefficients $\alpha_j$ can be still computed.

\section{Numerical solutions}
\label{sec:numerical_sol}
In this section, we present 
some numerical results for 1D coupled Burgers,
the 1D and 2D two-layer shallow water equations.

Inheriting the notation from Section \ref{subsec:Imple}, two different WENO methods have been implemented and tested:
\begin{itemize}
\item WENO methods with Upwind splitting and Strategy 1 in Section \ref{ss:wb_s1} for the treatment of the source terms: these family of schemes will be called Method 1.

\item WENO methods with LF splitting and Strategy 2 in Section \ref{ss:wb_s2} for the treatment of the source terms: these family of schemes will be called Method 2.

\end{itemize}
For all the methods:
\begin{itemize}

\item fifth-order WENO is used in all of the numerical tests but in accuracy test; 

\item the third order TVD Runge-Kutta method in \cite{SHU1988439} is used for time stepping, which is a convex combination of forward Euler steps;

\item the computation of eigenvalues and eigenvectors are computed in analytic form according to \eqref{eq:num_eig_values} and \eqref{eq:eigenvectorsA1}, respectively, and the LAPACK library is used to solve linear systems. 

\item CFL = 0.45 is used in all cases.
\item the free boundary conditions are imposed except in the steady-state solutions and accuracy test.

\end{itemize}
The numerical experiments have been implemented using FORTRAN $90$ compiled with the INTEL ifort compiler run in a Dell Precision workstation with 24 CPU cores and 128 gigabytes of memory. 

In many numerical tests the results are comparable and, when this is the case, we only show the results obtained with Method 2.

\subsection{1D coupled Burgers Equation}
\label{num:bug}
In this section, we consider the coupled Burgers' system \eqref{CBs} and compare the numerical methods obtained with the methods based on the fluctuations \eqref{fluc1} and \eqref{fluc2} based respectively on the choice of the family of paths $\Psi_1$ given by \eqref{psi1} and $\Psi_2$ given by \eqref{psi2}.

\subsubsection{Smooth solution}
\label{sec:coupled_rarefaction}
The initial condition
\begin{equation}
        \left[u_0(x),v_0(x)\right]^T =
          \left[u_0(x),v_0(x)\right]^T = 
       \left[ 0.25 \sin(\pi x), 0.25 \sin(\pi x) \right]^T  
\end{equation}
is first consider in the interval $[0,1]$ and periodic boundary conditions are imposed. Figure \ref{Fig:compare_rare} shows the numerical solutions obtained at time $t=0.1$ with a mesh of 200 points, when the solution is still smooth. Reference solutions have been computed in a mesh of 6400 points and Table \ref{Tab:cb_Accuracy} shows the errors and the accuracy of the methods: as it can be seen, both methods seem to converge to the same solution, but only the one consistent with the family \eqref{psi1} achieves the optimal accuracy while the one consistent with the family \eqref{psi2} in only first-order accurate. We note that the high-order accuracy condition \eqref{eq:path_cons} for family of paths is then necessary.
\begin{figure}[H]
    \centering
\includegraphics[scale=0.4]{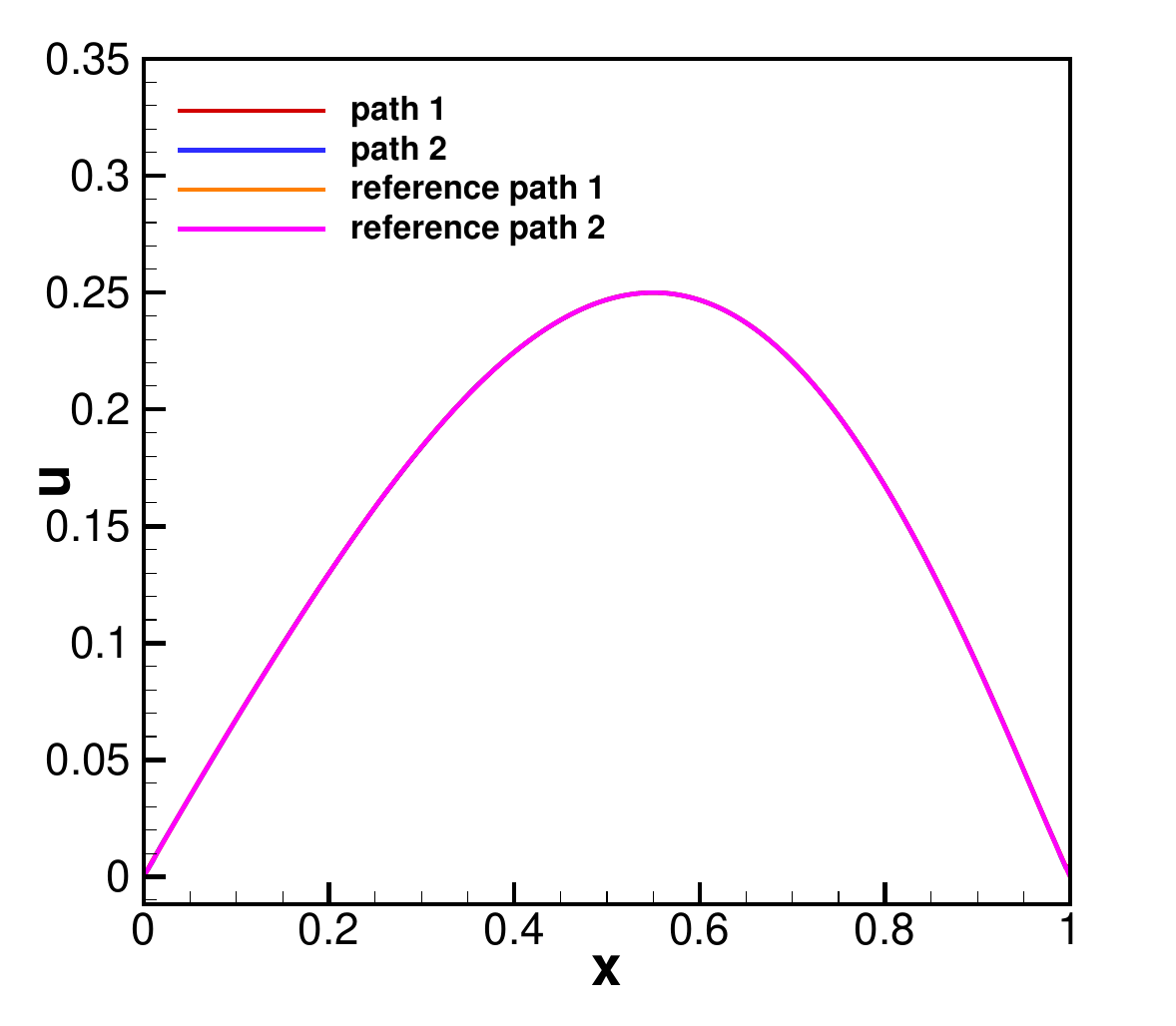}\qquad
\includegraphics[scale=0.4]{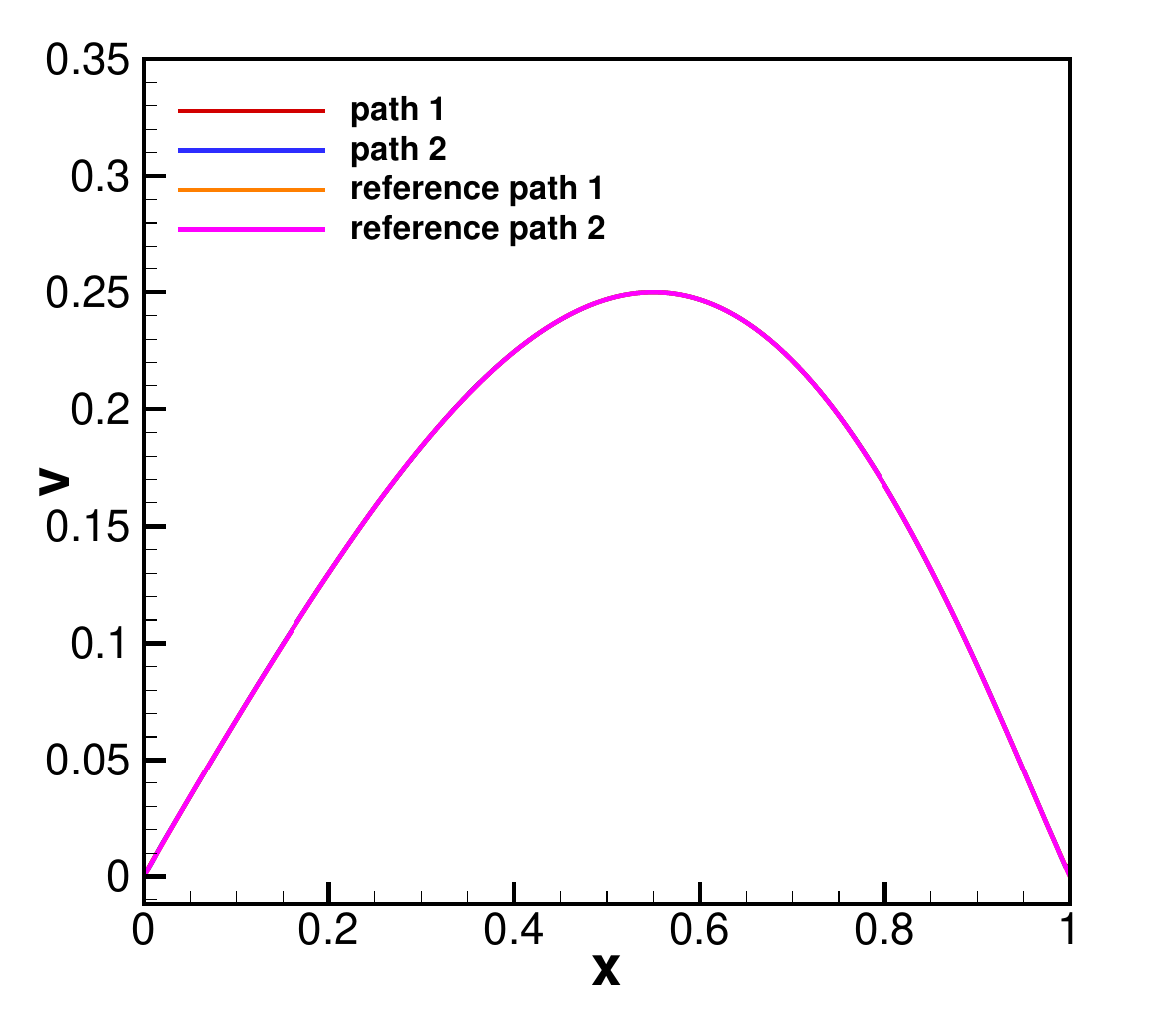} 
    \caption{ Section \ref{sec:coupled_rarefaction}:  smooth solution of the Coupled Burgers' system. Numerical solutions for $u$ (left) and $v$ (right) at $t=0.1$ obtained with the numerical methods consistent with the families of paths $\Psi_1$ and $\Psi_2$ using a mesh of 200 points.}
    \label{Fig:compare_rare}
\end{figure} 
\begin{table}[htbp]
\caption{Accuracy test with different paths, $L^\infty$ norm error for the coupled Burgers equations.}
\label{Tab:cb_Accuracy}
\begin{center}
\begin{tabular}{ccccc}\hline
N 
&\multicolumn{2}{c}{path 1} &\multicolumn{2}{c}{path 2}  \\
\cline{2-3}
\cline{4-5} 
&$L^\infty$\,error
&Order
&$L^\infty$\,error
&Order  \\  \hline       
25
&2.13e-6 &
&5.76e-4 &  \\
50
&5.01e-8 &5.41  
&2.86e-4 &1.01    \\
100
&1.29e-9 &5.28  
&1.42e-4 &1.01 \\
200
&3.65e-11 &5.14
&6.98e-5 &1.02 \\
400
&9.95e-13 &5.20
&3.38e-5 &1.05 \\ 
\hline
\end{tabular}
\end{center}
\end{table}
\subsubsection{Shock waves}
\label{sec:coupled_shock}
The goal of this test is to show that numerical methods based on different families of paths give different results in presence of discontinuous waves. We consider the Riemann problem given by the initial condition
\begin{equation}
\left[u_0(x),v_0(x)\right]^T =
       \begin{cases}              \left[2,2\right]^T,&\text{if}  \ x \leq 0.5, \\
              \left[1,1\right]^T,&\text{} \text{otherwise}.
        \end{cases}  
\end{equation}
The analytical solution is used as the reference solution.
If the family of straight segments $\Psi_1$ is chosen, the solution is an isolated shock wave traveling at speed 3. 
After numerically simulating the solution up to \( t = 0.1 \) using 200 uniform grids, the results are presented in Figure \ref{Fig:compare_shock}. As it can be seen, in this particular case the numerical method consistent with the family of paths \eqref{psi1} captures the correct solution, while the one consistent with the family \eqref{psi2} gives a solution with a stationary contact discontinuity at $x = 0.5$ and a different shock wave traveling at the same speed. 
\begin{figure}[H]
    \centering
\includegraphics[scale=0.4]{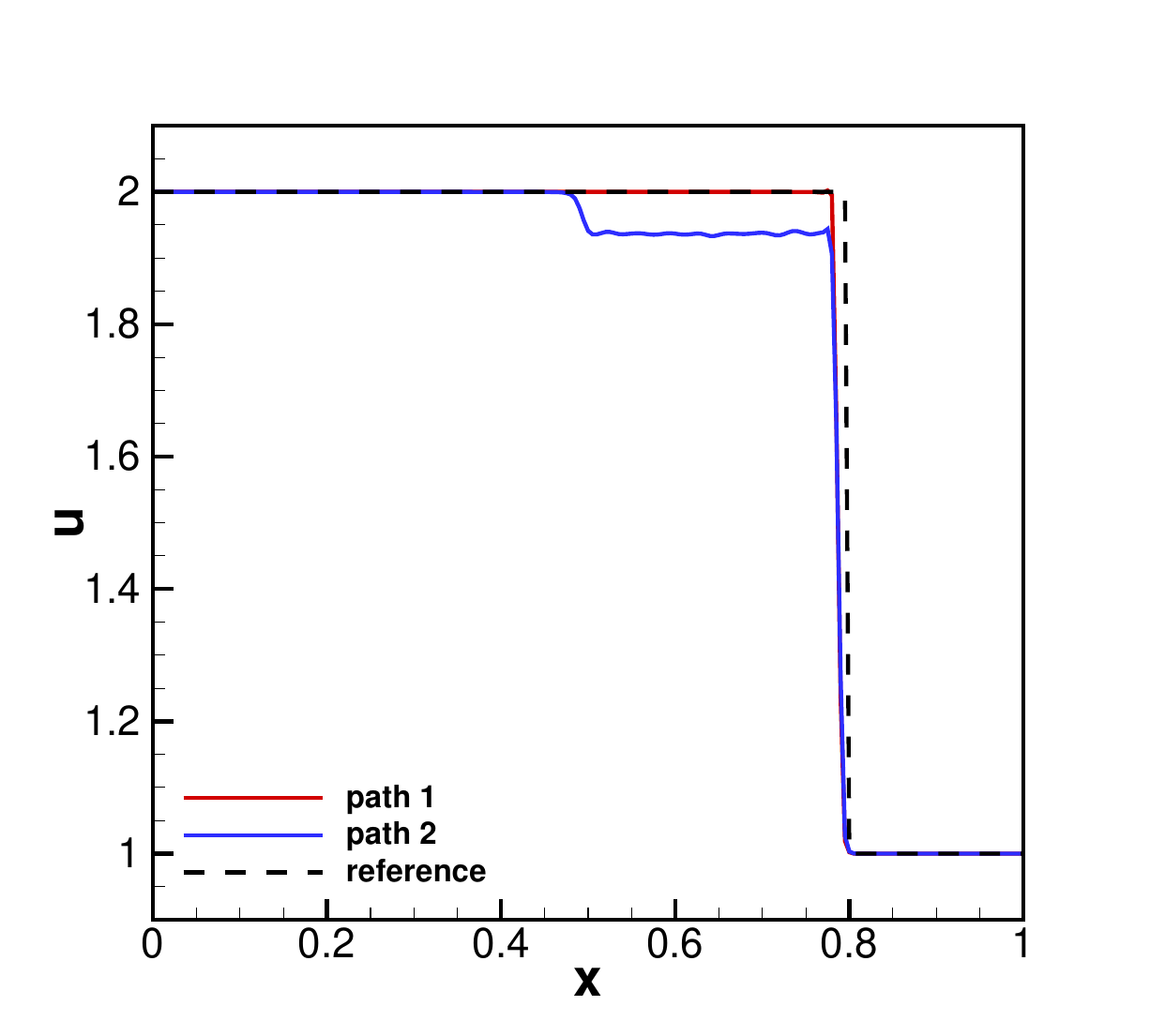}\quad
\includegraphics[scale=0.4]{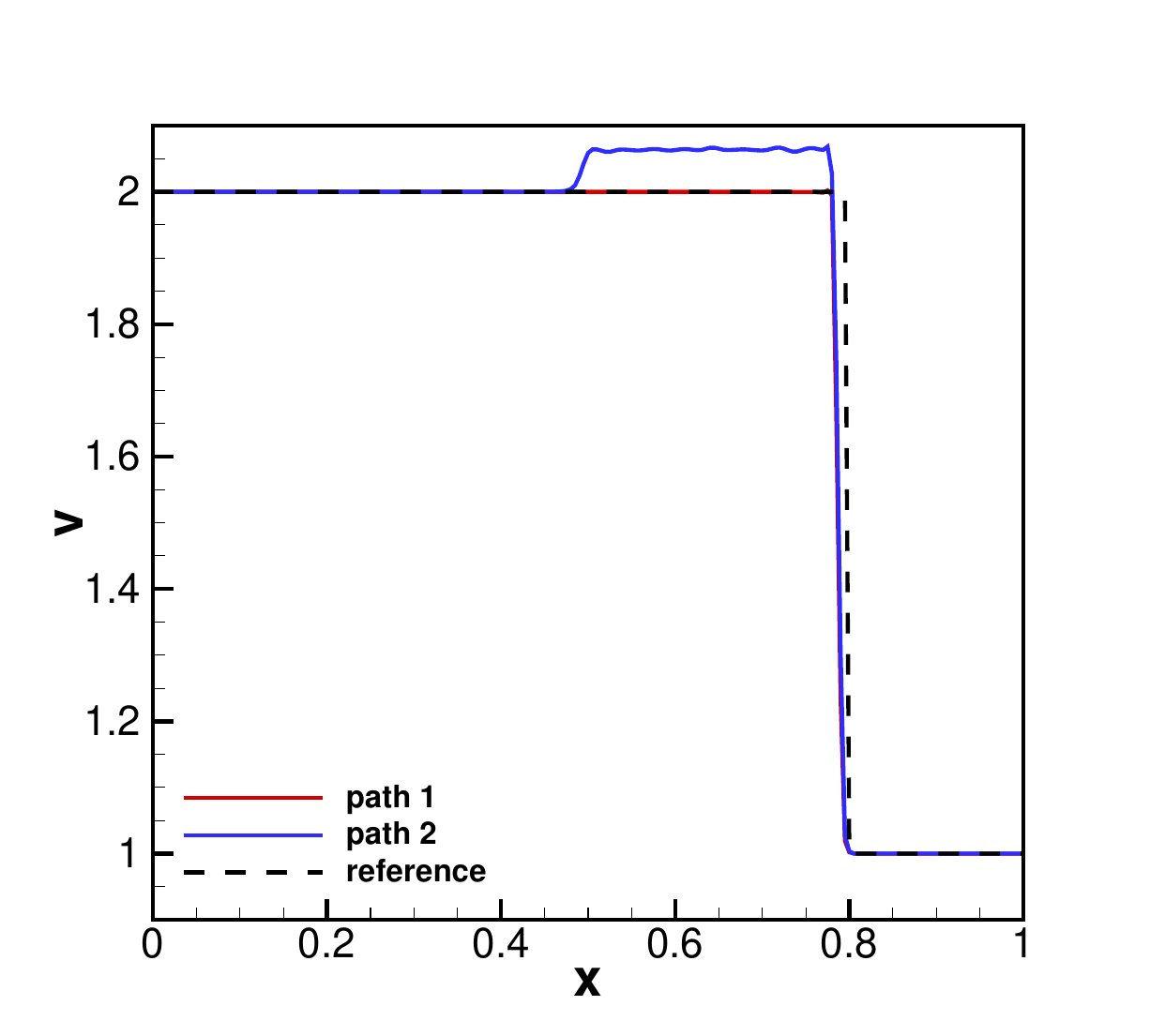} 
    \caption{ Section \ref{sec:coupled_shock}: Riemann problem for the coupled Burgers' system.  Numerical solutions for $u$ (left) and $v$ (right) at $t=0.1$ are obtained consistent with the families of paths $\Psi_1$ and $\Psi_2$ using a mesh of 200 points.}
\label{Fig:compare_shock}
\end{figure} 
\subsection{1D two-layer shallow water model}
\subsubsection{Small perturbations of a steady-state solution}
\label{subsec:small}
We use this test, taken from \cite{LEVEQUE1998346}, to verify the well-balanced property.
The density ratio is $r = 0.98 $ and the gravitational constant is $g = 10$. We consider smooth and discontinuous bottom topography. The smooth bottom topography is given by
$$
Z(x)= \begin{cases}0.25(\cos (10 \pi(x-0.5))+1)-2, & \text { if } 0.4<x<0.6, \\ -2, & \text { otherwise, }\end{cases}
$$
and the discontinuous one by
$$
Z(x)= \begin{cases}-1.5, & \text { if } x>0.5, \\ -2, & \text { otherwise. }\end{cases}
$$
The initial data is given by:
$$
\left(h_1, h_1 u_1, h_2, h_2 u_2\right)= \begin{cases}(1+\eta, 0,-1-Z(x),0), & \text { if } 0.1<x<0.2, \\ (1,0,-1-Z(x),0), & \text { otherwise. }\end{cases}
$$
The computational domain is $[-0.2,1]$ with extrapolation boundary conditions. We have first checked that for $\eta= 0$ the initial condition is preserved to machine accuracy, i.e. that the C-property is satisfied. Next, we set  $\eta= 0.00001$. The numerical solutions computed with $200$ points are compared with a reference solution computed with $2000$ points. The final time is $ T = 0.15$.
The numerical water surface, i.e. $\eta_1 = -Z + h_1 + h_2$, corresponding to the smooth and the discontinuous topographies are reported in Figure \ref{Fig:small_smooth_s123} (left and right respectively): it can be seen that small waves are captured without spurious oscillations as expected.
\begin{figure}[H]
    \centering
\includegraphics[scale=0.4]{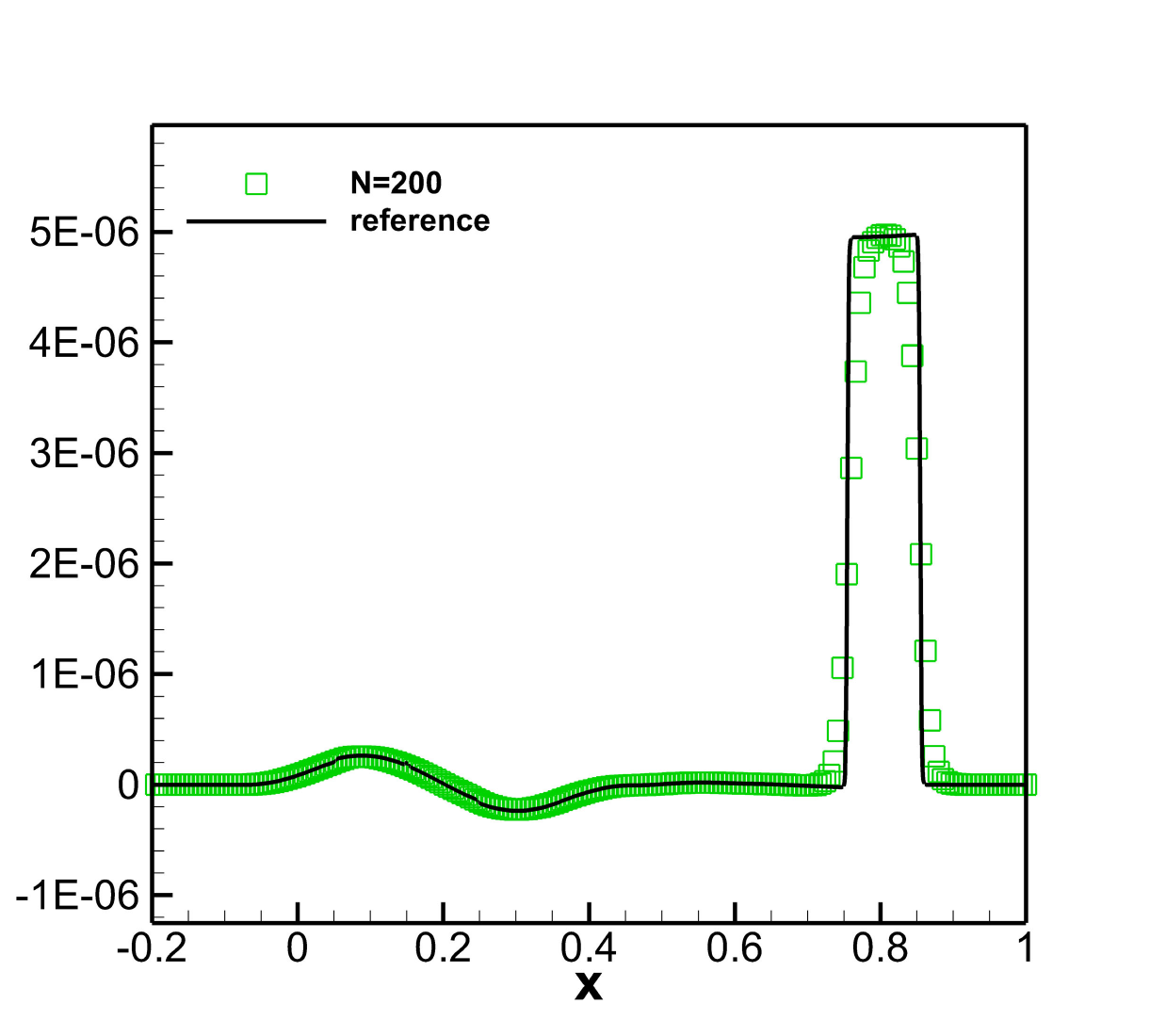} \qquad
\includegraphics[scale=0.4]{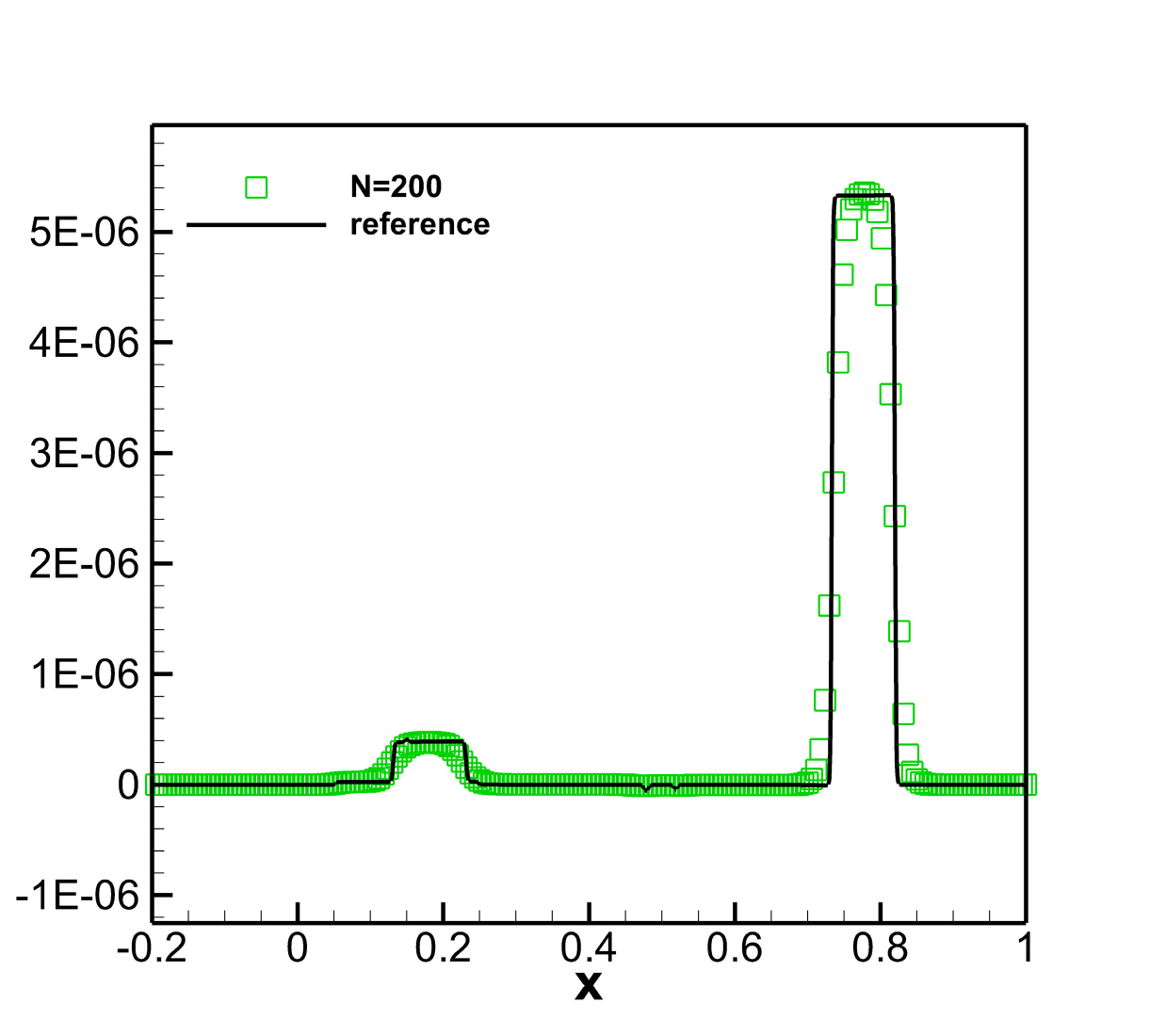} 
    \caption{Section \ref{subsec:small}:  small perturbation of a water-at-rest stationary solution with smooth (left) and discontinuous (right) bottom topographies. The numerical solution obtained with 200 points is compared with the reference solution obtained with 2000 points: water surface at $T = 0.15$.}
  \label{Fig:small_smooth_s123}
\end{figure}
\subsubsection{Riemann problems}
\label{sec:Interface}
We consider here two Riemann problems with flat bottom topography. This test is taken from \cite{Castro2001107}.  
In the first test, the computational domain is $[0,1]$ and the final time is $T=0.1$. The density ratio is $r=0.98$ and the gravitational constant is $g=10$. The initial condition is given by
$$
\left(h_1, h_1 u_1, h_2, h_2 u_2\right)= \begin{cases}(0.50,1.250,0.50,1.250), & \text { if } x<0.3, \\
(0.45,1.125,0.55,1.375), & \text { otherwise. }\end{cases}
$$
The flat bottom is placed at $Z(x) = -1$.
The water surface $E = - 1 + h_1 + h_2$, a zoom of $h_1$, and the velocity $u_1$ obtained at the final time with a mesh of 200 cells are compared in Figure \ref{Fig:s1_case1} with a reference solution obtained with 10000 points.  

In the second test case, taken from \cite{Kurganov2009}, the computational domain is $[-5,5]$ and the final time is $T=1$. The density ratio is again $r=0.98$ and the gravitational constant is $g=9.81$.
The initial data is given by
$$
\left(h_1(x, 0), q_1(x, 0), h_2(x, 0), q_2(x, 0)\right)= \begin{cases}(1.8,0,0.2,0), & \text { if } x<0, \\
(0.2,0,1.8,0), & \text { if } x>0 .\end{cases}
$$
The bottom is placed now at $Z(x) = -2$. The water surface $E = - 2 + h_1 + h_2$ and the interface $\eta = -2 + h_2$ obtained with a 500-point mesh are shown in Figure \ref{Fig:s1_case2}, together with a reference solution obtained with a fine mesh with $5000$ points.
\begin{figure}
    \centering
  \includegraphics[scale=0.35]{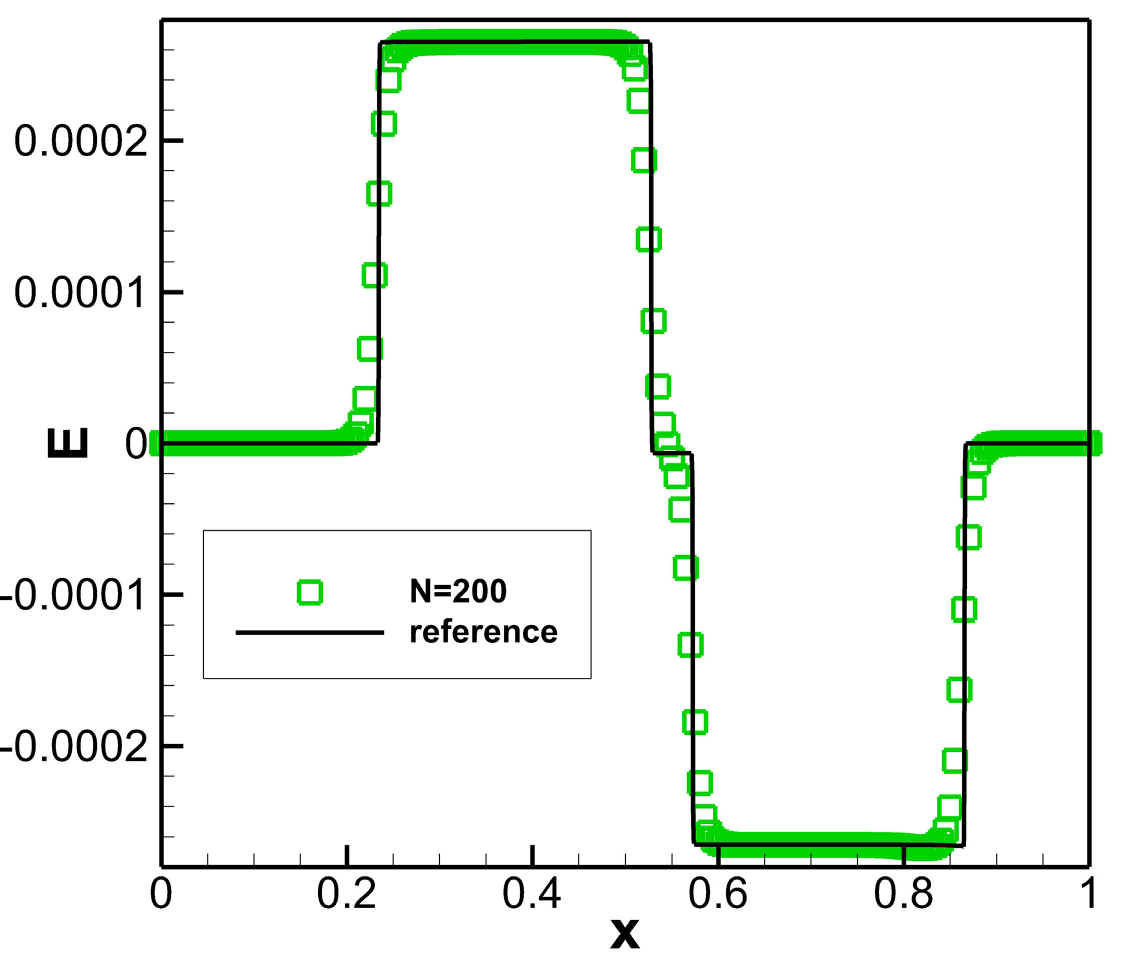}
  \qquad
\includegraphics[scale=0.35]
{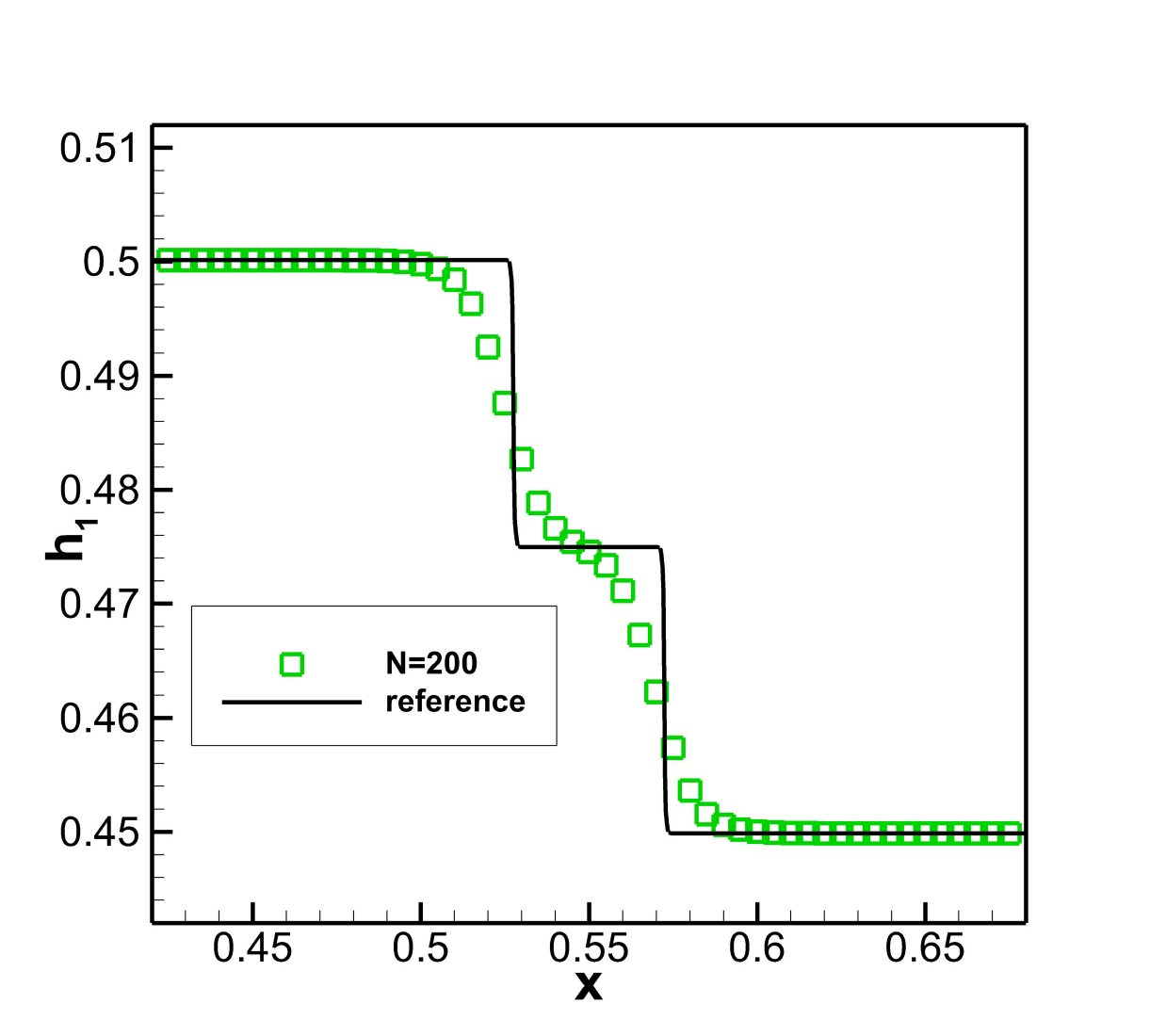}  \\
\includegraphics[scale=0.35]{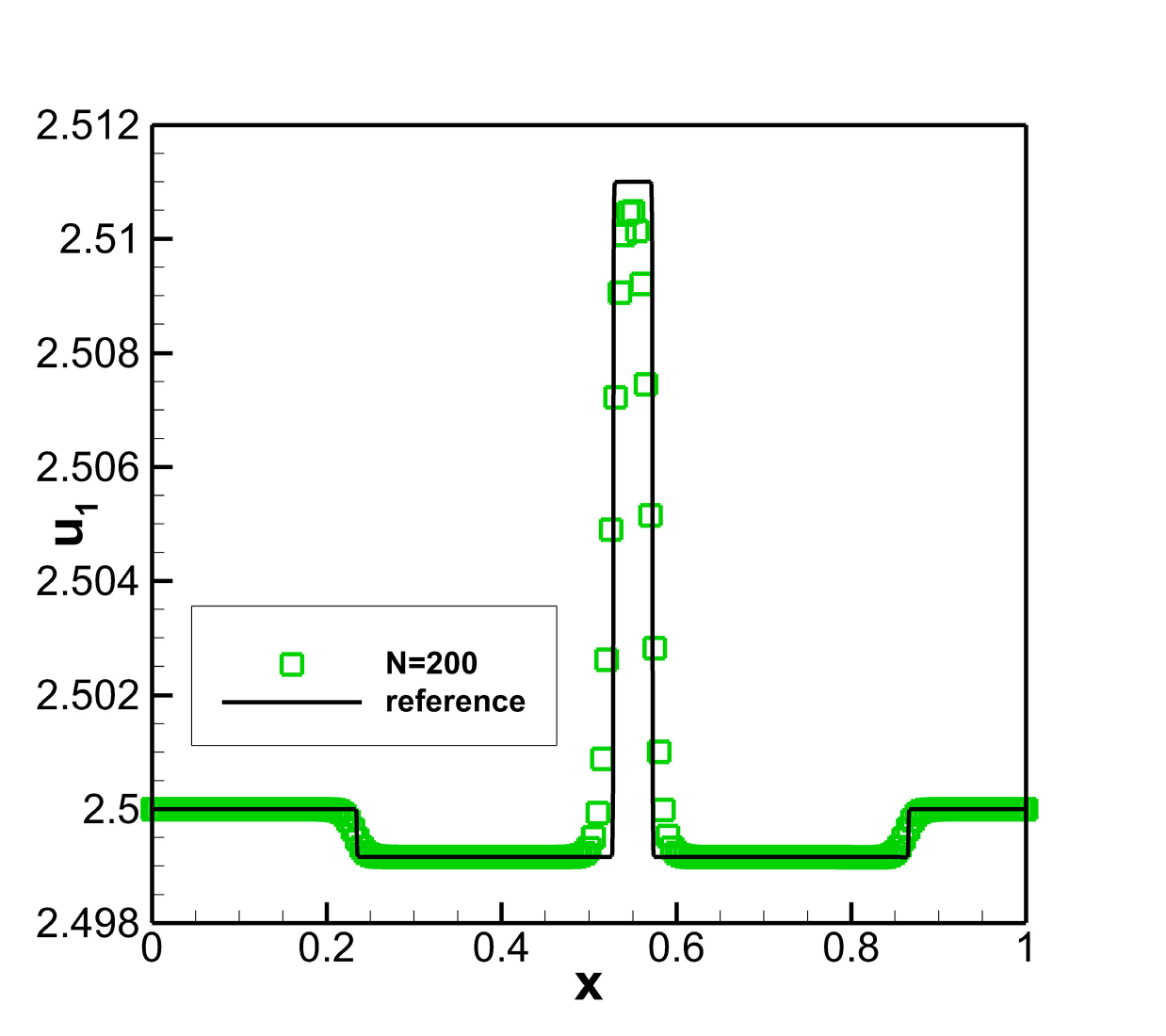}
    \caption{Section \ref{sec:Interface}: Riemann problem 1. The numerical solution computed with 200 points is compared to a reference solution computed with 10000 points. Top left: water surface; top right: $h_1$ (zoom); bottom: $u_1$.} \label{Fig:s1_case1}
\end{figure}
\begin{figure}
    \centering
\includegraphics[scale=0.35]{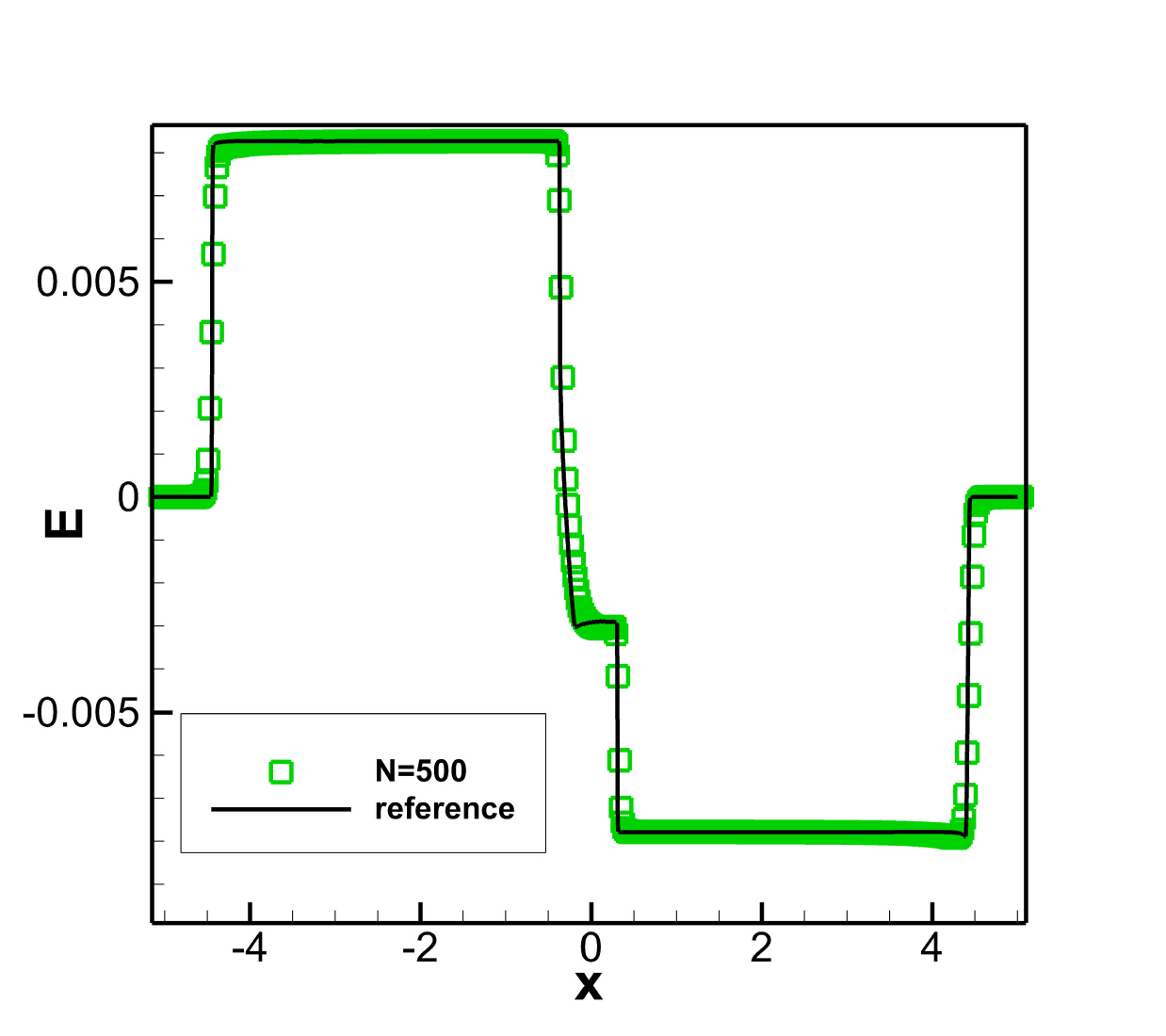}
  \qquad
\includegraphics[scale=0.35]
{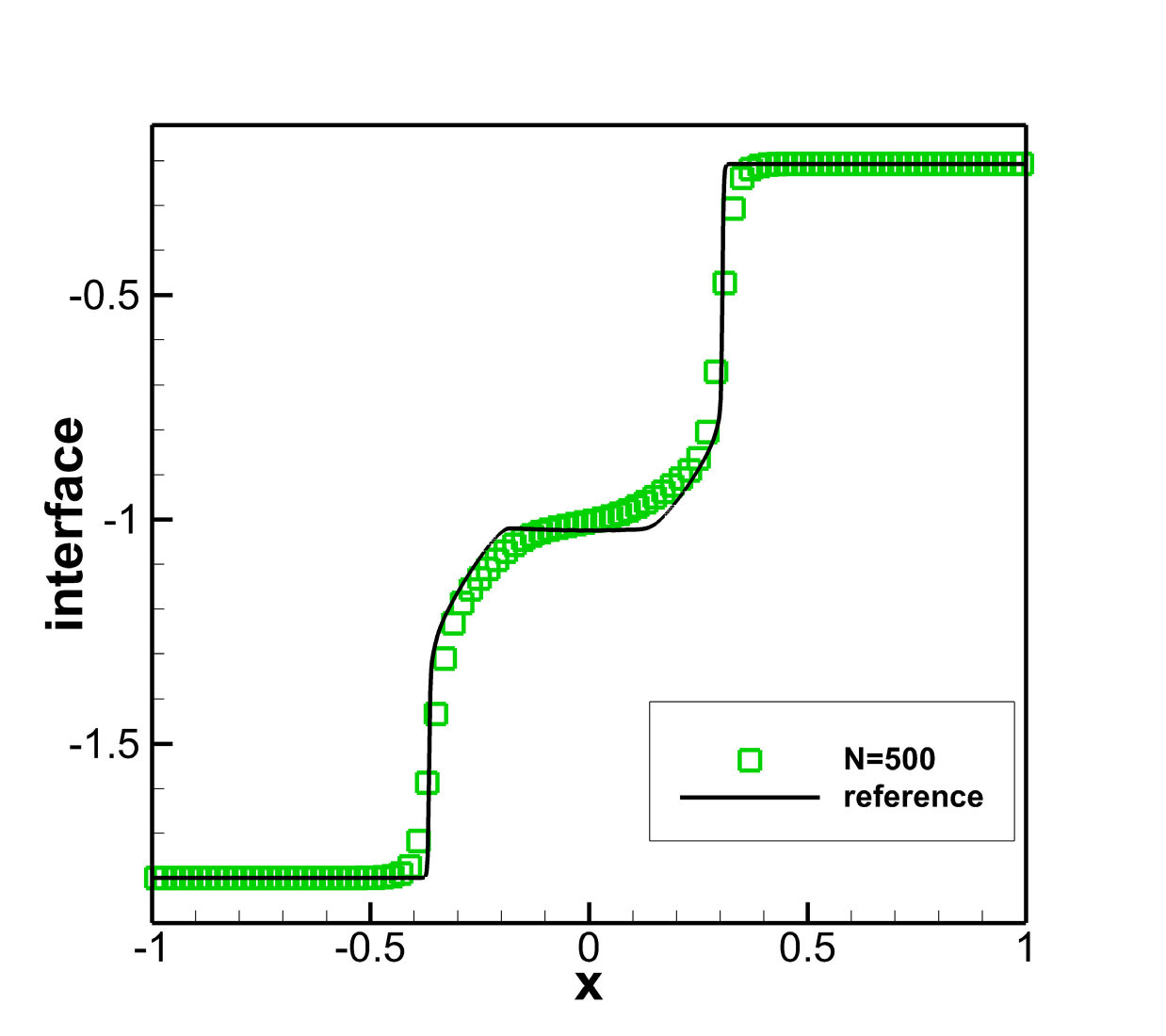}
    \caption{Section \ref{sec:Interface}: Riemann problem 2. The numerical solution obtained with $500$ points is compared to a reference solution computed with $5000$ points. Left: water surface; right: interface.}
    \label{Fig:s1_case2}
\end{figure}

\subsubsection{An internal dam-break with flat bottom}
\label{Test I_internal}
In this test, taken from \cite{Fernandez-Nieto2011}, an internal dam-break over a flat bottom is simulated. The computational domain is $[0, 10]$ with free boundary conditions. The final time is $T=10$. The gravitational constant $g=10$ and the density ratio $r = 0.98$.  The initial condition is
given by:
$$
\left(h_1(x, 0), q_1(x, 0), h_2(x, 0), q_2(x, 0)\right)= \begin{cases}(0.2,0,0.8,0), & \text { if } x<5, \\
(0.8,0,0.2,0), & \text { otherwise} .\end{cases}
$$

 The results computed with a $200$-point mesh are compared in Figure \ref{Fig:s1_An_internal_dam_breaking_1D} with a reference solution computed using $3200$ points: that solutions agree well with those in \cite{Fernandez-Nieto2011, KRVAVICA2018187}. 
\begin{figure}
    \centering
\includegraphics[scale=0.35]{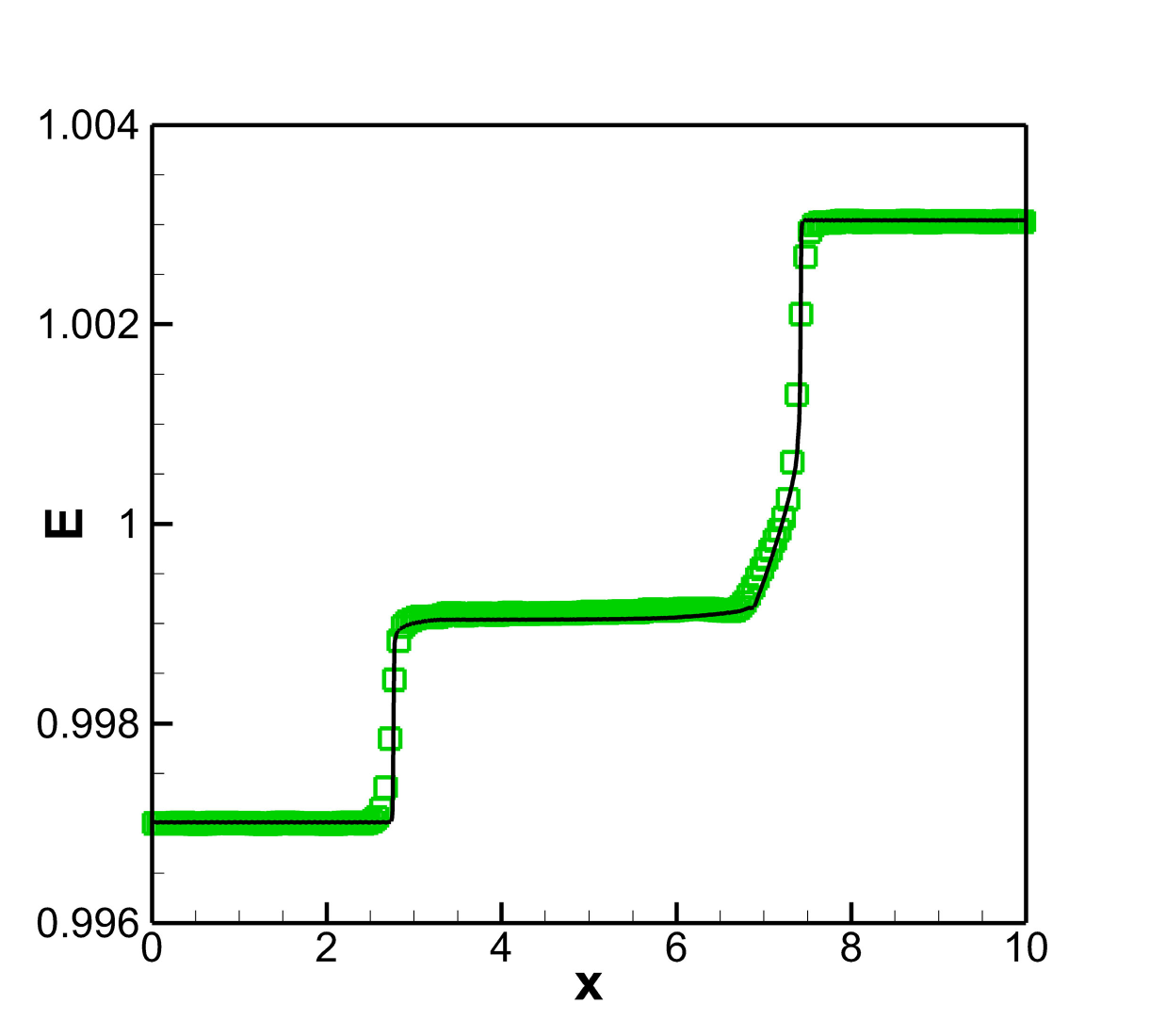}
  \qquad
\includegraphics[scale=0.35]
{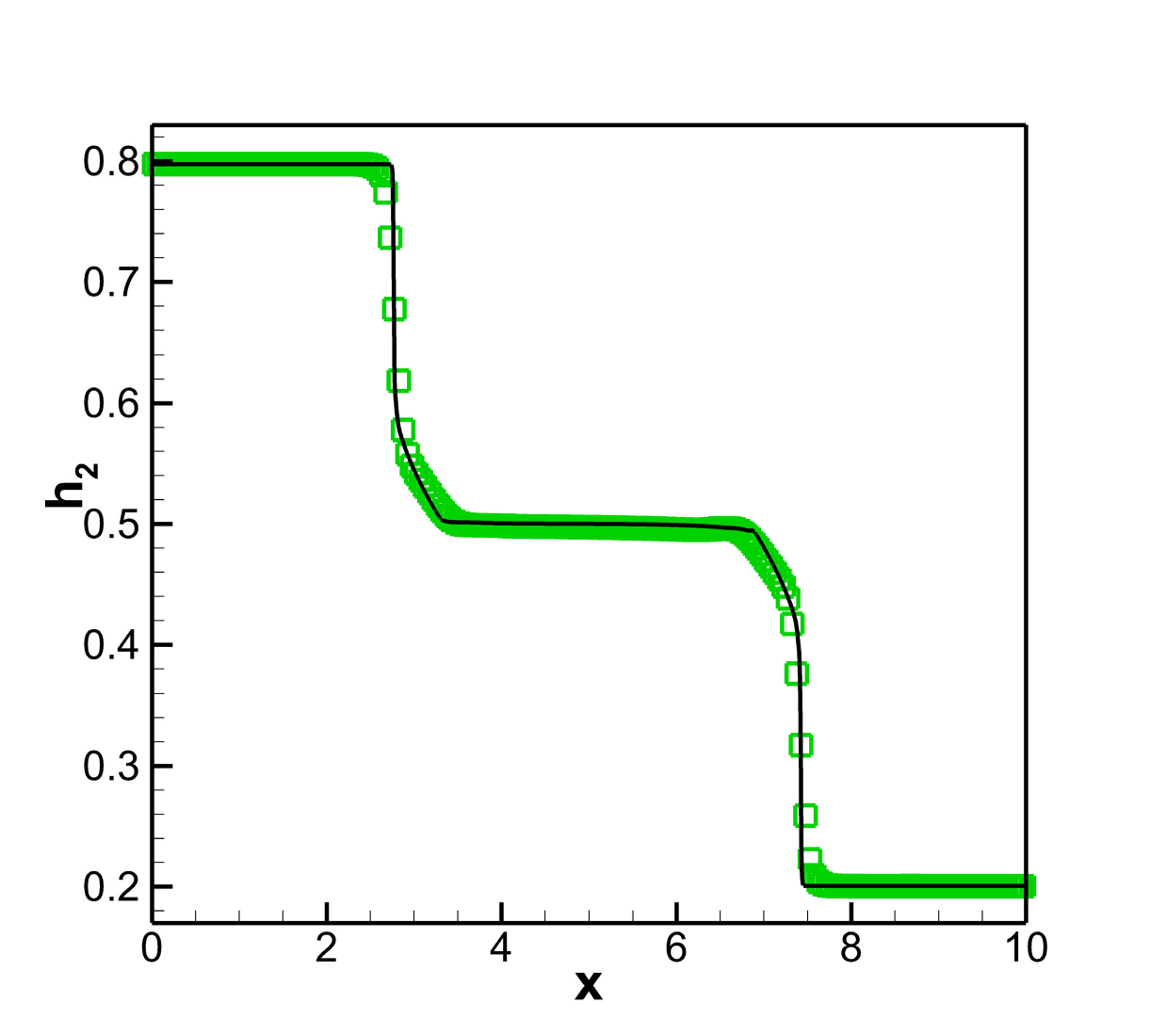}
\includegraphics[scale=0.35]{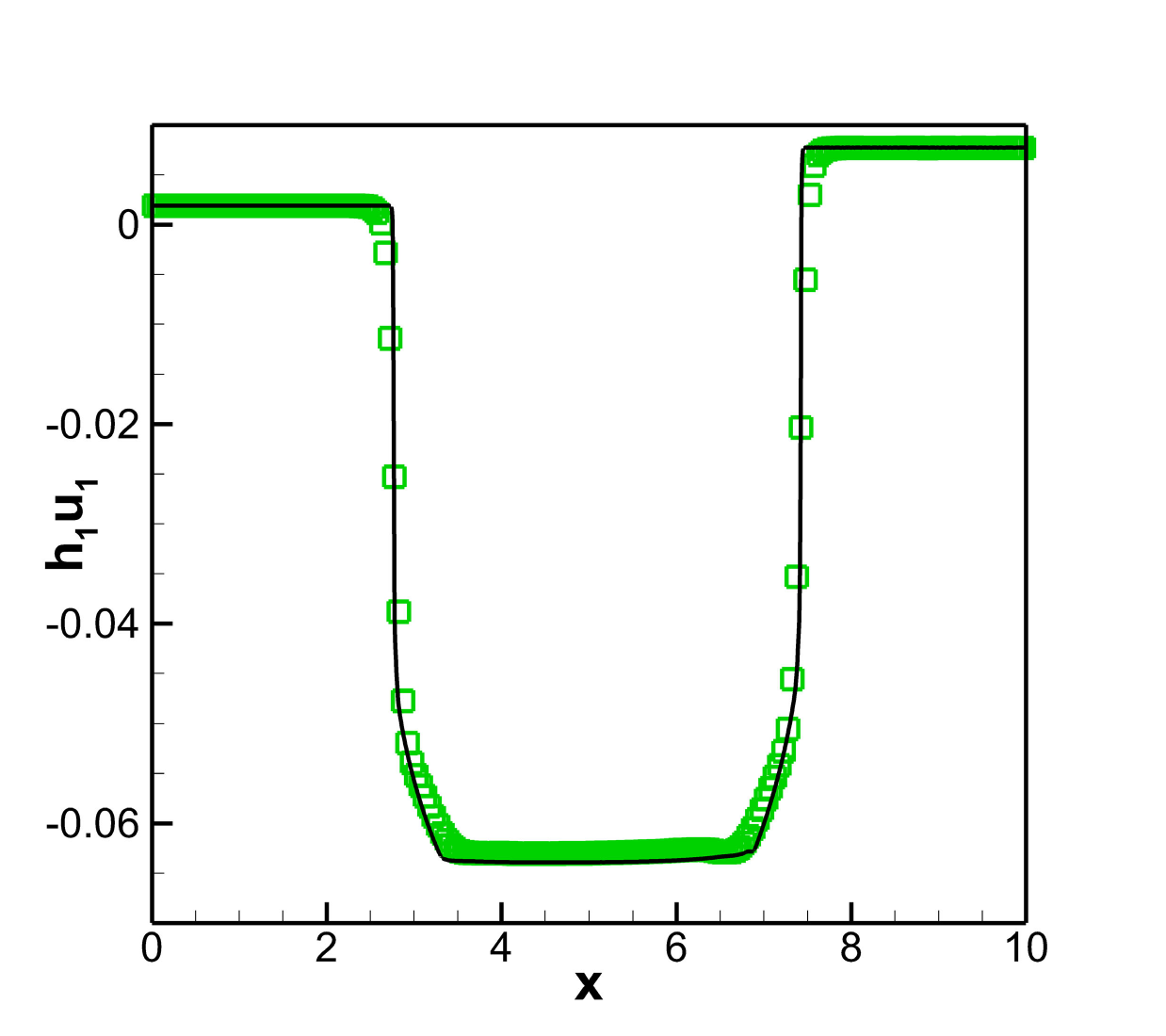}
  \qquad
\includegraphics[scale=0.35]
{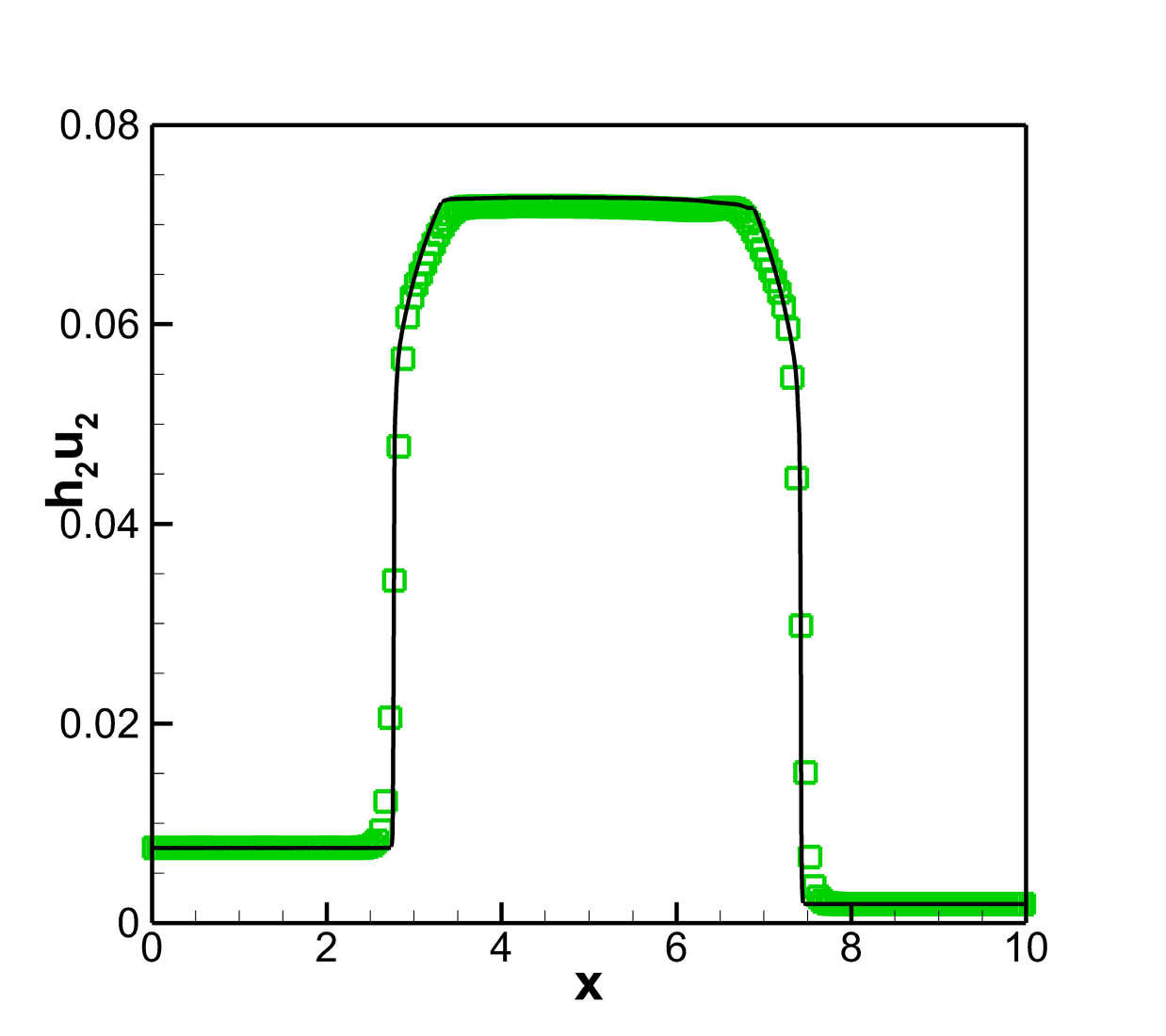}
    \caption{Section \ref{Test I_internal}: internal dam-break with flat bottom. The numerical solution computed with $200$ points is compared to a reference solution computed using $3200$ points: free surface (top-left), $h_2$ (top-right),  $h_1u_1$ (down-left), and $h_2u_2$ (down-right) at time $T = 10$.} \label{Fig:s1_An_internal_dam_breaking_1D}
\end{figure}
\subsubsection{Internal dam-break with non-flat bottom}
\label{sec:internal_nonflat}
In this test, taken from \cite{Chu2022}, a discontinuous stationary solution is reached starting from an internal dam-break initial condition given by
$$
\left(h_1, q_1, h_2+Z(x), q_2\right)(x, 0)= \begin{cases}(1.6,0,-1.6,0) & \text { if } x<0, \\ (0.7,0,-0.7,0) & \text { otherwise, }\end{cases}
$$
and the bottom topography is given by 
$$
Z(x)=0.25 e^{-x^2}-2.
$$
The constant gravitational acceleration is $g=9.81$ and the density ratio is $r=0.998$. The computational domain is  $[-5,5]$. The steady state obtained with $500$ points is shown in Figure \ref{Fig:internal_dam_break_s123}. The converged results agree well with reference solutions computed with $2000$ grids and with those in \cite{Chu2022}.
\begin{figure}[H]
    \centering
\includegraphics[scale=0.4] {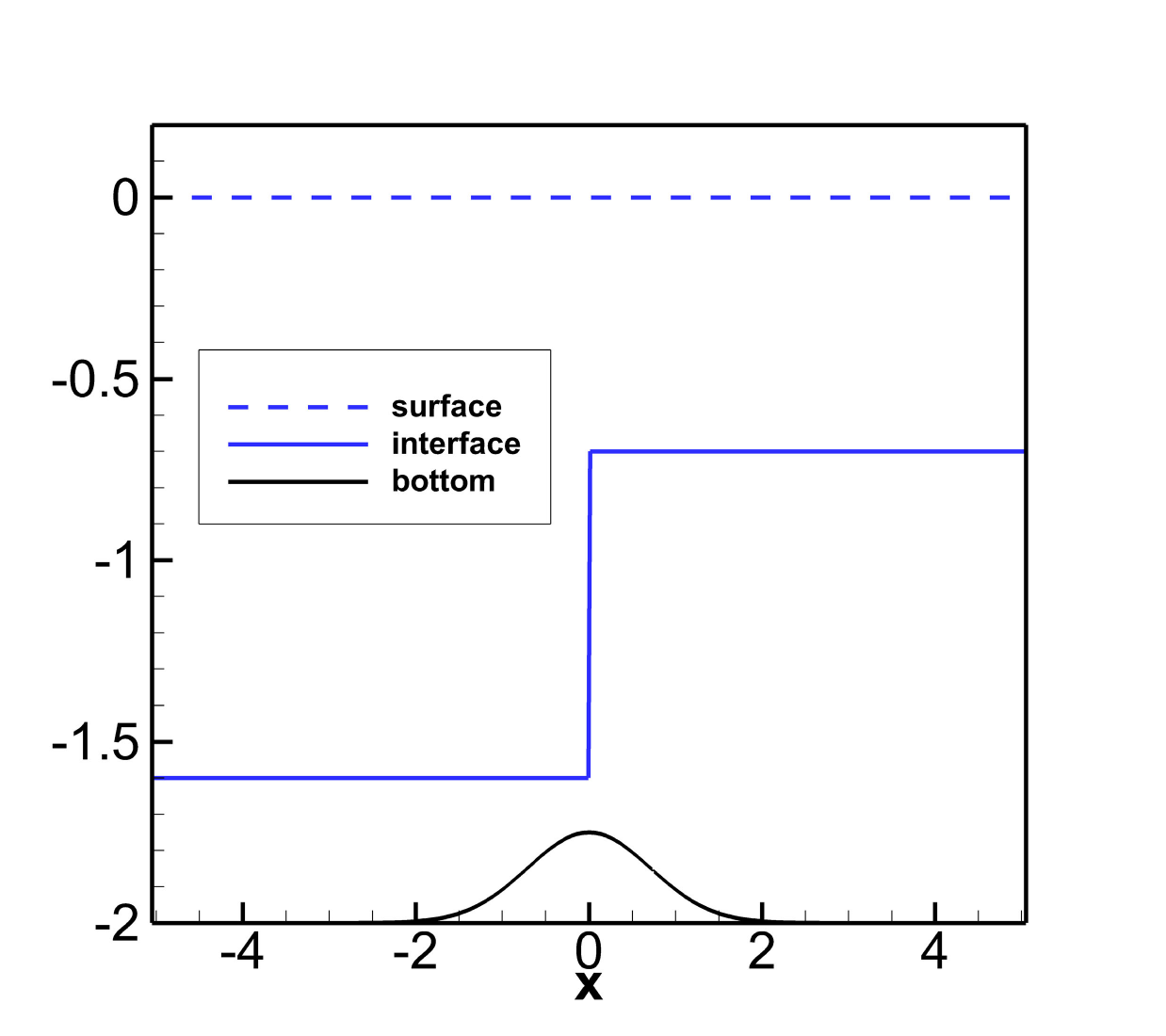}
\includegraphics[scale=0.4] {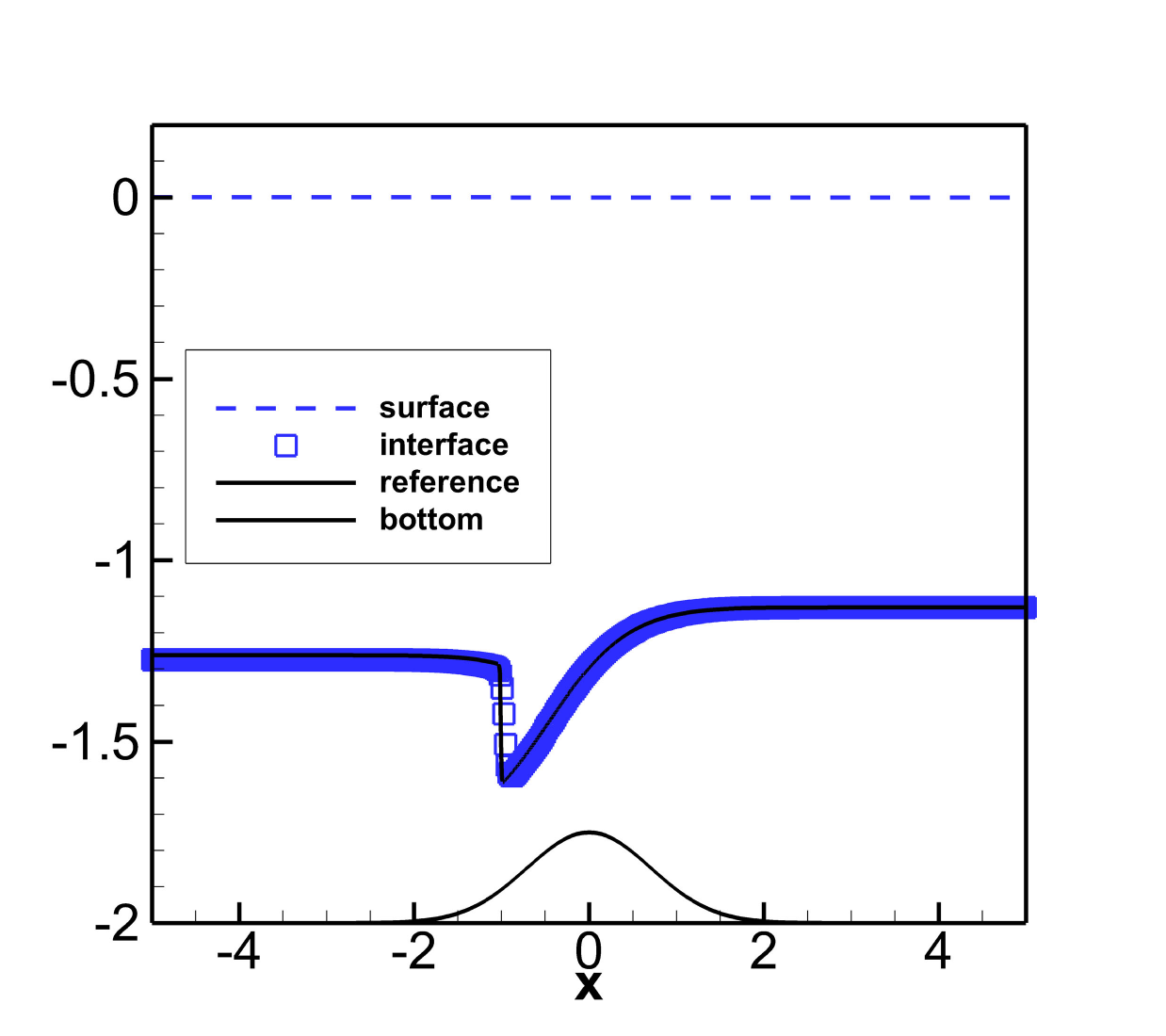}
    \caption{ Section \ref{sec:internal_nonflat}: internal dam-break with non-flat bottom. The numerical solution obtained with $500$ points is compared to a reference solution computed with $2000$ points.  Free surface, interface, and bottom topography: initial condition (left), reached steady state (right). }  \label{Fig:internal_dam_break_s123}
\end{figure}
\subsubsection{Accuracy test}
\label{Sec:order}
In this test, taken from \cite{Chu2022}, we check the empirical order of accuracy of the methods using a smooth solution. The bottom is flat, $g=9.81, r=0.98$, and the initial condition is given by:
$$
\left(h_1(x, 0), q_1(x, 0), h_2(x, 0), q_2(x, 0)\right)=\left(1-\frac{1}{2} \sin (8 x), 0,0.6+\frac{1}{2} \sin (8 x), 0\right) .
$$
 The computational domain is $[-\pi / 2,3 \pi / 2]$ and periodic boundary conditions are considered. We compute the numerical solutions until $T=0.1$.
  We use 5th, and 7th order WENO-Z for this smooth problem to check the numerical accuracy. The reference solutions for 5th are computed by each order WENO scheme with 6400 uniform grids. We use 12800 grids to obtain the reference solution for the 7th order accuracy test to get a more accurate reference. The third-order TVD Runge Kutta is used in all cases for time stepping and, in order to reach the same order of accuracy in time and space, the time-step is set to  $\Delta t^{\frac{k}{3}}$, where $k$ is the WENO order and $\Delta t$ is the step given by the CFL condition. 
 
 Tables \ref{Tab:Accuracy3} show the $L^\infty$ error and the empirical order of convergence: in all cases, the optimal order is reached.  
 
\begin{table}[htbp]
\caption{Section \ref{Sec:order}. $L^\infty$ error and convergence rates using WENO3, WENO5, and WENO7}.
\label{Tab:Accuracy3}
\begin{center}
\begin{tabular}{ccccccccccc}\hline
N
&\multicolumn{2}{c}{3th order test} 
&\multicolumn{2}{c}{5th order test}  
&\multicolumn{2}{c}{7th order test}\\
\cline{2-3}
\cline{4-5}
\cline{6-7}
&$h_1$ 
&$h_1+h_2$ 
&$h_1$ 
&$h_1+h_2$  
&$h_1$ 
&$h_1+h_2$ 
\\  \hline
25
&3.32e-1\;\;\;---\;\;\;  &5.84e-3\;\;\;---\;\;\;   
&5.38e-2\;\;\;---\;\;\; 
&1.42e-3\;\;\;---\;\;\; 
&1.62e-1\;\;\;---\;\;\; 
&3.37e-3\;\;\;---\;\;\;  \\
50  
&1.10e-2 (1.60)&2.10e-3 (1.48)    
&4.65e-3 (3.53)
&4.50e-4 (1.66)
&1.20e-2 (3.76)  
&6.26e-4 (2.43)\\
100    
&2.08e-2 (2.40)  &4.96e-4 (2.08)
&4.84e-4 (3.26)
&1.42e-5 (4.98) 
&9.58e-4 (3.65)  
&1.82e-5 (5.10)\\
200     
&3.19e-3 (2.71) &7.43e-5 (2.74) 
&1.95e-5 (4.63)
&5.52e-7 (4.69)
&1.94e-5 (5.63)       
&1.82e-7 (6.64)\\
400      
&4.24e-4 (2.91) &9.75e-6 (2.93) 
&5.92e-7 (5.04)
&1.74e-8 (4.99) 
&2.04e-7 (6.57)  
&1.49e-9 (6.94)\\
800     
&5.31e-5 (3.00) &1.23e-6 (2.99)
&1.76e-8 (5.07)
&4.33e-10 (5.33)
&1.64e-9 (6.96)& 1.01e-11(7.20)\\
\hline
\end{tabular}
\end{center}
\end{table}
   
\subsection{2D two-layer shallow water model}
\subsubsection{2-D steady-state solution}
\label{sec:steady-state-2d}
This test is used to check the well-balanced property of the methods for water-at-rest solutions. The constant density ratio is $r = 0.98$ and the gravitational constant is
$g = 10$. The bottom topography is given by a smooth function
$$
Z(x, y)=0.05 e^{-100\left(x^2+y^2\right)}-1,
$$
and the initial condition is given by
$$
h_1 = 0.5, \; h_2  = 1 - Z(x,y), \; u_{1,1} = u_{1,2} = u_{2,1} = u_{2,2} = 0.
$$
The computational domain is $[-1,1] \times [-1,1]$. The final time is $T=0.1$. This initial boundary value problem is a modification of
the exact C-property test, proposed in \cite{xing2005high} for the shallow water equations.
The numerical $L^1$ errors corresponding to two meshes of $50 \times 50$ and $100 \times 100$ points are shown in Table \ref{Tab:error_WB_2d}, as it can be seen the initial condition is preserved to machine accuracy.

We have used this test to compare the computational costs of Methods 1 and 2: the CPU times corresponding to a computation with $200 \times 200$ points are shown in Table \ref{Tab:time_WB_2d}. The computational costs of Strategy 1 and 2 are comparable while Strategy 1 is slightly more computationally expensive.
 
\begin{table}[htbp]
\caption{Section \ref{sec:steady-state-2d}. $L^1$ errors at time $0.1$ using two meshes of $50 \times 50$ and $100 \times 100$ points.}
\label{Tab:error_WB_2d}
\begin{center}
\begin{tabular}{cccccccc} \hline
 \multicolumn{1}{c}{}
 &\multicolumn{1}{c}{$h_1$} & {$h_1u_1$} & {$h_1v_1$}
 &\multicolumn{1}{c}{}
 &\multicolumn{1}{c}{$h_2$} & {$h_2u_1$} & {$h_2v_2$}
 \\  
 \cline{1-8}
 $50\times 50$  &7.77e-16  &1.36e-16 &2.79e-16& &9.81e-16 &9.78e-15  &1.36e-16\\
 $100\times 100$ &1.98e-15  &1.88e-16 &5.71e-16& &1.44e-15 &1.32e-14  &1.88e-16 \\
 \hline
 \end{tabular}
 \end{center}
\end{table}
\begin{table}[htbp]
\caption{Section \ref{sec:steady-state-2d}. Methods 1 and 2: CPU times and speedup}.
 \label{Tab:time_WB_2d}
\begin{center}
\begin{tabular}{ccc} \hline
& CPU time & Speedup ratio\\
\hline
Method 1 & 68.38 s & \\
Method 2 & 63.54 s &1.08 \\
 \hline
\end{tabular}
\end{center}
\end{table}

\subsubsection{Interface propagation with flat bottom}
\label{subsec:Interface_flat_2d}
This test, taken from \cite{Kurganov2009}, is aimed to verify the robustness of the numerical method.
The initial condition is
$$
\left(h_1, h_1u_1, h_1v_1, h_2, h_2u_2, h_2v_2\right)(x, y, 0)= \begin{cases}(0.50,1.250,1.250,0.50,1.250,1 .250), & \text { if }(x, y) \in \Omega, \\ (0.45,1.125,1.125,0.55,1.375,1.375), & \text { otherwise,}\end{cases}
$$
where $\Omega=\{x<-0.5, y<0\} \cup\left\{(x+0.5)^2+(y+0.5)^2<0.25\right\} \cup\{x<0, y<-0.5\}$. The constant gravitational acceleration is $g=10$ and the density ratio is $r=0.98$. The flat bottom topography is given by $Z(x, y) \equiv -1$. The computational domain is $[-0.55,0.7]\times[-0.55,0.7]$ and the final time, $T=0.1$. The computational results obtained with two meshes of $400\times 400$ and $800\times 800$ are shown in Figure \ref{Fig:interface_propagation_in_2D_with_flat_bottom}.
\begin{figure}
    \centering
 \includegraphics[scale=0.38]{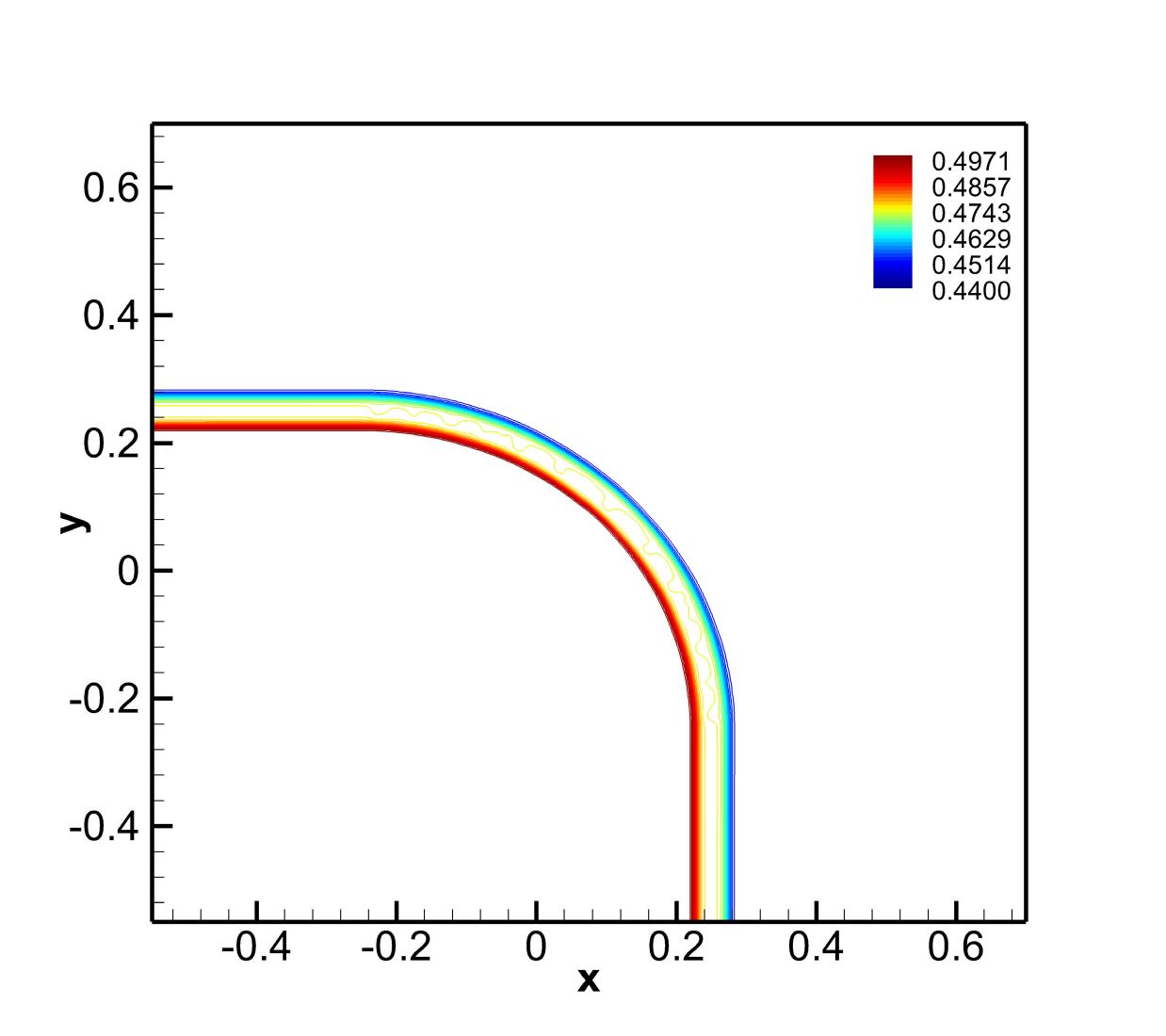}
  \qquad
\includegraphics[scale=0.38]{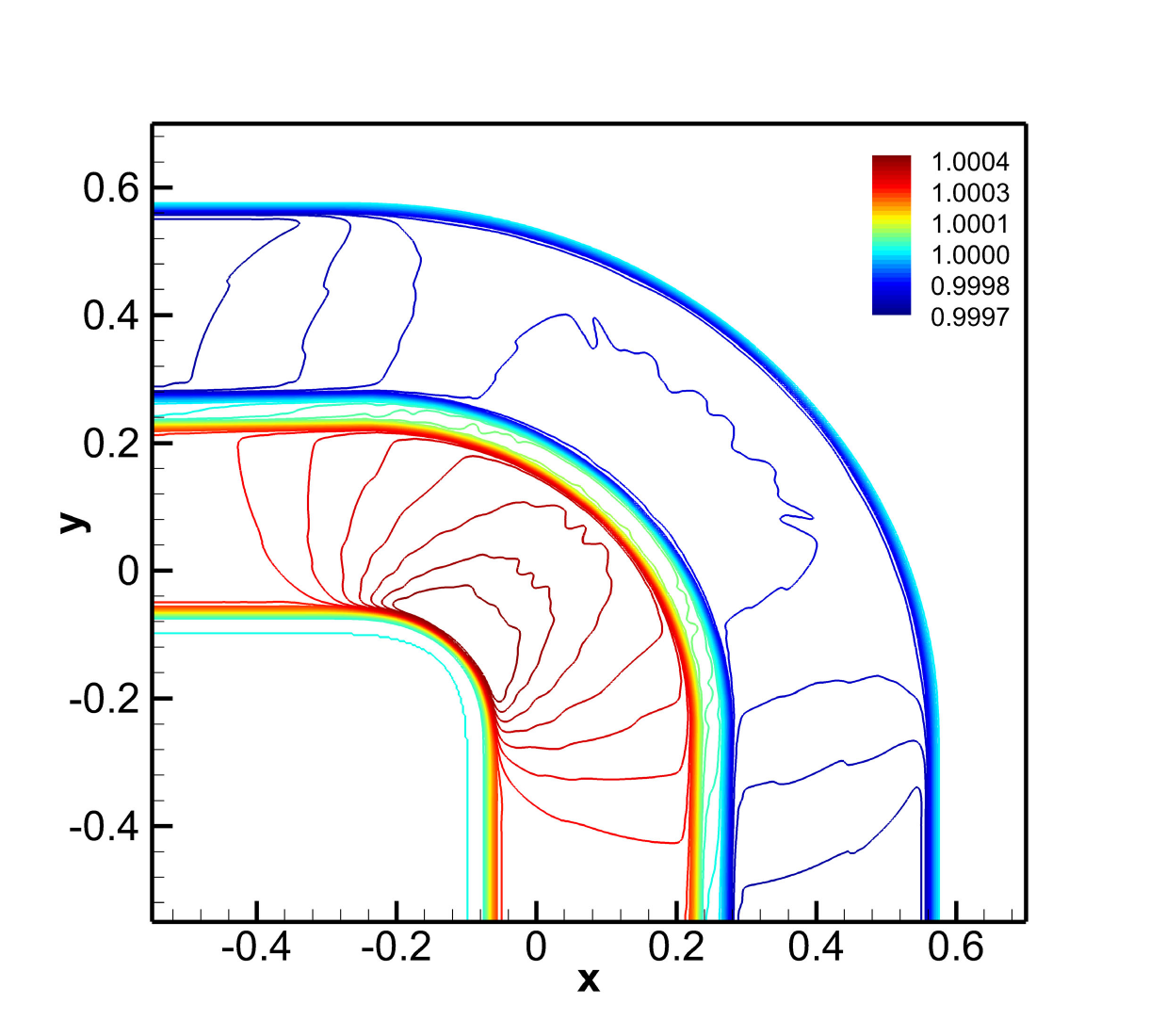}\\
\includegraphics[scale=0.38]{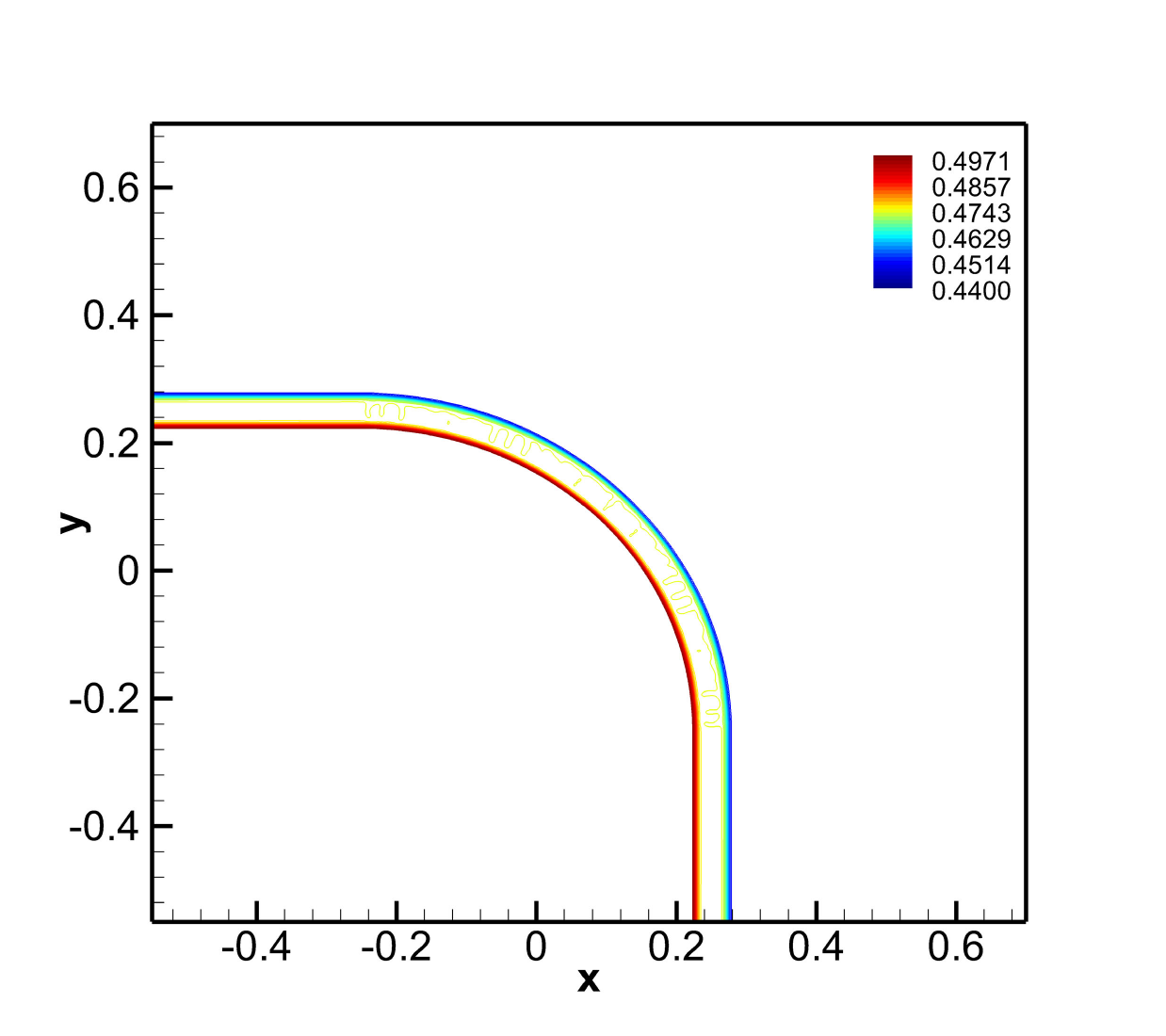}
  \qquad
\includegraphics[scale=0.38]{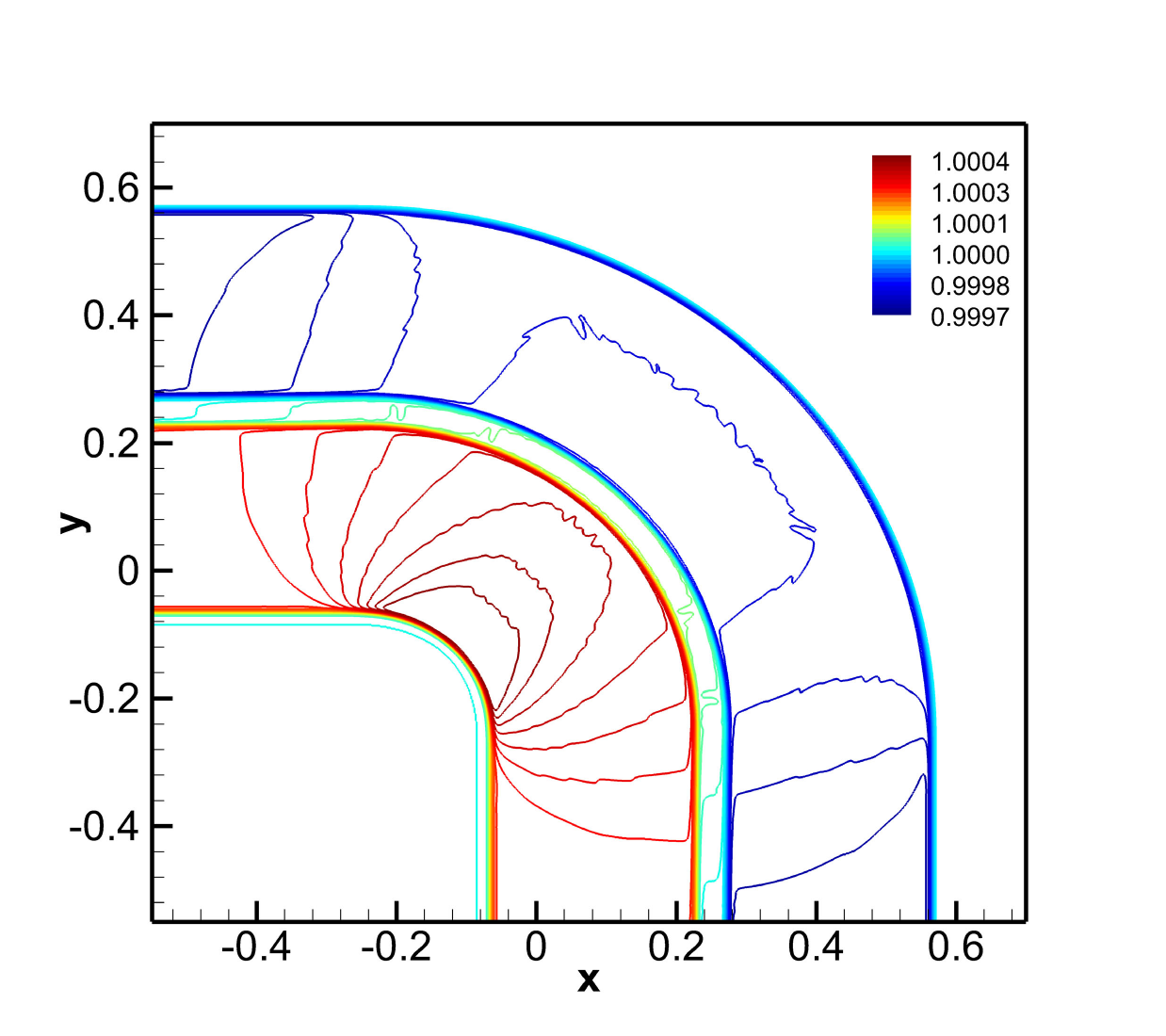}
    \caption{Section \ref{subsec:Interface_flat_2d}: interface propagation in 2D with flat bottom. Numerical solutions obtained with $400\times 400$ (top row) and $800\times 800$ (bottom row) grids: upper layer thickness $h_1$ (left) and water surface $ h_1 + h_2 $ (right).}
\label{Fig:interface_propagation_in_2D_with_flat_bottom}
\end{figure}
\subsubsection{Interface propagation with non-flat bottom}
\label{subsec:Interface_nonflat_2d}
This test is similar to the one in Section \ref{subsec:Interface_flat_2d} but with a non-flat bottom given by
$$
Z(x, y)=0.05 e^{-100\left(x^2+y^2\right)}-1,
$$
and the following initial data:
$$
\left(h_1, h_1u_1, h_1v_1, h_2, h_2u_2, h_2v_2\right)(x, y, 0)= \begin{cases}(0.50,1.250,1.250,0.50-Z,1.250,1.250), & \text { if }(x, y) \in \Omega, \\ (0.45,1.125,1.125,0.55-Z,1.375,1.375), & \text { otherwise. }\end{cases}
$$
The computational domain, the values of $g$ and $r$, and the final time are the same.  We show again the results computed with $400\times 400$ and $800\times 800$ points in Figure \ref{Fig:interface_propagation_in_2D_with_nonflat_bottom_robust}. The numerical solutions are in good agreement with those
presented in \cite{Liu2021} but finer structures of the flow are actually captured here.
\begin{figure}
    \centering
 \includegraphics[scale=0.38]{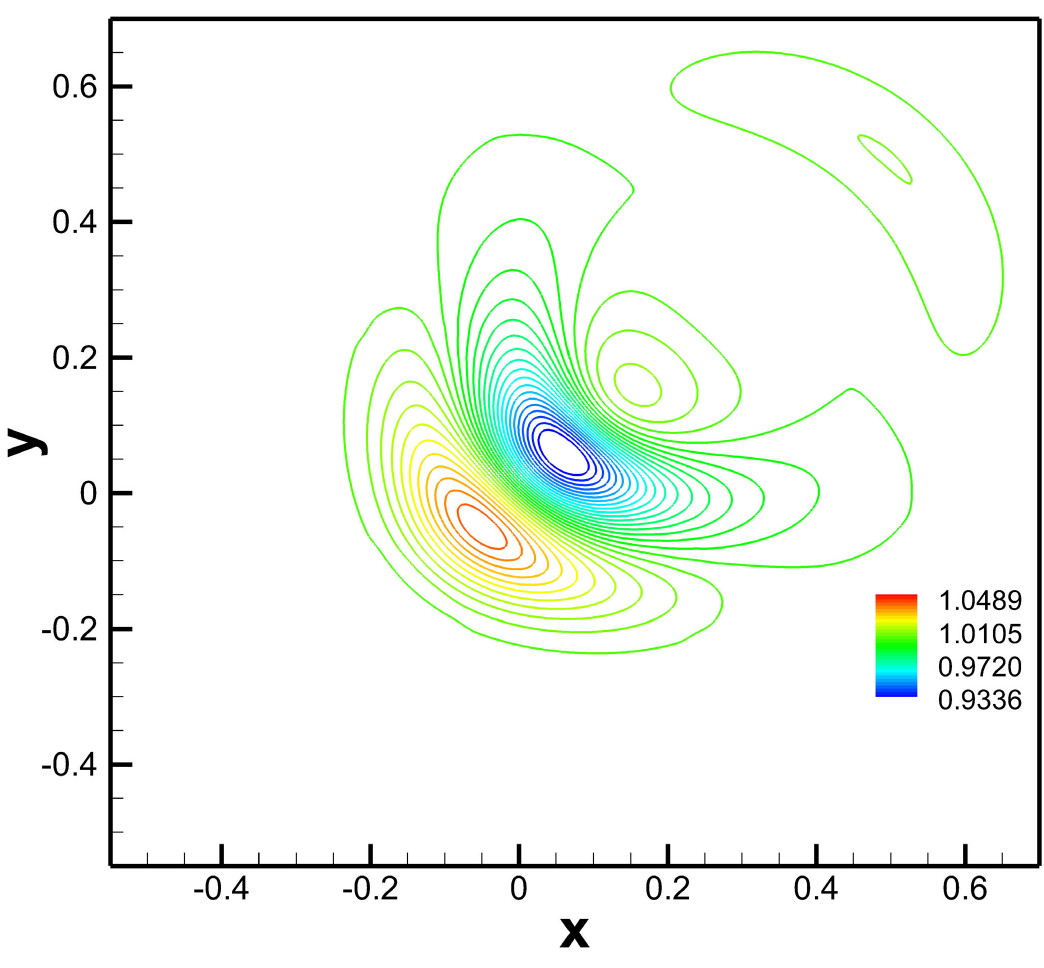} 
  \qquad
\includegraphics[scale=0.38]{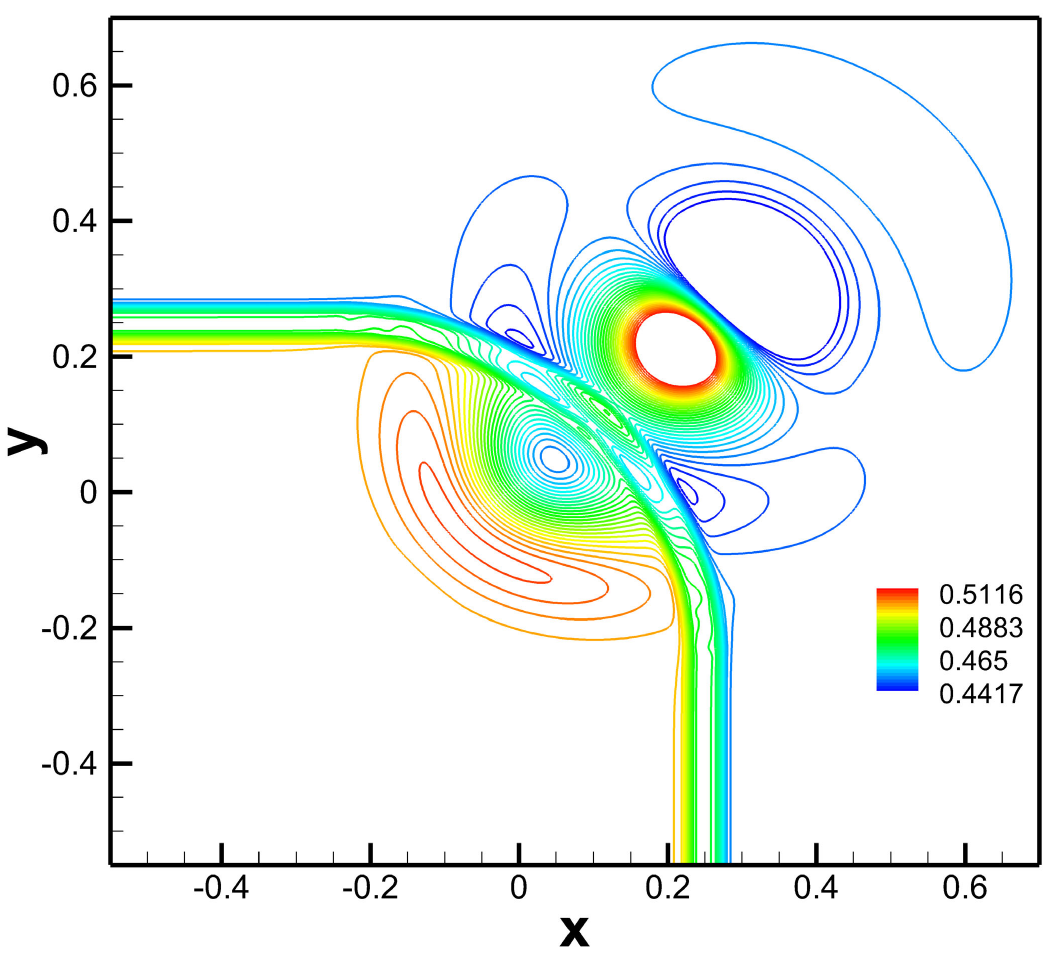}\\ 
\includegraphics[scale=0.38]{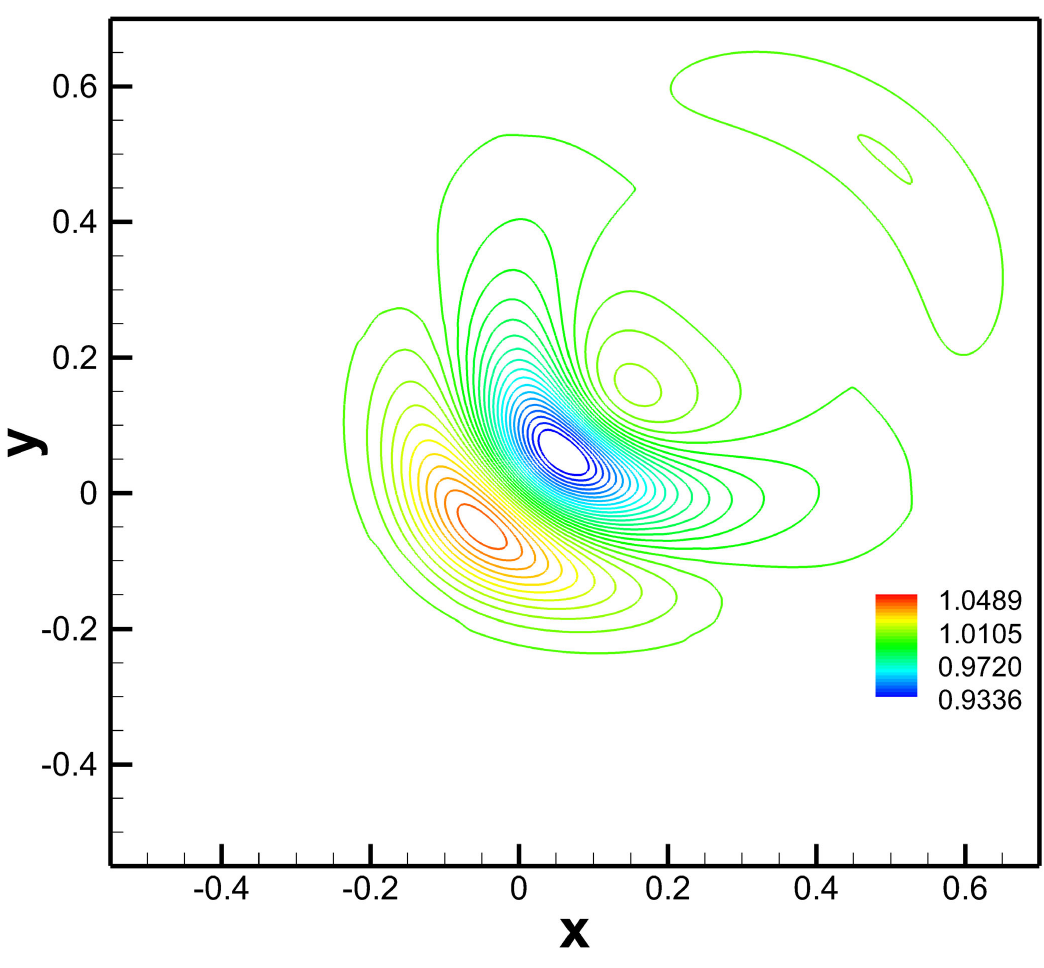}
  \qquad
\includegraphics[scale=0.38]{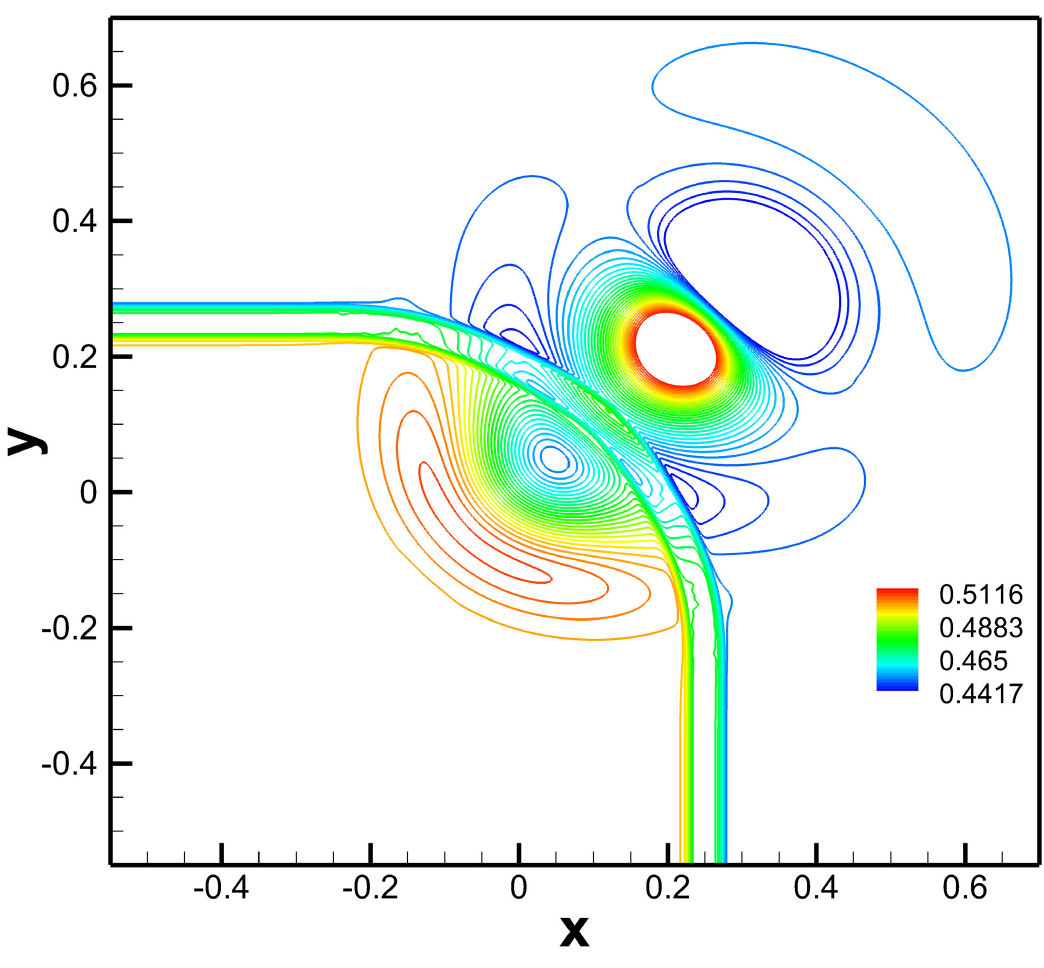}
    \caption{Section \ref{subsec:Interface_nonflat_2d}: interface propagation in 2-D with nonflat bottom. Numerical solutions computed with $400 \times 400$ (top row) and $800 \times 800$ (bottom row) grids: water surface $h_1 + h_2 + Z$ (left) and upper layer thickness $h_1$ (right).}  \label{Fig:interface_propagation_in_2D_with_nonflat_bottom_robust}
\end{figure}
\subsubsection{Internal circular dam break over flat bottom topography}
\label{subsec:Internal_Circular_2d}
This example is taken from \cite{Castro2009}: an internal circular dam-break problem with a flat bottom is considered. The initial conditions are given by
$$
\left(h_1, q_1, p_1, h_2, q_2, p_2\right)(x, y, 0)= \begin{cases}(1.8,0,0,-1.8-Z,0,0), & \text { if } x^2+y^2>4, \\ (0.2,0,0,-0.2-Z,0,0), & \text { otherwise. }\end{cases}
$$
In this test, we consider $g=9.81$, $r=0.998$, and  $Z \equiv -2$. The final time is $T=20$.
We show the contour lines of the water interface for times $t = 4, 20$ in Figure \ref{Fig:circular_bottom_interface_2d}.
Diagonal slices $y = x$ of the numerical results computed with $200\times 200$ points for times $t = 4, 6, 10, 14, 16, 20$ are
shown in Figure \ref{Fig:internal_circular_bottom_interface}. 
It is worth mentioning that the results agree well with those in \cite{Castro2009, Chu2022}.
\begin{figure}
    \centering
 \includegraphics[scale=0.35]{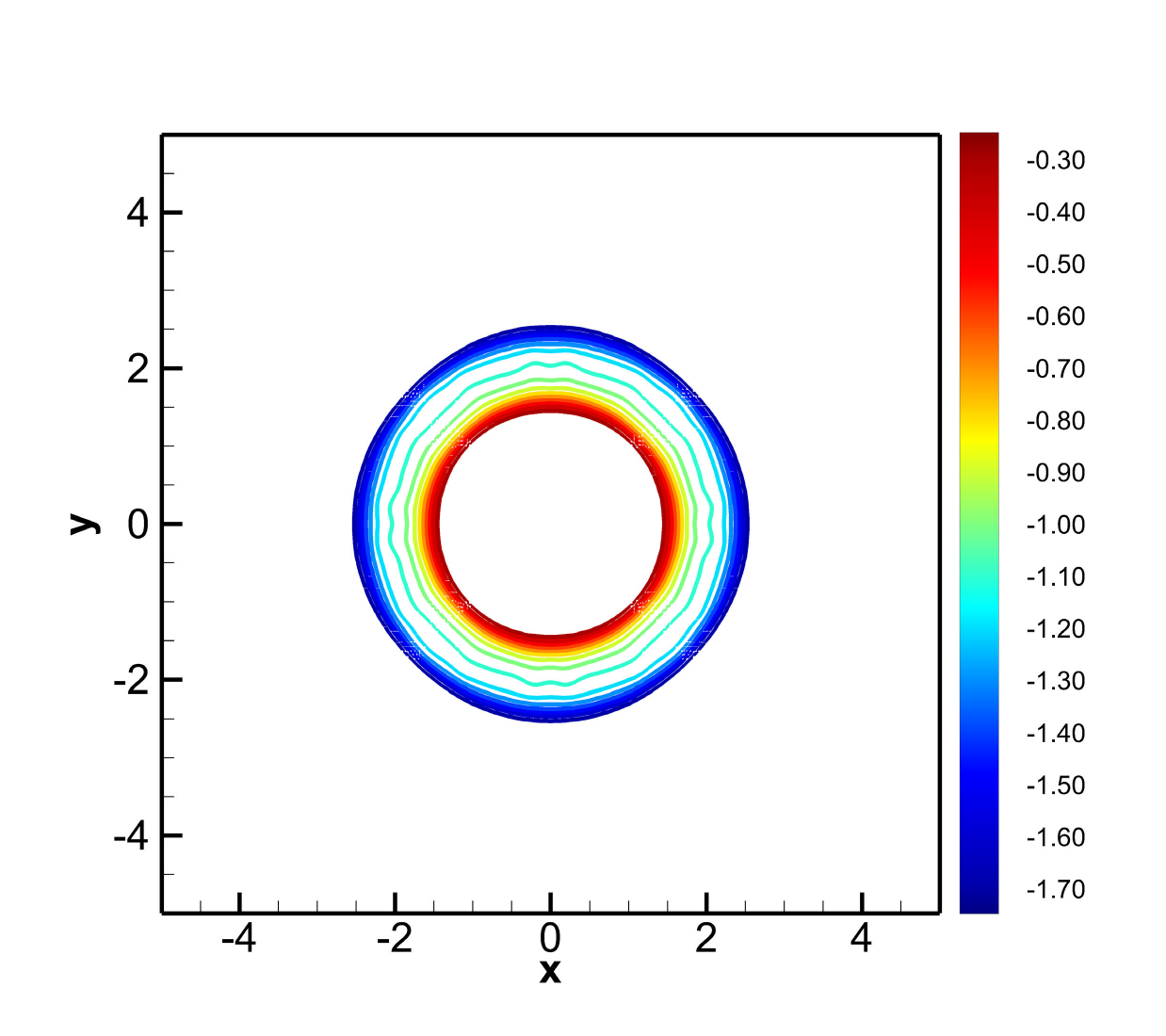}
  \qquad
\includegraphics[scale=0.35]{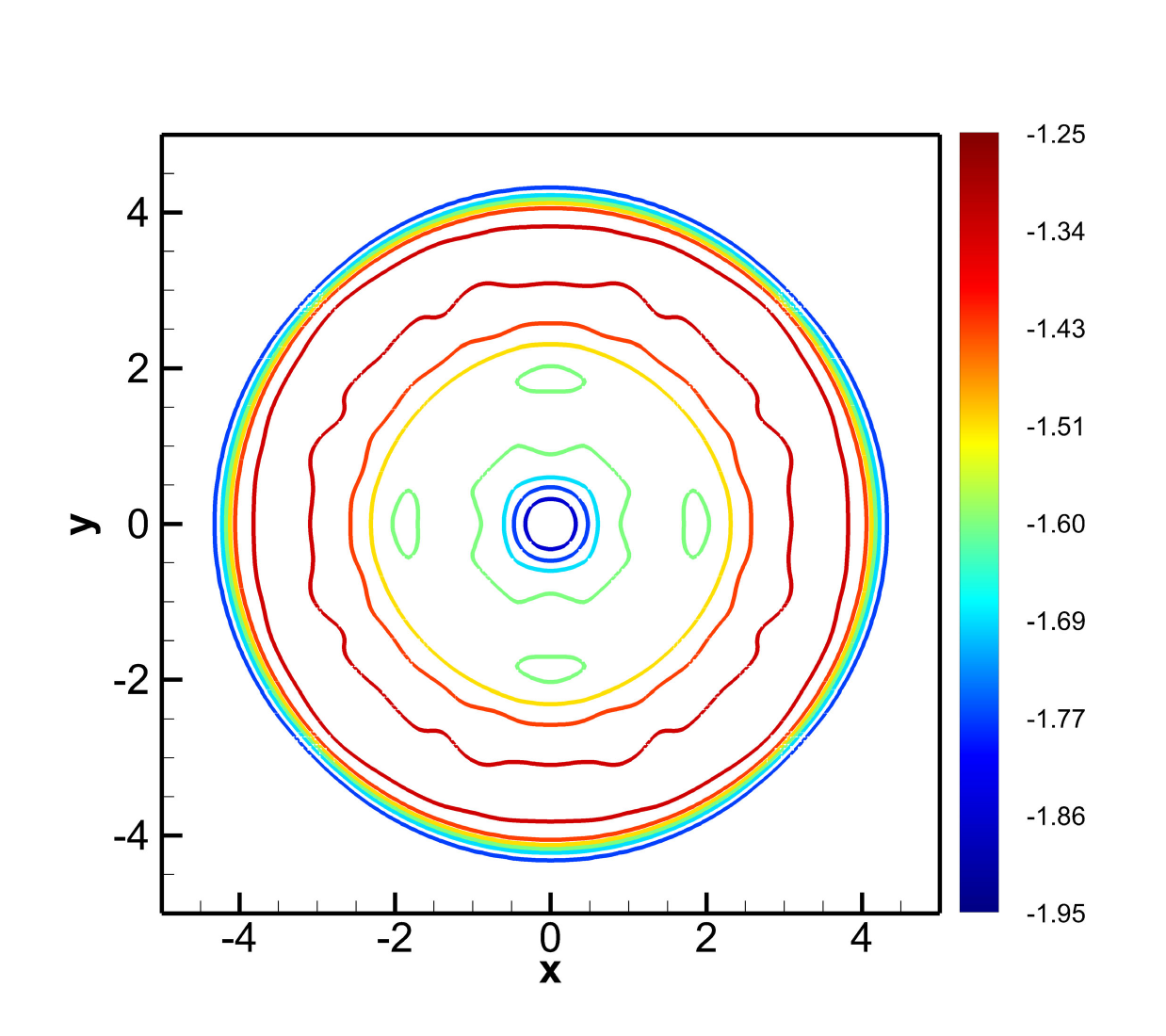}
\caption{Section \ref{subsec:Internal_Circular_2d}: internal circular dam breaking in 2D with a flat bottom.  Contour lines of the interface $h_2+Z$ at times $t=4$(left), $t=20$(right). }   \label{Fig:circular_bottom_interface_2d}
\end{figure}
\begin{figure}
    \centering
 \includegraphics[scale=0.35]{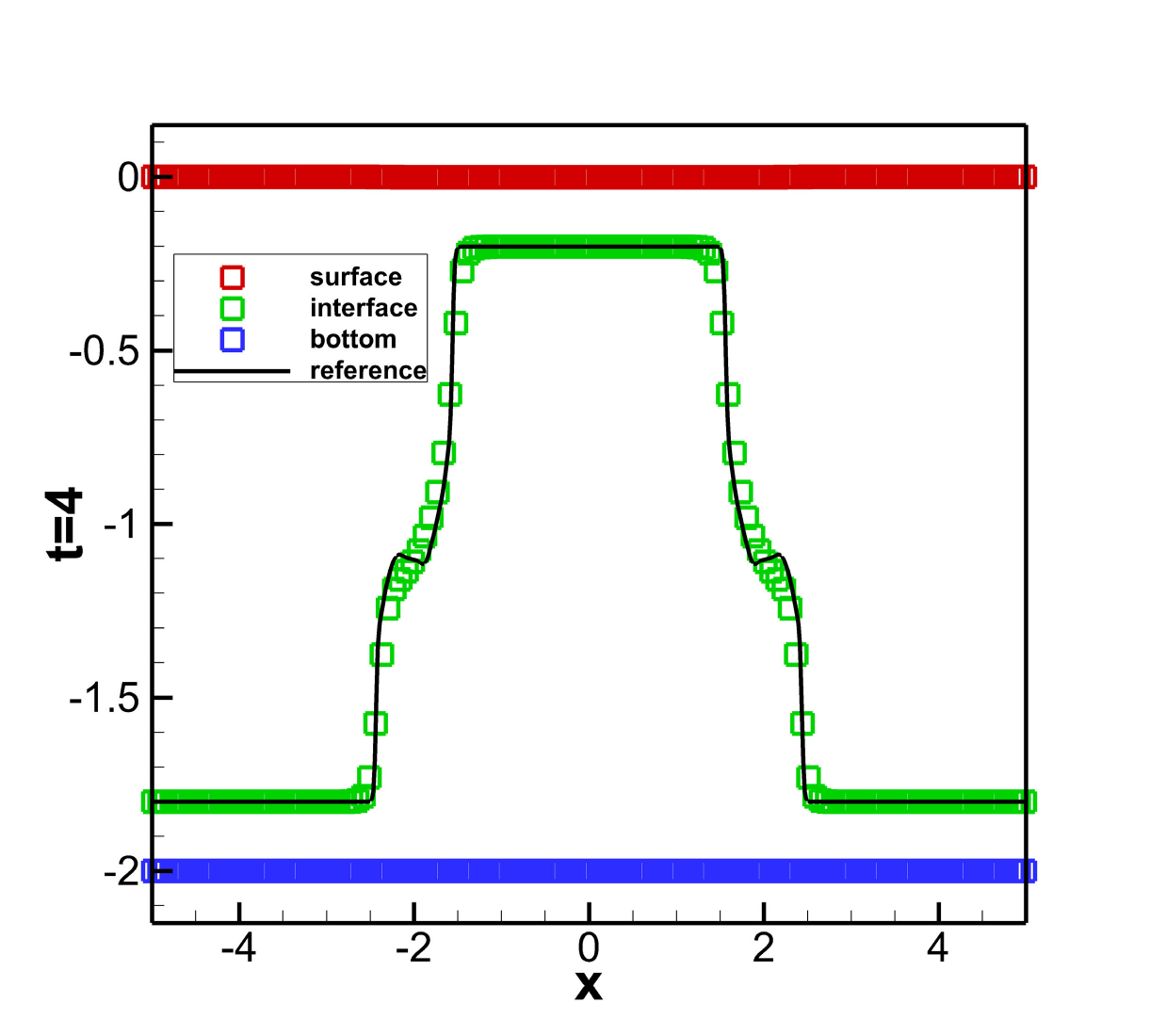}
  \qquad
\includegraphics[scale=0.35]{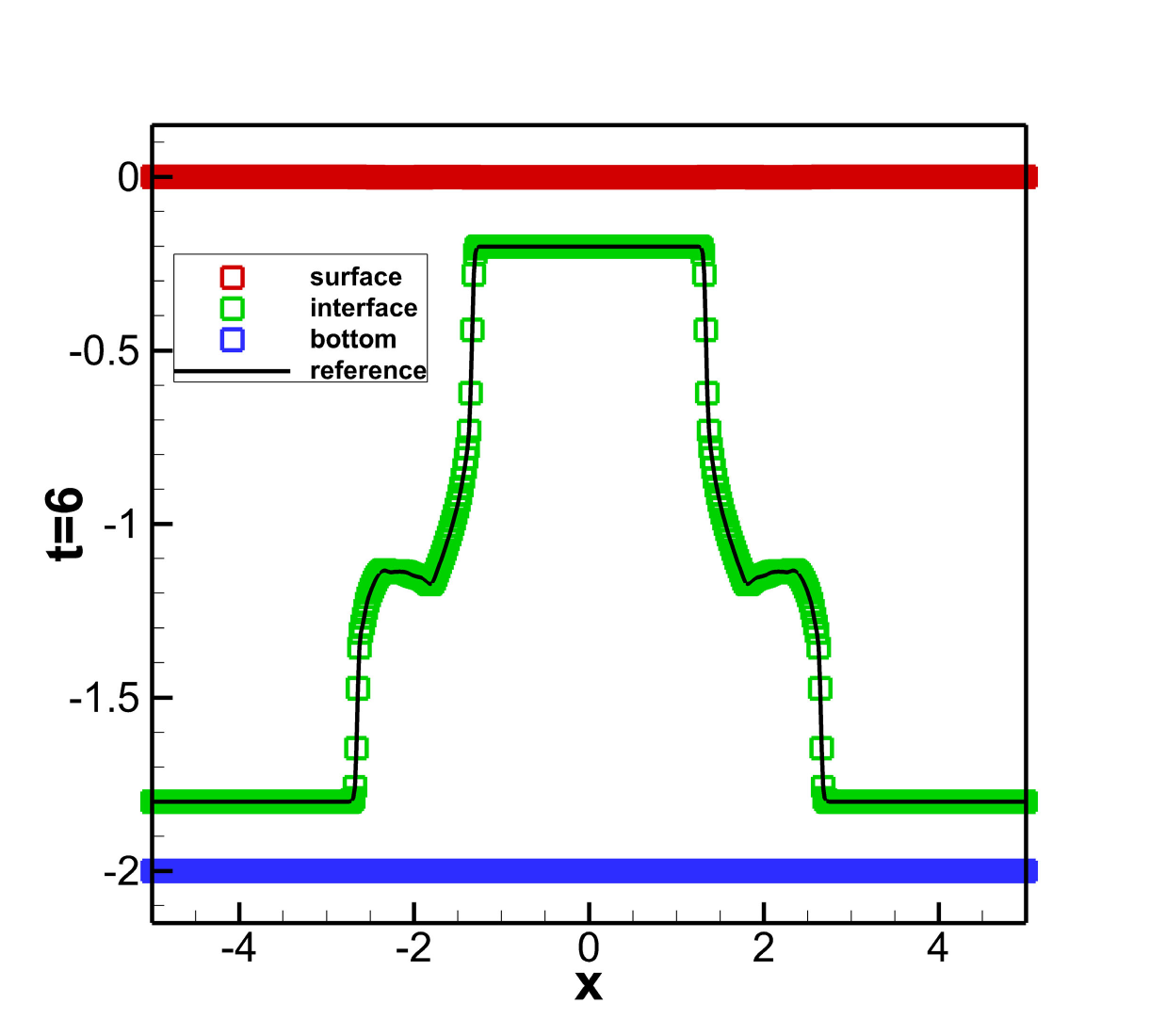} \\
\includegraphics[scale=0.35]{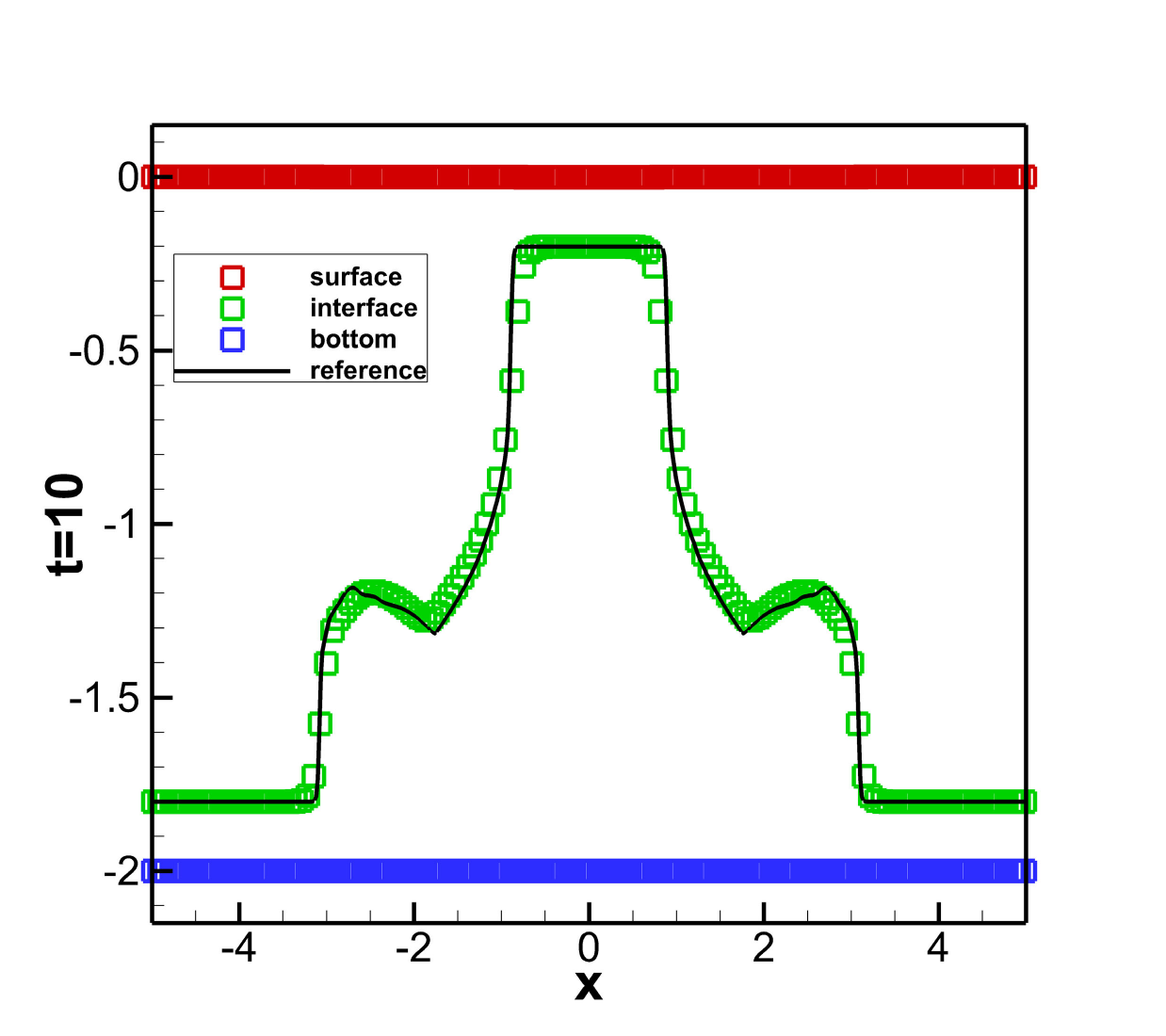}
  \qquad
\includegraphics[scale=0.35]{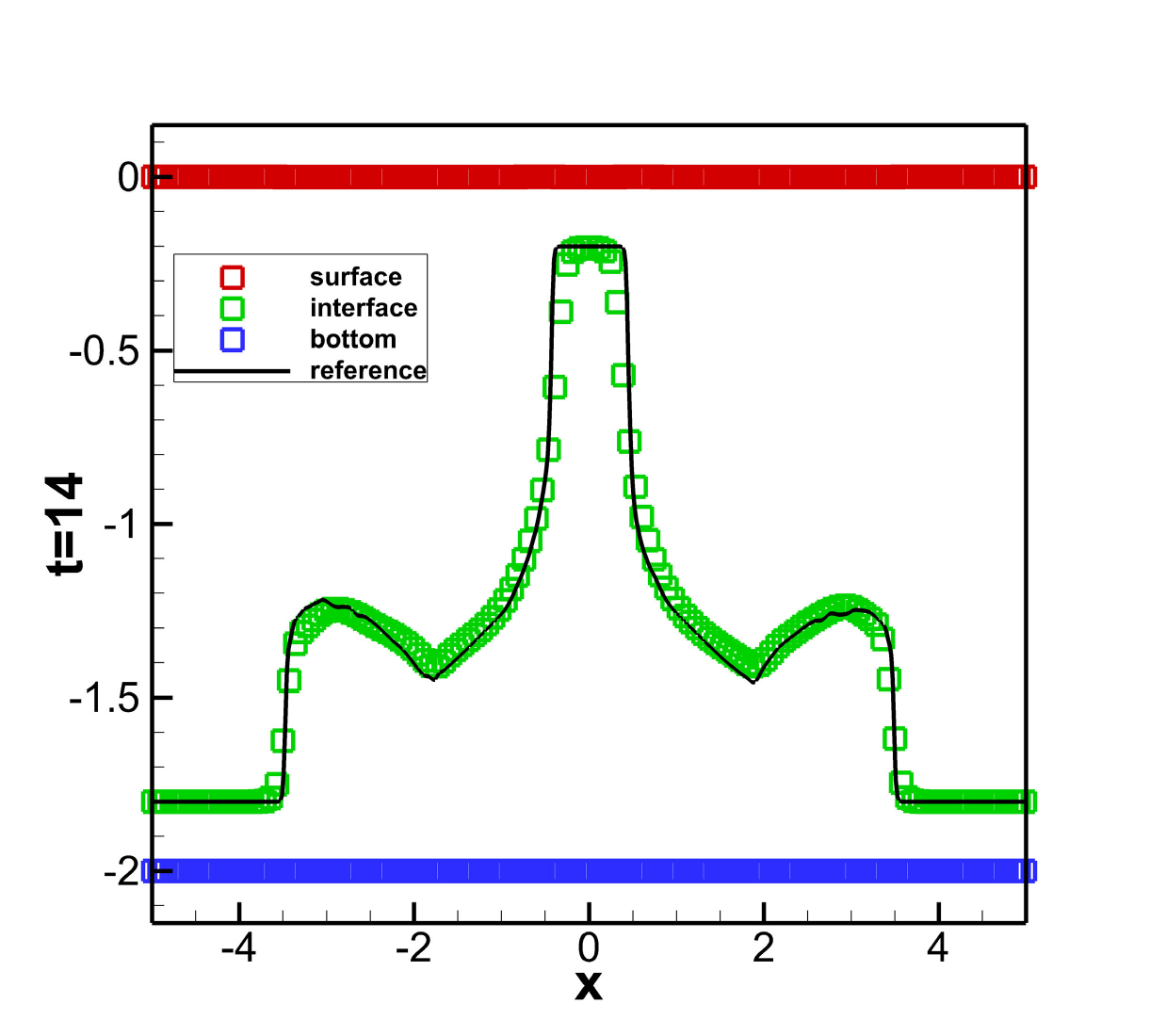} \\
 \includegraphics[scale=0.35]{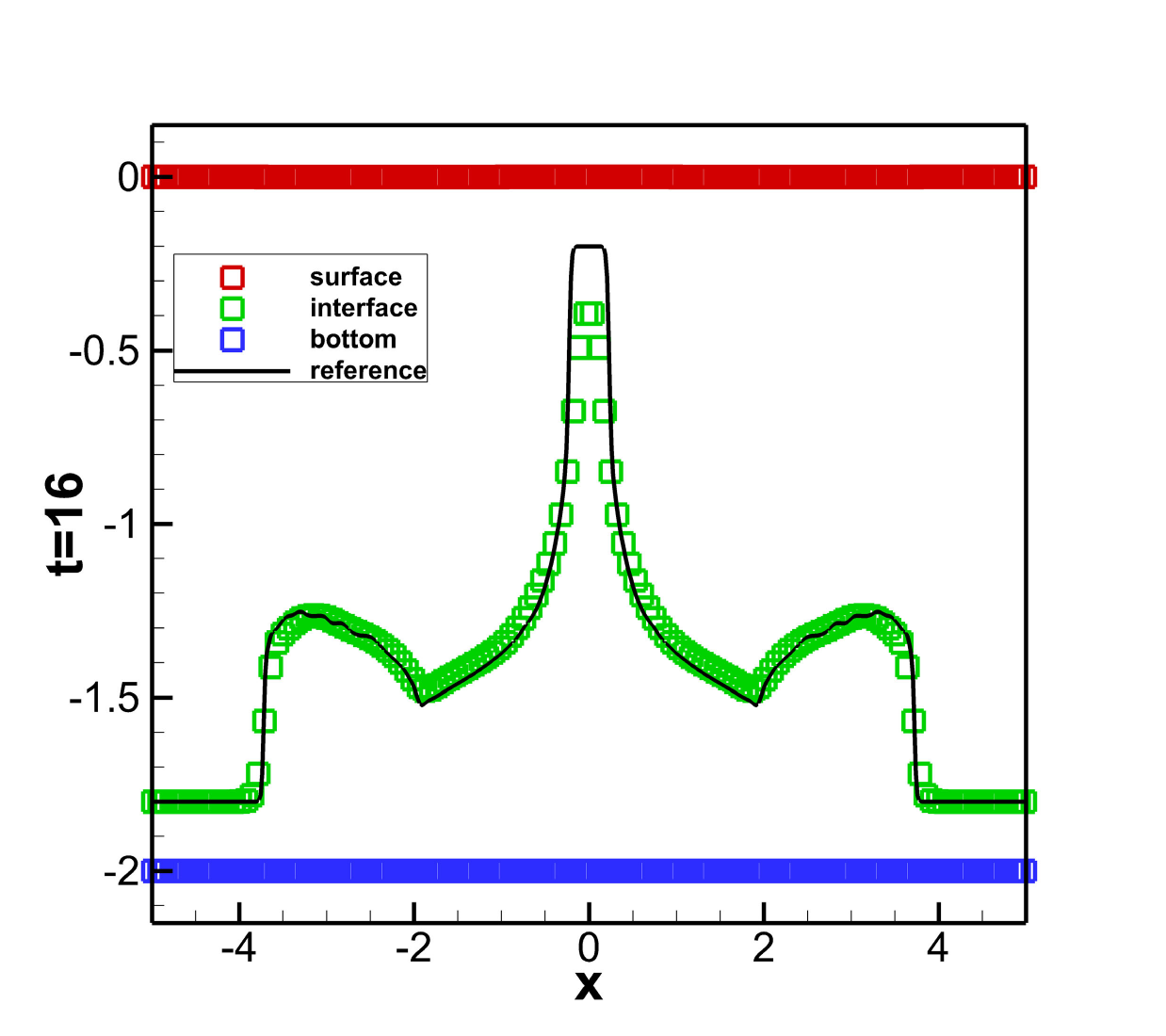}
  \qquad
\includegraphics[scale=0.35]{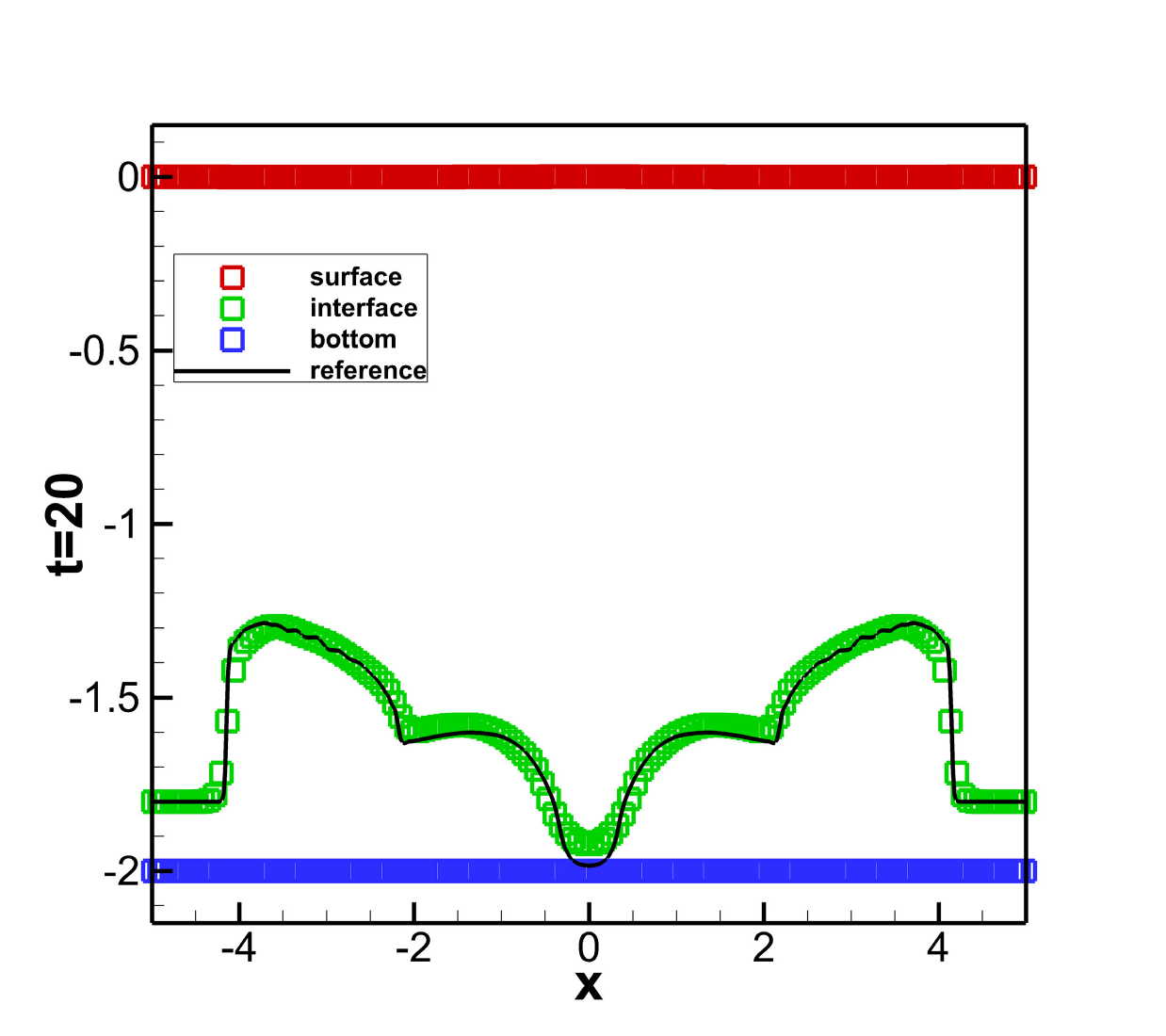}
    \caption{Section \ref{subsec:Internal_Circular_2d}: internal circular dam breaking in 2D with a flat bottom. Numerical solutions computed with $200\times 200$ grids. Diagonal slices of the water surface $h_1+h_2+Z$, interface $h_2+Z$ and bottom at times $t=4, 6, 10, 14, 16, 20$ (from top to bottom and from left to right). }   \label{Fig:internal_circular_bottom_interface}
\end{figure}
\subsubsection{Internal circular dam break over nonflat bottom topography}
\label{subsec:Internal_nonflat_2d}
In this test, we consider nonflat bottom $$Z(x, y)  = 0.5 e^{\left(x^2+y^2\right)}-2,$$
and the initial condition
$$
\left(h_1, q_1, p_1, h_2, q_2, p_2\right)(x, y, 0)= \begin{cases}(1.8,0,0,-1.8-Z,0,0), & \text { if } x^2+y^2>1, \\ (0.2,0,0,-0.2-Z,0,0), & \text { otherwise. }\end{cases}
$$
The constant gravitational acceleration is $g=9.81$ and the density ratio is $r=0.98$. The computational domain is $[-2,2] \times[-2,2]$ and a uniform mesh with $200\times 200$ grids is considered. The contour lines and diagonal slices of the water interface at $t=1, t=2$ are shown in Figure \ref{Fig:internal_2D_nonflat}.
 
\begin{figure}[H]
    \centering
 \includegraphics[scale=0.35]{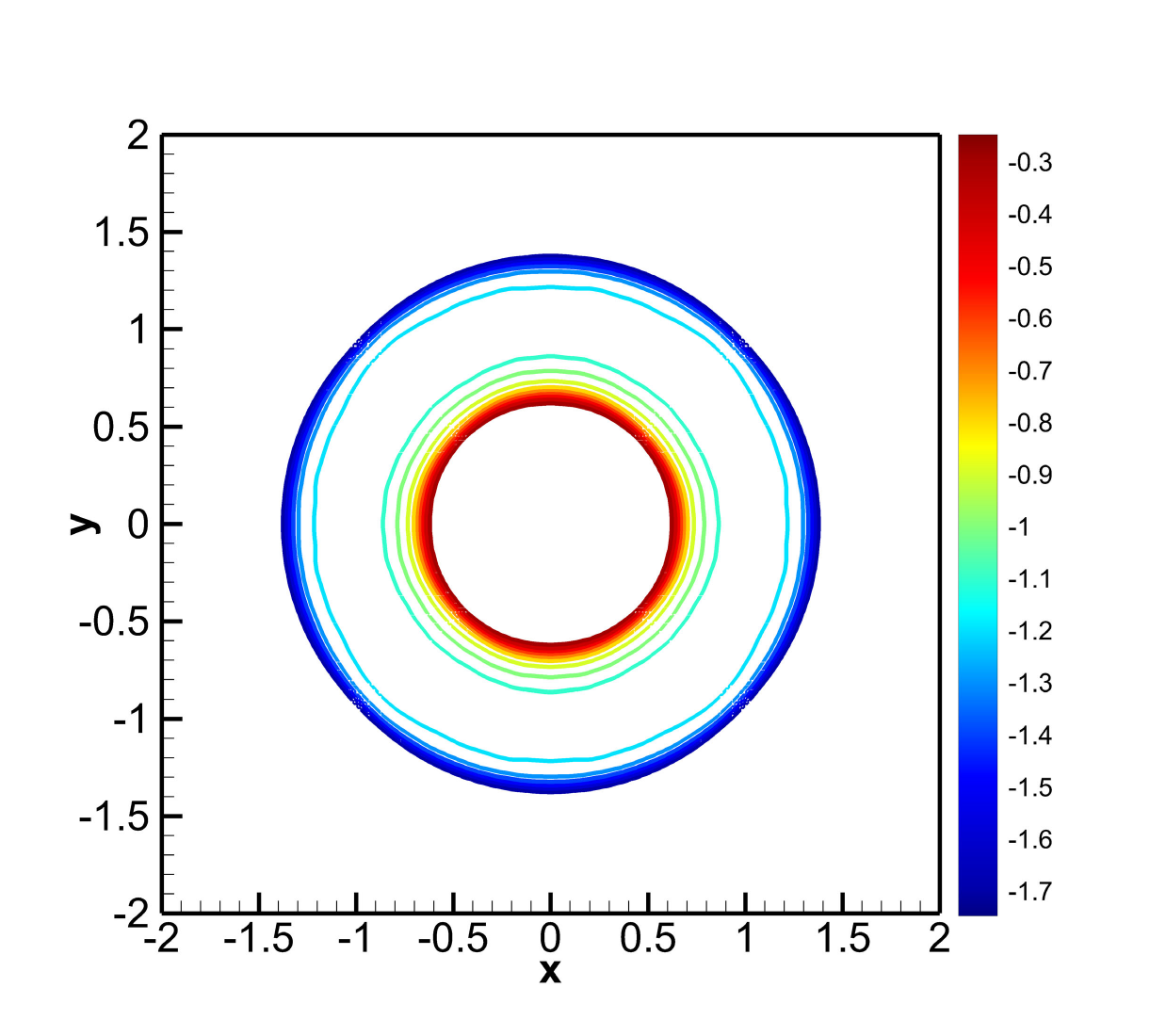} 
  \qquad
\includegraphics[scale=0.35]{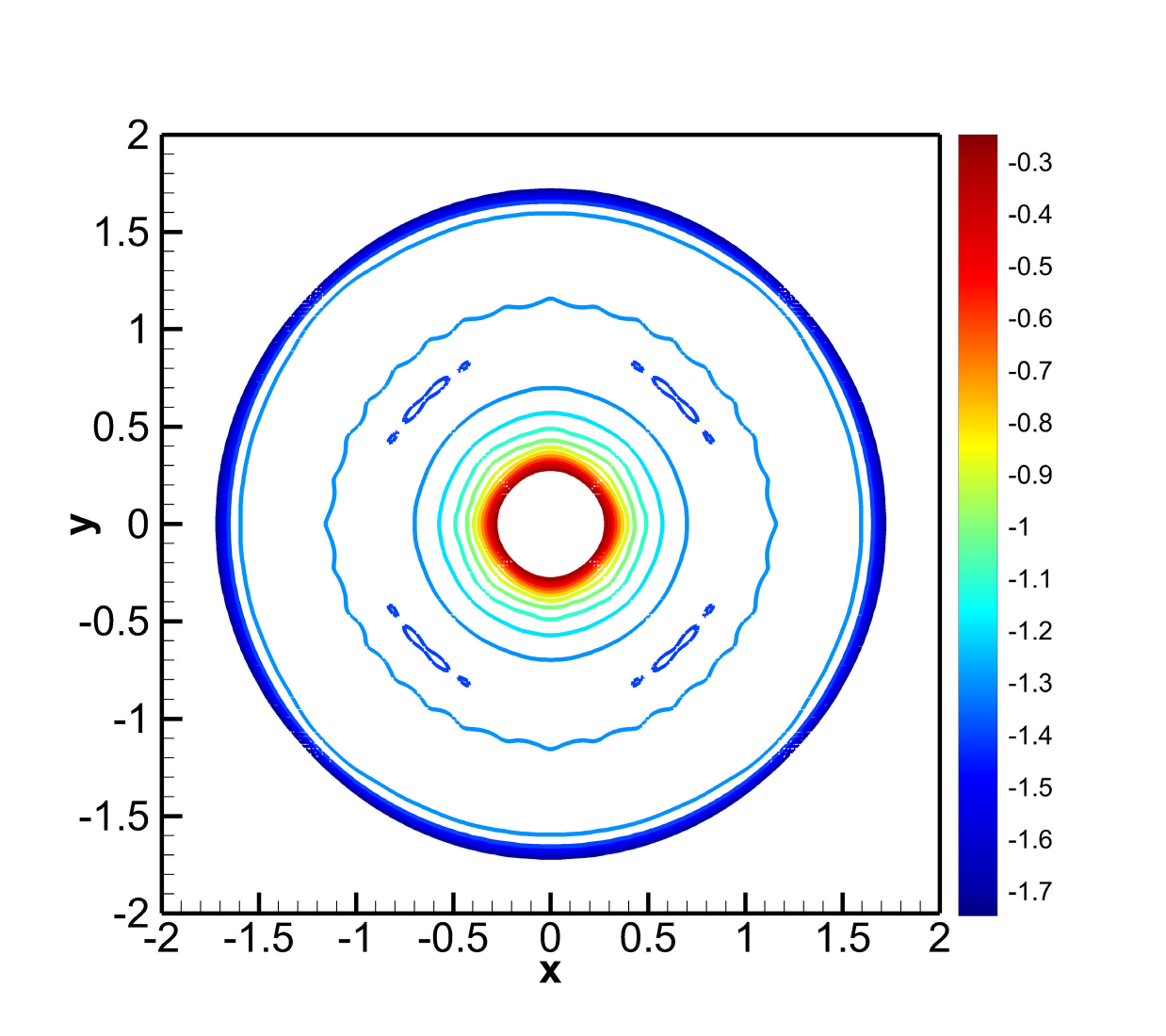}\\
\includegraphics[scale=0.35]{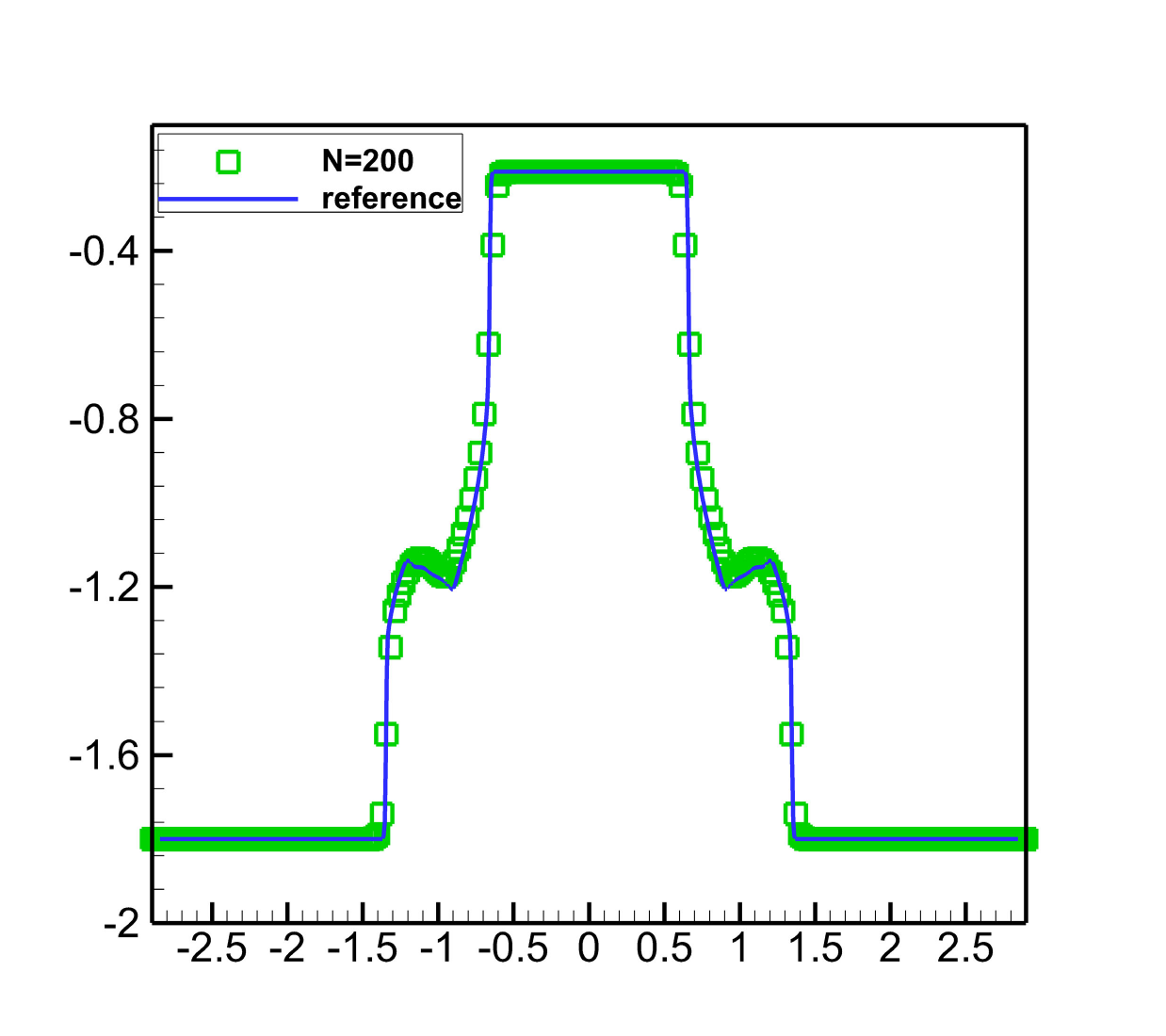}
  \qquad
\includegraphics[scale=0.35]{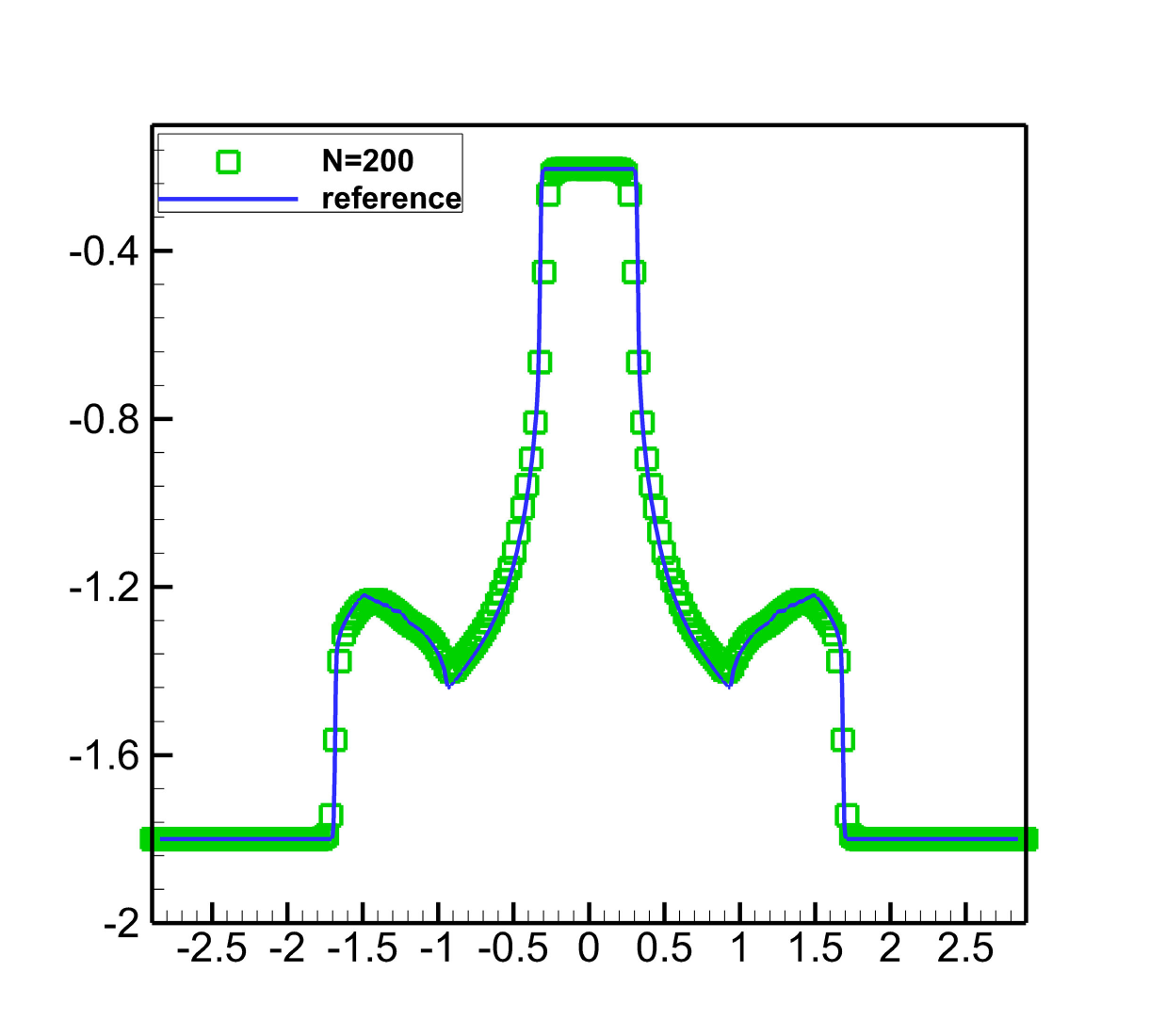}
    \caption{Section \ref{subsec:Internal_nonflat_2d}: internal circular dam break over nonflat bottom topography. Numerical solutions computed with $200\times 200$ grids.  Contour lines of the interface $h_2+Z$ at times $t=1$ (top left) and $t=2$ (top right); diagonal slices of the interface at time $t = 1$ (down left) and $t = 2$ (down right). }     \label{Fig:internal_2D_nonflat}
\end{figure}
\section{Conclusion}
\label{sec:conclusion}

A new family of high-order WENO finite-difference methods to solve hyperbolic nonconserative PDE systems have been proposed. These methods are based on a general strategy in which, instead of reconstructing fluxes using a WENO operator, what we reconstruct is the nonconservative products of the system which are computed using the selected family of paths. Moreover, if a Roe linearization is available, the nonconservative products can be computed through matrix-vector operations instead of path-integrals. The main advantages of this theory are the following:
\begin{itemize}
\item the high-order accuracy in space only depends on the order of the selected WENO operator, provided that the path satisfies the symmetry property;

\item no integrals in the cells have to be computed so that reconstructions with uniform accuracy in the entire cells are required;

\item a unified WENO framework that treats non-conservative equations consistently with conservative equations. The framework is compatible with state reconstruction techniques.

\end{itemize}
The methods have been extended to systems with source terms to design well-balanced methods. This methodology has been successfully applied to obtain high-order numerical methods for the 1D and 2D two-layer shallow water equations that preserve water-at-rest steady states. A number of numerical tests confirm the high-order accuracy of the methods as well as their shock-capturing and well-balanced properties.

The second strategy introduced here to design well-balanced methods can be combined with the technique developed in
\cite{math9151799}
to obtain numerical methods that preserve both water-at-rest and moving equilibria for the 1D shallow-water model or some particular families of stationary solutions in the 2D case: this will be done in an upcoming work.

Another further development concerns the convergence of the methods: as it happens for general finite-difference and related methods, the convergence of the numerical results to functions that are weak solutions of the system according to the selected family of paths is not ensured for the schemes introduced here. Some techniques recently developed in the context of high-order finite-volume methods in \cite{PIMENTELGARCIA2022111152, pimentelgarca2024highorder} can be adapted to the methods introduced here to ensure that isolated shocks that satisfy the generalized Rankine-Hugoniot associated to the selected family of paths are correctly captured.

\appendix
\section{Fifth-order WENOZ reconstruction}\label{app:A}

Let us recall here the expression of the fifth-order WENOZ reconstruction used to compute  ${F}_{i+1/2}$ as an example. The expression of ${F}_{i-1/2}$ can be obtained then using the mirror principle.
Given $F_j = F(U_j)$, $j = i-2,i-1,i,i+1,i+2$, ${F}_{i+1/2}$ is computed as follows:
\begin{equation}
    {F}_{i+1/2} =  \sum_{k = 0}^{2}\omega_{k} {F}_{i + 1/2}^{k},
 \end{equation}
 where
 \begin{equation*}
    F_{i + 1/2}^0=\frac{1}{6}(2F_{i - 2}-7F_{i - 1}+11F_{i}),
F_{i + 1/2}^1=\frac{1}{6}(-F_{i - 1}+5F_{i}+2F_{i + 1}),
F_{i + 1/2}^2=\frac{1}{6}(2F_{i}+5F_{i + 1}-F_{i + 2}),
\end{equation*}
are third-order interpolation formulas computed in 3 $substencils$, and 
$\omega_k$, $k=0,1,2$ are  nonlinear weights to be computed.
In the classical WENO-JS scheme \cite{JIANG1996202} the nonlinear weights 
are given by
\begin{equation*}
    \omega_{k}^{(JS)}=\frac{\alpha_{k}}{\sum_{j = 0}^{2}\alpha_{j}}, \quad \alpha_{k}=\frac{d_{k}}{(\beta_{k}+\varepsilon)^{p}}, \quad k = 0,1,2.
\end{equation*}
Here $d_{0}=1/10$, $d_{1}=6/10$, ${d}_{2}=3/10$, are  the ideal weights   leading to the global fifth-order interpolation formula
\begin{equation*}
   F_{i + 1/2} =\frac{1}{60}(2F_{i- 2}-13F_{i - 1}+47F_{i}+27F_{i + 1}-3F_{i + 2});
\end{equation*}
 $\beta_k$ are the smoothness indicators
\begin{equation*}
    \beta_{k}=\sum_{l = 1}^{2}\Delta x^{2l - 1}\int_{x_{i-\frac{1}{2}}}^{x_{i+\frac{1}{2}}}\left(\frac{d^{l}}{dx^{l}}F^{k}_{i+\frac{1}{2}}(x)\right)^{2}dx, \quad k = 0,1,2,
\end{equation*}
that are defined so that the weights are close to the ideal ones in smooth regions but, when a discontinuity is detected in the stencil,  the contribution of the sub-stencil containing it is close to 0 (non-oscillatory weights).
 WENO-Z scheme \cite{CASTRO20111766} propose the global smooth indicator $\tau_{5}=|\beta_{2}-\beta_{0}|$ to achieve the optimal accuracy at the critical points,
 \begin{equation*}
 \omega_{k}^{(Z)}=\frac{\alpha_{k}}{\sum_{l = 0}^{2}\alpha_{l}}, \quad
\alpha_{k}=d_{k}\left(1+\left(\frac{\tau_{5}}{\beta_{k}+\varepsilon}\right)^{p}\right), \quad  k = 0,1,2,
 \end{equation*}
 $ \varepsilon = 10^{-12}, p = 2$ are used in this study.
There is a vast literature related to the computation of optimal order smoothness indicators: see for instance \cite{BALSARA2016780, LEVY2000407}.

\section*{Acknowledgements}
The work of B. Ren is supported by the China Scholarship Council. The work of C. Parés is supported by Spanish projects PDC2022-133663-C21 and PID2022-137637NB-C21 funded by MCIN/AEI/10.13039/501100011033 and FSE+.

\noindent \section*{In memoriam}
 
\noindent This paper is dedicated to the memory of Prof. Arturo Hidalgo L\'opez
($^*$July 03\textsuperscript{rd} 1966 - $\dagger$August 26\textsuperscript{th} 2024) of the Universidad Politecnica de Madrid. 

\bibliographystyle{siam}
\bibliography{Reference}

\begin{thebibliography}{10}

\bibitem{Relaxation:10.1137/06067167X}
{\sc R.~Abgrall and S.~Karni}, {\em Two-layer shallow water system: A relaxation approach}, SIAM Journal on Scientific Computing, 31 (2009), pp.~1603--1627.

\bibitem{S1064827503431090}
{\sc E.~Audusse, F.~Bouchut, M.-O. Bristeau, R.~Klein, and B.~Perthame}, {\em A fast and stable well-balanced scheme with hydrostatic reconstruction for shallow water flows}, SIAM Journal on Scientific Computing, 25 (2004), pp.~2050--2065.

\bibitem{Balsara2024}
{\sc D.~S. Balsara, D.~Bhoriya, C.-W. Shu, and H.~Kumar}, {\em Efficient finite difference {WENO} scheme for hyperbolic systems with non-conservative products}, Communications on Applied Mathematics and Computation, 6 (2024), pp.~907--962.

\bibitem{BALSARA2016780}
{\sc D.~S. Balsara, S.~Garain, and C.-W. Shu}, {\em An efficient class of {WENO} schemes with adaptive order}, Journal of Computational Physics, 326 (2016), pp.~780--804.

\bibitem{BeljadidCiCP-21-913}
{\sc A.~Beljadid, P.~LeFloch, Siddhartha, and C.~Parés}, {\em Schemes with well-controlled dissipation. {H}yperbolic systems in nonconservative form}, Communications in Computational Physics, 21 (2018), pp.~913--946.

\bibitem{BERMUDEZ19941049}
{\sc A.~Bermúdez and M.~E. Vázquez-Cendón}, {\em Upwind methods for hyperbolic conservation laws with source terms}, Computers \& Fluids, 23 (1994), pp.~1049--1071.

\bibitem{Berthon}
{\sc C.~Berthon}, {\em Nonlinear scheme for approximating a non-conservative hyperbolic system}, Comptes Rendus Mathématiques de l'Acad\'emi des Sciences, 395 (2002), pp.~1069--1072.

\bibitem{BORGES20083191}
{\sc R.~Borges, M.~Carmona, B.~Costa, and W.~S. Don}, {\em An improved weighted essentially non-oscillatory scheme for hyperbolic conservation laws}, Journal of Computational Physics, 227 (2008), pp.~3191 -- 3211.

\bibitem{François_Bouchut_2010}
{\sc F.~Bouchut and V.~Zeitlin}, {\em A robust well-balanced scheme for multi-layer shallow water equations}, Discrete and Continuous Dynamical Systems - B, 13 (2010), pp.~739--758.

\bibitem{CAO2023111790}
{\sc Y.~Cao, A.~Kurganov, Y.~Liu, and V.~Zeitlin}, {\em Flux globalization based well-balanced path-conservative central-upwind scheme for two-layer thermal rotating shallow water equations}, Journal of Computational Physics, 474 (2023), p.~111790.

\bibitem{CASTRO20111766}
{\sc M.~Castro, B.~Costa, and W.~S. Don}, {\em High order weighted essentially non-oscillatory {WENO-Z} schemes for hyperbolic conservation laws}, Journal of Computational Physics, 230 (2011), pp.~1766--1792.

\bibitem{Diaz_Cheng_Chertock_Kurganov_2014}
{\sc M.~J. Castro, Y.~Cheng, A.~Chertock, and A.~Kurganov}, {\em Solving two-mode shallow water equations using finite volume methods}, Communications in Computational Physics, 16 (2014), p.~1323–1354.

\bibitem{Castro2009}
{\sc M.~J. Castro, E.~D. Fernández-Nieto, A.~M. Ferreiro, J.~A. García-Rodríguez, and C.~Parés}, {\em High order extensions of {R}oe schemes for {T}wo-dimensional nonconservative hyperbolic systems}, Journal of Scientific Computing, 39 (2009), pp.~67--114.

\bibitem{Castro-Diaz2011}
{\sc M.~J. Castro, E.~D. Fernández-Nieto, J.~M. González-Vida, and C.~Parés}, {\em Numerical treatment of the loss of hyperbolicity of the two-layer shallow-water system}, Journal of Scientific Computing, 48 (2011), p.~16–40.

\bibitem{CFMP2013}
{\sc M.~J. Castro, U.~S. Fjordholm, S.~Mishra, and C.~Par\'{e}s}, {\em Entropy conservative and entropy stable schemes for nonconservative hyperbolic systems}, SIAM Journal on Numerical Analysis, 51 (2013), pp.~1371--1391.

\bibitem{Castro_2006}
{\sc M.~J. Castro, J.~M. Gallardo, and C.~Parés}, {\em High order finite volume schemes based on reconstruction of states for solving hyperbolic systems with nonconservative products. {A}pplications to shallow-water systems}, Mathematics of Computation, 75 (2006), p.~1103–1135.

\bibitem{CASTRO2004202}
{\sc M.~J. Castro, J.~A. García-Rodríguez, J.~M. González-Vida, J.~Macías, C.~Parés, and M.~Vázquez-Cendón}, {\em Numerical simulation of two-layer shallow water flows through channels with irregular geometry}, Journal of Computational Physics, 195 (2004), pp.~202--235.

\bibitem{Castro_2008}
{\sc M.~J. Castro, P.~G. LeFloch, M.~L. Muñoz-Ruiz, and C.~Parés}, {\em Why many theories of shock waves are necessary: Convergence error in formally path-consistent schemes}, Journal of Computational Physics, 227 (2008), p.~8107–8129.

\bibitem{Castro2001107}
{\sc M.~J. Castro, J.~Macías, and C.~Parés}, {\em A {Q}-scheme for a class of systems of coupled conservation laws with source term. {A}pplication to a two-layer 1-{D} shallow water system}, Mathematical Modelling and Numerical Analysis, 35 (2001), p.~107 – 127.

\bibitem{CASTRO2017131}
{\sc M.~J. Castro, T.~Morales~de Luna, and C.~Parés}, {\em Chapter 6 - {W}ell-balanced schemes and path-conservative numerical methods}, in Handbook of Numerical Methods for Hyperbolic Problems, R.~Abgrall and C.-W. Shu, eds., vol.~18 of Handbook of Numerical Analysis, Elsevier, 2017, pp.~131--175.

\bibitem{Castro2010Mc}
{\sc M.~J. Castro, A.~Pardo, C.~Par\'{e}s, and E.~F. Toro}, {\em On some fast well-balanced first order solvers for nonconservative systems}, Mathematics of Computation, 79 (2010), pp.~1427--1472.

\bibitem{Castro20072055}
{\sc M.~J. Castro, A.~Pardo, and C.~Parés}, {\em Well-balanced numerical schemes based on a generalized hydrostatic reconstruction technique}, Mathematical Models and Methods in Applied Sciences, 17 (2007), p.~2055 – 2113.

\bibitem{doi:10.1137/110828873}
{\sc M.~J. Castro, C.~Par\'{e}s, G.~Puppo, and G.~Russo}, {\em Central schemes for nonconservative hyperbolic systems}, SIAM Journal on Scientific Computing, 34 (2012), pp.~B523--B558.

\bibitem{Castro2020jsc}
{\sc M.~J. Castro and C.~Parés}, {\em Well-balanced high-order finite volume methods for systems of balance laws}, Journal of Scientific Computing, 82(2), 48 (2020).

\bibitem{CHALMERS2013111}
{\sc N.~Chalmers and E.~Lorin}, {\em On the numerical approximation of one-dimensional nonconservative hyperbolic systems}, Journal of Computational Science, 4 (2013), pp.~111--124.
\newblock Computational Methods for Hyperbolic Problems.

\bibitem{Chu2022}
{\sc S.~Chu, A.~Kurganov, and M.~Na}, {\em Fifth-order {A-WENO} schemes based on the path-conservative central-upwind method}, Journal of Computational Physics, 469 (2022), p.~111508.

\bibitem{DONG2024106193}
{\sc J.~Dong and X.~Qian}, {\em A robust numerical scheme based on auxiliary interface variables and monotone-preserving reconstructions for two-layer shallow water equations with wet–dry fronts}, Computers \& Fluids, 272 (2024), p.~106193.

\bibitem{Fernandez-Nieto2011}
{\sc E.~D. Fernández-Nieto, M.~J. Castro, and C.~Parés}, {\em On an intermediate field capturing {R}iemann solver based on a parabolic viscosity matrix for the two-layer shallow water system}, Journal of Scientific Computing, 48 (2011), pp.~117--140.

\bibitem{doi:10.1137/0733001}
{\sc J.~M. Greenberg and A.~Y. Leroux}, {\em A well-balanced scheme for the numerical processing of source terms in hyperbolic equations}, SIAM Journal on Numerical Analysis, 33 (1996), pp.~1--16.

\bibitem{math9151799}
{\sc I.~Gómez-Bueno, M.~J. Castro, C.~Parés, and G.~Russo}, {\em Collocation methods for high-order well-balanced methods for systems of balance laws}, Mathematics, 9 (2021).

\bibitem{HARTEN1983235}
{\sc A.~Harten and J.~M. Hyman}, {\em Self adjusting grid methods for one-dimensional hyperbolic conservation laws}, Journal of Computational Physics, 50 (1983), pp.~235--269.

\bibitem{HENRICK2005542}
{\sc A.~K. Henrick, T.~D. Aslam, and J.~M. Powers}, {\em Mapped weighted essentially non-oscillatory schemes: Achieving optimal order near critical points}, Journal of Computational Physics, 207 (2005), pp.~542--567.

\bibitem{JIANG1996202}
{\sc G.-S. Jiang and C.-W. Shu}, {\em Efficient implementation of weighted {ENO} schemes}, Journal of Computational Physics, 126 (1996), pp.~202 -- 228.

\bibitem{KRVAVICA2018187}
{\sc N.~Krvavica, M.~Tuhtan, and G.~Jelenić}, {\em Analytical implementation of {R}oe solver for two-layer shallow water equations with accurate treatment for loss of hyperbolicity}, Advances in Water Resources, 122 (2018), pp.~187--205.

\bibitem{Kurganov2009}
{\sc A.~Kurganov and G.~Petrova}, {\em Central-upwind schemes for two-layer shallow water equations}, SIAM Journal on Scientific Computing, 31 (2009), pp.~1742--1773.

\bibitem{LEVEQUE1998346}
{\sc R.~J. LeVeque}, {\em Balancing source terms and flux gradients in high-resolution {G}odunov methods: The quasi-steady wave-propagation algorithm}, Journal of Computational Physics, 146 (1998), pp.~346--365.

\bibitem{LEVY2000407}
{\sc D.~Levy, G.~Puppo, and G.~Russo}, {\em On the behavior of the total variation in {CWENO} methods for conservation laws}, Applied Numerical Mathematics, 33 (2000), pp.~407--414.

\bibitem{Liu2021}
{\sc X.~Liu}, {\em A new well-balanced finite-volume scheme on unstructured triangular grids for two-dimensional two-layer shallow water flows with wet-dry fronts}, Journal of Computational Physics, 438 (2021), p.~110380.

\bibitem{Mandli_2013}
{\sc K.~T. Mandli}, {\em A numerical method for the two layer shallow water equations with dry states}, Ocean Modelling, 72 (2013), p.~80–91.

\bibitem{Maso1995DefinitionAW}
{\sc G.~D. Maso, P.~L. Floch, and F.~Murat}, {\em Definition and weak stability of nonconservative products}, Journal de Math{\'e}matiques Pures et Appliqu{\'e}es, 74 (1995), pp.~483--548.

\bibitem{10.2166/hydro.2020.207}
{\sc J.~Murillo, S.~Martinez-Aranda, A.~Navas-Montilla, and P.~García-Navarro}, {\em {Adaptation of flux-based solvers to 2D two-layer shallow flows with variable density including numerical treatment of the loss of hyperbolicity and drying/wetting fronts}}, Journal of Hydroinformatics, 22 (2020), pp.~972--1014.

\bibitem{doi:10.1137/050628052}
{\sc C.~Par\'{e}s}, {\em Numerical methods for nonconservative hyperbolic systems: a theoretical framework}, SIAM Journal on Numerical Analysis, 44 (2006), pp.~300--321.

\bibitem{M2AN_2004}
{\sc C.~Par\'es and M.~J. Castro}, {\em On the well-balance property of {Roe's} method for nonconservative hyperbolic systems. {Applications} to shallow-water systems}, ESAIM: Mod\'elisation math\'ematique et analyse num\'erique, 38 (2004), pp.~821--852.

\bibitem{PARES2021109880}
{\sc C.~Parés and C.~Parés-Pulido}, {\em Well-balanced high-order finite difference methods for systems of balance laws}, Journal of Computational Physics, 425 (2021), p.~109880.

\bibitem{PIMENTELGARCIA2022111152}
{\sc E.~Pimentel-García, M.~J. Castro, C.~Chalons, T.~{Morales de Luna}, and C.~Parés}, {\em In-cell discontinuous reconstruction path-conservative methods for non conservative hyperbolic systems - {S}econd-order extension}, Journal of Computational Physics, 459 (2022), p.~111152.

\bibitem{pimentelgarca2024highorder}
{\sc E.~Pimentel-García, M.~J. Castro, C.~Chalons, and C.~Parés}, {\em High-order in-cell discontinuous reconstruction path-conservative methods for nonconservative hyperbolic systems -- {DR.MOOD} generalization}, Numerical Methods for Partial Differential Equations, 40 (2024), p.~e23133.

\bibitem{ROE1997250}
{\sc P.~Roe}, {\em Approximate {R}iemann {S}olvers, parameter vectors, and difference schemes}, Journal of Computational Physics, 135 (1997), pp.~250--258.

\bibitem{Schijf1953TheoreticalCO}
{\sc J.~Schijf and J.~C. Sch{\"o}nfled}, {\em Theoretical considerations on the motion of salt and fresh water}, Proceedings Minnesota International Hydraulic Convention,  (1953), p.~321–333.

\bibitem{Shu_1998}
{\sc C.-W. Shu}, {\em Essentially non-oscillatory and weighted essentially non-oscillatory schemes for hyperbolic conservation laws}, Springer Berlin Heidelberg, 1998, p.~325–432.

\bibitem{SHU1988439}
{\sc C.-W. Shu and S.~Osher}, {\em Efficient implementation of essentially non-oscillatory shock-capturing schemes}, Journal of Computational Physics, 77 (1988), pp.~439 -- 471.

\bibitem{TOUMI1992360}
{\sc I.~Toumi}, {\em A weak formulation of {R}oe's approximate {R}iemann solver}, Journal of Computational Physics, 102 (1992), pp.~360--373.

\bibitem{xing2005high}
{\sc Y.~Xing and C.-W. Shu}, {\em High order finite difference {WENO} schemes with the exact conservation property for the shallow water equations}, Journal of Computational Physics, 208 (2005), pp.~206--227.

\bibitem{Zhu2017}
{\sc J.~Zhu and J.~Qiu}, {\em A new type of finite volume {WENO} schemes for hyperbolic conservation laws}, Journal of Scientific Computing, 73 (2017), p.~1338–1359.

\end{thebibliography}
\end{document}